\documentclass[11pt]{amsart}
\usepackage{latexsym}
\usepackage{rotating, enumerate}
\usepackage{amsfonts}
\usepackage{amssymb}
\usepackage{amsmath}
\usepackage{graphics, color}
\definecolor{darkblue}{rgb}{0,0,0.4}
\usepackage[colorlinks=true, citecolor=darkblue, filecolor=darkblue,
linkcolor=darkblue, urlcolor=darkblue]{hyperref}
\input xy
\xyoption{all}
\DeclareMathOperator*{\mocolim}{colim}
\DeclareMathOperator*{\colim}{colim}

\DeclareMathOperator*{\coprodmo}{\coprod}

\newtheorem{theorem}{Theorem}[section]
\newtheorem{lemma}[theorem]{Lemma}

\newtheorem{proposition}[theorem]{Proposition}

\newtheorem{definition}[theorem]{Definition}
\newtheorem{remark}[theorem]{Remark}

\newtheorem{construction}[theorem]{Construction}

\begin{document}

\newcommand{\e}{\emph}
\newcommand{\fD}{\mathfrak{D}}
\newcommand{\fB}{\mathfrak{B}}

\newcommand{\bbA}{\mathbb{A}}
\newcommand{\bbC}{\mathbb{C}}
\newcommand{\N}{\mathbb{N}}
\newcommand{\bbN}{\mathbb{N}}
\newcommand{\R}{\mathbb{R}}
\newcommand{\bbR}{\mathbb{R}}
\newcommand{\bbL}{\mathbb{L}}
\newcommand{\bbZ}{\mathbb{Z}}
\newcommand{\bbW}{\mathbb{W}}
\newcommand{\bbK}{\mathbb{K}}

\newcommand{\A}{\mathcal{A}}
\newcommand{\B}{\mathcal{B}}
\newcommand{\C}{\mathcal{C}}
\newcommand{\mcC}{\mathcal{C}}
\newcommand{\D}{\mathcal{D}}
\newcommand{\mcD}{\mathcal{D}}
\newcommand{\E}{\mathcal{E}}
\newcommand{\mF}{\mathcal{F}}
\newcommand{\G}{\mathcal{G}}
\newcommand{\mcG}{\mathcal{G}}
\newcommand{\mH}{\mathcal{H}}
\newcommand{\mcH}{\mathcal{H}} 
\newcommand{\I}{\mathcal{I}}
\newcommand{\mL}{\mathcal{L}}
\newcommand{\mcN}{\mathcal{N}}
\newcommand{\mO}{\mathcal{O}}
\newcommand{\mcP}{\mathcal{P}}
\newcommand{\Q}{\mathcal{Q}}
\newcommand{\mcR}{\mathcal{R}}
\newcommand{\V}{\mathcal{V}}
\newcommand{\mS}{\mathcal{S}}
\newcommand{\U}{\mathcal{U}}
\newcommand{\W}{\mathcal{W}}
\newcommand{\X}{\mathcal{X}}
\newcommand{\mX}{\mathcal{X}}
\newcommand{\Y}{\mathcal{Y}}
\newcommand{\Z}{\mathcal{Z}}

\newcommand{\bs}{\mathbf{s}}
\newcommand{\bn}{\mathbf{n}}
\newcommand{\bm}{\mathbf{m}}
\newcommand{\bq}{\mathbf{q}}
\newcommand{\bbf}{\mathbf{f}}
\newcommand{\bbg}{\mathbf{g}}
\newcommand{\bt}{\mathbf{t}}
\newcommand{\bv}{\mathbf{v}}
\newcommand{\bN}{\mathbf{N}}
\newcommand{\bT}{\mathbf{T}}
\newcommand{\bR}{\mathbf{R}}

\newcommand{\sd}{d_{\Delta}}
\newcommand{\la}{\langle}
\newcommand{\ra}{\rangle}
\newcommand{\rdisj}{\stackrel{r}{\coprod}}
\newcommand{\xmaps}{\xrightarrow}

\newcommand{\hofib}{\textrm{hofib}}

\newcommand{\contains}{\supseteq}
\newcommand{\Img}{\textrm{Im}}
\newcommand{\ts}{\textsuperscript}
\newcommand{\bC}{\mathbf{C}}
\newcommand{\supp}{\textrm{supp}}
\newcommand{\Sperp}{\mS^{\perp}}
\newcommand{\Aperp}{A^{\perp}}
\newcommand{\goesto}{\mapsto}
\newcommand{\bG}{\mathcal{G}_0} 
\newcommand{\sth}[1]{#1^{\mathrm{th}}}
\newcommand{\abs}[1]{\left| #1\right|}
\newcommand{\ord}[1]{\Delta \left( #1 \right)}
\newcommand{\leqs}{\leqslant}
\newcommand{\geqs}{\geqslant}
\newcommand{\heq}{\simeq}
\newcommand{\iso}{\simeq}
\newcommand{\maps}{\longrightarrow}
\newcommand{\injects}{\hookrightarrow}
\newcommand{\homeo}{\cong}
\newcommand{\surjects}{\twoheadrightarrow}
\newcommand{\isom}{\cong}
\newcommand{\cross}{\times}
\newcommand{\srm}[1]{\stackrel{#1}{\maps}}
\newcommand{\srt}[1]{\stackrel{#1}{\to}}
\newcommand{\normal}{\vartriangleleft}
\newcommand{\wt}[1]{\widetilde{#1}} 
\newcommand{\fc}{\mathcal{A}_{\mathrm{flat}}} 
\newcommand{\Gr}{\textrm{Gr}}
\newcommand{\Seq}{\textrm{\underline{\bf Seq}}}
\newcommand{\Ob}{\textrm{Ob}}
\newcommand{\Id}{\textrm{Id}}
\newcommand{\Ai}{\mathcal{A}_1\cap\mathcal{A}_2}
\newcommand{\Ab}{\A_b}
\newcommand{\Ac}{\mathcal{A}_c}
\newcommand{\Ainf}{\mathcal{A}_\infty}
\newcommand{\bAc}{\overline{\mathcal{A}_c}}
\newcommand{\Acp}{\Ac^+}
\newcommand{\bAcp}{\bAc^+}
\newcommand{\bAcpp}{\bAc^{++}}
\newcommand{\vect}[1]{\stackrel{\rightharpoonup}{\mathbf #1}}
\newcommand{\SR}{\mathcal{SR}}
\newcommand{\SRe}{\mathcal{SR}^{\mathrm{even}}}
\newcommand{\Rep}{\mathrm{Rep}}
\newcommand{\SRep}{\mathrm{SRep}}
\newcommand{\Hom}{\mathrm{Hom}}
\newcommand{\Lie}{\mathrm{Lie}}
\newcommand{\diam}{\mathrm{diam}}
\newcommand{\K}{K}
\newcommand{\mK}{\mathcal{K}_{\mathrm{def}}}
\newcommand{\SK}{SK_{\mathrm{def}}}
\newcommand{\dom}{\mathrm{dom}}
\newcommand{\codom}{\mathrm{codom}}
\newcommand{\Mor}{\mathrm{Mor}}
\newcommand{\+}[1]{\underline{#1}_+}
\newcommand{\Fin}{\Gamma^{\mathrm{op}}}
\newcommand{\f}[1]{\underline{#1}}
\newcommand{\Stab}{\mathrm{Stab}}
\newcommand{\Css}{\mathcal{C}_{ss}}
\newcommand{\Map}{\mathrm{Map}}
\newcommand{\flatc}{\mathcal{A}_{\mathrm{flat}}}
\newcommand{\F}[1]{\mathrm{Flag}(\vect{#1})}
\newcommand{\p}{\vect{p}}
\newcommand{\avg}{\mathrm{avg}}
\newcommand{\smsh}[1]{\ensuremath{\mathop{\wedge}_{#1}}}
\newcommand{\Vect}{\mathrm{Vect}}
\newcommand{\convto}{\Longrightarrow}
\newcommand{\sm}{\wedge}

\def\co{\colon\thinspace}

\def\endrem{}
\def\colon{{:}\;}

\title[Finite decomposition complexity and algebraic $K$--theory]{Finite decomposition complexity and the integral Novikov conjecture for higher algebraic $K$--theory}

\author[Ramras]{Daniel A. Ramras}
\address{New Mexico State University\\
Department of Mathematical Sciences\\
P.O. Box 30001\\
Department 3MB\\
Las Cruces, New Mexico 88003-8001 U.S.A.}
\email{ramras@nmsu.edu}

\author[Tessera]{Romain Tessera}
\address{UMPA, ENS de Lyon\\
 46 all\'{e}e d'Italie\\
  69364 Lyon Cedex 07\\
   France}
\email{tessera@phare.normalesup.org}

\author[Yu]{Guoliang Yu}
\address{Department of Mathematics\\
1326 Stevenson Center\\
\newline Vanderbilt University\\
Nashville, TN 37240 U.S.A.}
\email{guoliang.yu@vanderbilt.edu}


\thanks{The first author was partially supported by NSF grants DMS-0804553/0968766.\\
The second author was partially supported by NSF grant  DMS-0706486 and ANR grants AGORA and BLANC.\\
The third author was partially supported by NSF grants DMS-0600216 and DMS-1101195}


 \begin{abstract}  Decomposition complexity for metric spaces was recently introduced by Guentner, Tessera, and Yu as a natural generalization of asymptotic dimension.  We prove a vanishing result for the continuously controlled algebraic $K$--theory of bounded geometry metric spaces with finite decomposition complexity.  This leads to a proof of the integral $K$--theoretic Novikov conjecture, regarding split injectivity of the $K$--theoretic assembly map, for groups with finite decomposition complexity and finite CW models for their classifying spaces.  By work of Guentner, Tessera, and Yu, this  includes all (geometrically finite) linear groups.  
\end{abstract}

\maketitle

\section{Introduction}$\label{intro}$

Decomposition complexity for metric spaces, introduced by Guentner, Tessera, and Yu~\cite{GTY-FDC, GTY-rigid}, is a natural inductive generalization of the much-studied notion of asymptotic dimension.  Roughly speaking, decomposition complexity measures the difficulty of decomposing a metric space into uniformly bounded pieces that are well-separated from one another.
The class of metric spaces with \e{finite} decomposition complexity (FDC), as defined in Definition~\ref{fdc},  contains all metric spaces with finite asymptotic dimension \cite[Theorem 4.1]{GTY-FDC}, as well as all countable linear groups equipped with a proper (left-)invariant metric (\cite[Theorem 3.0.1]{GTY-rigid} and~\cite[Theorem 5.2.2]{GTY-FDC}).  
In this article, we study the integral Novikov conjecture for the algebraic $K$--theory of group rings $R[\Gamma]$, where $\Gamma$ has FDC.

For a discrete group $\Gamma$, the classical Novikov conjecture on the homotopy invariance of higher signatures is implied by rational injectivity of the Baum--Connes assembly map~\cite{Baum-Connes}.  In Yu~\cite{Yu-BC} and Skandalis--Tu--Yu~\cite{STY-coarse-BC}, injectivity of the Baum-Connes map was proved for groups coarsely embeddable into Hilbert space.
Using this result, Guentner, Higson, and Weinberger~\cite{G-H-W} proved the Novikov conjecture for linear groups.  This inspired the work of Guentner, Tessera, and Yu~\cite{GTY-rigid}, who proved the \e{integral} Novikov conjecture (establishing integral injectivity of the $L$--theoretic assembly map) for geometrically finite FDC groups (i.e. those with a finite CW model for their classifying space), and hence the stable Borel Conjecture for closed aspherical manifolds whose fundamental groups have FDC.

The algebraic $K$--theory Novikov conjecture claims that Loday's assembly map~\cite{Loday} 
\begin{equation}\label{assembly-map-intro}H_* (B\Gamma; \bbK (R)) \maps K_* (R[\Gamma])\end{equation}
is (rationally) injective.
Here $\Gamma$ is a finitely generated group and $R$ is an associative, unital ring (not necessarily commutative).  The domain of the assembly map is the homology of $\Gamma$ with coefficients in the (non-connective) $K$--theory spectrum of $R$, and the range is the (non-connective) $K$--theory of the group ring $R[\Gamma]$.  For discussions of this conjecture and its relations to geometry, see Hsiang~\cite{Hsiang-geom-applications} and Farrell--Jones~\cite{Farrell-Jones-isom}.  A great deal is known about the map (\ref{assembly-map-intro}): B\"okstedt, Hsiang, and Madsen~\cite{BHM} proved that (\ref{assembly-map-intro}) is \e{rationally} injective for
$R=\bbZ$ under the assumption that $H_*(\Gamma; \bbZ)$ is finitely generated in each degree.  Integral injectivity results were proven for geometrically finite groups with finite asymptotic dimension by Bartels~\cite{Bartels} and Carlsson--Goldfarb~\cite{Carlsson-Goldfarb}, building on Yu's work~\cite{Yu} (which established injectivity of the Baum--Connes assembly map for groups with finite asymptotic dimension).  In Section~\ref{assembly-sec}, we prove the following generalization of \cite{Bartels, Carlsson-Goldfarb}.

\begin{theorem}$\label{assembly}$
Let $\Gamma$ be a  group with finite decomposition complexity, and assume there exists a universal principal $\Gamma$--bundle $E\Gamma\to B\Gamma$ with $B\Gamma$ a finite CW complex.  Then for every ring $R$, the $K$--theoretic assembly map
$$H_* (B\Gamma; \bbK (R)) \maps K_* (R [\Gamma])$$
is a split injection for all $*\in \bbZ$.
\end{theorem}

We note that in Theorem~\ref{assembly}, the ring $R$ may be replaced by any additive category $\A$, as will be clear from our proof.  (Then $K_* (R[\Gamma])$ must be replaced by the $K$--theory of the category $\A[\Gamma]$, as defined in Bartels~\cite{Bartels}.)

Analogous methods yield an integral injectivity result for the assembly map associated to Ranicki's ultimate lower quadratic $L$--theory $\bbL^{-\infty}$.  We also obtain a large-scale version of the Borel Conjecture for bounded $K$--theory (Theorem~\ref{bdd-Borel}), analogous to~\cite[Theorems 4.3.1, 4.4.1]{GTY-rigid}.  

 Guentner, Tessera, and Yu~\cite{GTY-rigid} studied the Ranicki--Yamasaki controlled (lower) algebraic $K$-- and $L$--groups~\cite{R-Y-K-theory, R-Y-L-theory} of FDC metric spaces, and established a large-scale vanishing result formulated in terms of Rips complexes.  They used this result to study related assembly maps, leading to important geometric rigidity results (in particular, the stable Borel Conjecture).  The key technical result in the present paper is a  vanishing theorem for  \e{continuously} controlled $K$--theory, analogous to~\cite[Theorem 5.1]{GTY-rigid}.
 
\begin{theorem} $\label{vanishing-intro}$
If $X$ is a  metric space with bounded geometry and finite decomposition complexity, then
$\mocolim_s K^c_* (P_s X) = 0$ for all $*\in \bbZ$.
\end{theorem} 

 This theorem is proven in Section~\ref{FDC}.
Here $K^c_* (Z)$ denotes the continuously controlled $K$--theory of the metric space $Z$ (see Section~\ref{modules}), and \e{bounded geometry} means that for each $r>0$, there exists $N\in \bbN$ such that each ball of radius $r$ contains at most $N$ elements.  Given a bounded geometry metric space $X$ and a positive number $s$, the Rips complex $P_s (X)$ is formed from the vertex set $X$ by laying down a simplex $\langle x_0, \ldots, x_n\rangle$ whenever the pairwise distances $d(x_i, x_j)$ are all at most $s$.

The analogous result from~\cite{GTY-rigid} is proven using controlled Mayer--Vietoris sequences for Ranicki and Yamasaki's controlled lower $K$-- and $L$--groups~\cite{R-Y-K-theory, R-Y-L-theory}.  In that flavor of controlled algebra, one imposes universal bounds on the propagation of morphisms, and the Mayer--Vietoris sequences are only exact in a weak sense involving these bounds.  While it may be possible to construct quantitative versions of higher algebraic $K$--groups (analogous to the Ranicki--Yamasaki controlled lower $K$--groups, and to recent work of Oyono-Oyono and Yu in operator $K$--theory~\cite{OOY-quant}) such a theory does not currently exist.  Instead, we produce analogous (strictly exact) Mayer--Vietoris sequences in \e{continuously} controlled $K$--theory.  Loosely speaking, this corresponds (in low dimensions) to taking colimits over the propagation bounds in the Ranicki--Yamasaki theory.  Our Mayer--Vietoris sequences are produced using the machinery of Karoubi filtrations  as developed, for instance,  in C{\'a}rdenas--Pedersen~\cite{Cardenas-Pedersen}.  

In broad strokes, the proof of Theorem~\ref{vanishing-intro} is similar to the arguments in~\cite[Section 6]{GTY-rigid}.
The starting point is that the theorem holds for \e{bounded} metric spaces.
In~\cite[Section 6]{GTY-rigid}, controlled Mayer--Vietoris sequences were applied to a space $X$ covered by two subspaces, each  an $r$--disjoint union of smaller subspaces.  Great care was taken in order to keep $r$  large with respect to the other parameters involved, e.g. the Rips complex parameter and the propagation bound on morphisms.
In the present work we consider all at once a \e{sequence} of such decompositions of $X$, whose disjointness tends to infinity.  For each continuously controlled $K$--theory class $x\in K_*^c (P_s X)$, we show that at sufficiently high stages in the sequence of decompositions, $x$ can be build from classes supported on the (relative) Rips complexes of the individual factors appearing in the decompositions.  
An inductive process ensues, in which we further decompose the spaces appearing at each level of the previous sequence of decompositions.  Metric spaces with finite decomposition complexity are essentially those for which this process eventually results in (uniformly) bounded pieces.
Such considerations lead to the notion of a \e{decomposed sequence}, introduced in Section~\ref{MV-sec}.
Our approach avoids much of the intricate manipulation of various constants in~\cite[Section 6]{GTY-rigid}, but the price we pay  is that we must deal with more complicated objects than simply a metric space decomposed as a union of two subspaces.

Our approach to the assembly map makes crucial use of both ordinary Rips complexes $P_s (X)$ and the relative Rips complexes introduced in \cite{GTY-rigid}.  As the parameter $s$ increases, the simplices in $P_s (X)$ wipe out any small-scale features of $X$ and expose the large-scale structure of the space.
When $X$ is a torsion-free  group $\Gamma$ equipped with the word metric associated to a finite generating set, the Rips complexes also give a sequence of cocompact $\Gamma$--spaces approximating the universal free $\Gamma$--space $E\Gamma$ (if $\Gamma$ has torsion, they approximate the universal space for \e{proper} actions). 
Theorem~\ref{assembly} is deduced from Theorem~\ref{vanishing-intro} through a comparison between $E\Gamma$ and the Rips complexes, which shows that when $\Gamma$ is geometrically finite and has FDC, the controlled $K$--theory of $E\Gamma$ vanishes (Theorem~\ref{controlled-vanishing-EG}).

In earlier work on assembly maps in higher algebraic $K$--theory, nerves of coverings (as in Bartels~\cite{Bartels} or Carlsson--Goldfarb~\cite{Carlsson-Goldfarb}) or compactifications of the universal space $E\Gamma$ (as in Carlsson--Pedersen~\cite{Carlsson-Pedersen} or Rosenthal~\cite{Rosenthal}) played roles similar to the Rips complexes used here.  Unlike coverings and compactifications, Rips complexes  are built in a canonical way from the underlying metric space.  Together with their dual relationships to the large-scale geometry of $\Gamma$ and to the universal space $E\Gamma$, this makes Rips complexes ideally suited to the study of assembly maps.

\vspace{.2in}
\noindent {\bf Organization:}   Section~\ref{modules} reviews notions from geometric algebra.  Section~\ref{K-filt} establishes algebraic facts about Karoubi filtrations that underly our controlled Mayer--Vietoris sequences.  The sequences themselves are constructed in Section~\ref{MV-sec}.  This section begins with a general Mayer--Vietoris sequence for proper metric spaces, and then specializes this sequence to Rips complexes and relative Rips complexes.  Section~\ref{MV-sec} also introduces the terminology of \e{decomposed sequences} used extensively in Section~\ref{FDC}.
In Section~\ref{Rips-sec}, we review the necessary metric properties of Rips complexes and relative Rips complexes.  Section~\ref{FDC} reviews the notion of finite decomposition complexity and establishes our vanishing theorem for continuously controlled $K$--theory.  Assembly maps for $K$-- and $L$--theory are studied in the final section. 

\vspace{.2in}
To aid readability, we have attempted to make our indexing sets as explicit as possible.  In some arguments, the same indexed family occurs several times in one argument, and in such cases we will abbreviate expressions like   $\{Z_\alpha\}_{\alpha\in A}$ to $\{Z_\alpha\}_{\alpha}$ after their first appearance.

\vspace{.2in}
\noindent {\bf Acknowledgements:} We thank Daniel Kasprowski for pointing out an error in a previous version of the paper, and the referee for offering many suggestions that improved the exposition. The first author also  thanks Ben Wieland for helpful conversations.  


\section{Geometric modules}$\label{modules}$

Throughout this paper all metrics will be allowed to take on the value $\infty$, and all categories will be assumed to be small.  If $X$ is a metric space and $x\in X$, we set $B_r (x) = \{ y\in X \, :\, d(x, y)< r\}$ and if $Z\subset X$, we set $N_r (Z) = \{y\in X \, :\, d(y, Z) < r\}$.  We call a metric space \e{proper} if the closed ball $\{ y\in X \, :\, d(x, y)\leqs r\}$ is compact for every $x\in X$ and every $r>0$.

\begin{definition} Let $\A$ be an additive category (we think of the objects of $\A$ as ``modules'').  A geometric $\A$--module over a metric space $X$ is a function $M: X\to \Ob (\A)$.  We say that $M$ is \e{locally finite} if  its support $\supp(M) = \{x\in X: M(x) \neq 0\}$ is locally finite in $X$, in the sense that for   each compact set $K \subset X$, $\supp(M) \cap K$ is finite.  (If $X$ is proper, this is equivalent to requiring that each $x\in X$ has a neighborhood $U_x$ such that $\supp(M)\cap U_x$ is finite.)  We will usually abbreviate $M(x)$ by $M_x$, and for any subspace $Y\subset X$ we define $M(Y)$ to be the geometric module given by
$$M(Y)_x = \left\{ \begin{array}{ll} M_x, \,\,\,\,\, x\in Y,\\
      0, \,\,\,\,\,\,\, x\notin Y
      \end{array}\right.
$$

A morphism $\phi$ from a geometric module $M$ to a geometric module $N$ is a collection of morphisms $\phi_{xy} : M_y \to N_x$ for all pairs $(x,y)\in X\cross X$, subject to the condition that for each $x\in X$, the sets 
$$\{y\in X \, : \, \phi_{xy} \neq 0\} \textrm{\,\,\, and \,\,\,} \{y\in X\, : \, \phi_{yx} \neq 0\}$$
are finite.  
\end{definition}

One may think of $\phi = \{\phi_{xy}\}$ as a matrix indexed by the points in $X$, in which each row and each column has only finitely many non-zero entries.

We will deal with a fixed additive category $\A$ throughout the paper, and we will refer to geometric $\A$--modules simply as geometric modules.  The main case of interest is when $\A$ is (a skeleton of) the category of finitely generated free $R$--modules for some associative unital ring $R$.

Geometric modules and their morphisms form an additive category $\A(X)$, in which composition of morphisms is simply matrix multiplication (which is well-defined due to the row- and column-finiteness of these matrices) and addition of morphisms is defined via entry-wise sum of matrices (using the additive structure of $\A$).  
Direct sums of objects in $\A(X)$ are formed by taking direct sums pointwise over $X$.
The categories we are interested in will impose important additional support conditions on the morphisms $\phi$.

\begin{definition} 
 We say that a morphism $\phi : M \to N$ of geometric modules over $X$ has \emph{finite propagation} (or is \emph{bounded}) if there exists $R>0$ such that $\phi_{xy} = 0$ whenever $d (x, y) > R$.
 \end{definition}
 
 We may now consider the subcategory of locally finite geometric modules and $\emph{bounded}$ morphisms
 $$\Ab (X) \subset \A(X).$$
 This is again an additive category, and its $K$--theory is, by definition, the bounded $K$--theory of $X$ with coefficients in $\A$.  
 
 \begin{remark} $\label{K-rmk}$ Throughout this paper, the $K$--theory of an additive category $\C$ will mean the non-connective $K$--theory spectrum $\bbK (\C)$ as defined, for example, in~\cite[Section 8]{Cardenas-Pedersen}.
This means we consider $\C$ as a Waldhausen category, in which cofibrations are (up to isomorphism) inclusions of direct summands and weak equivalences are isomorphisms.  Since inclusions of direct summands can be characterized in terms of split exact sequences, additive functors $\C\to \D$ always preserve these notions of cofibration and weak equivalence, and hence induce maps $\bbK(\C)\to \bbK(\D)$.  We set $K_* (\C) = \pi_* \bbK (\C)$ for $*\in \bbZ$.
 \end{remark}

Next, we will consider the notion of \e{continuously controlled morphisms}, which will be the main object of study in this paper.  Here and in what follows, we give the half-open interval $[0,1)$ the usual Euclidean metric $d(s, t) = |s-t|$, and for metric spaces $(X, d_X)$ and $(Y, d_Y)$, we give $X\cross Y$ the metric $d\left((x,y), (x', y')\right) = d_X (x, x') + d_Y (y, y')$.  The following definition appears in Weiss~\cite{Weiss}, and is a slight variation on the work of Anderson--Connolly--Ferry--Pedersen~\cite{ACFP}.

\begin{definition} $\label{control}$
A morphism $\phi: M\to N$ of geometric modules over $X\cross [0,1)$ is \emph{continuously controlled at 1} if for each $x\in X$ and each neighborhood $U$ of $(x,1)$ in $X\cross [0,1]$, there exists a (necessarily smaller) neighborhood $V$ of $(x,1)$ such that $\phi$ does not cross $U\setminus V$: that is, if $v\in V$ and $y\notin U$, then $\phi_{yv} = \phi_{vy} = 0$.  
\end{definition}

It is an exercise to check that the collection of continuously controlled morphisms in $\Ab(X\cross [0,1))$ form a subcategory.  Since the control condition only depends on the support of the morphism, this collection of morphisms is also closed under addition and negation, and direct sums in this subcategory agree with direct sums in $\Ab (X\cross [0,1))$. 

\begin{definition}$\label{controlled-mod}$ Let $X$ be a proper metric space.  The category of locally finite geometric modules over $X\cross [0,1)$ and continuously controlled morphisms, denoted 
$\Ac (X)$, is the subcategory of $\Ab(X\cross [0,1))$ containing all objects, but only those morphisms with continuous control at 1.  As explained above, this is an additive subcategory of $\A(X\cross [0,1))$.

For $Z\subset X$,  we will write $\Ac^{X} (Z)$ for the category of controlled modules on $Z\cross [0,1)$, where $Z$ has the metric inherited from $X$.  (This will be especially relevant when $Z$ and $X$ are simplicial complexes, since then $Z$ has its own intrinsic simplicial metric, giving rise to a different category of controlled modules.)  

Given a closed subset $Z \subset X$, we define   
$$\Ac^{X +} (Z) \subset \Ac (X)$$
to be the full subcategory on those geometric modules $M\in \Ac(X)$
which are supported ``near" $Z\cross [0,1)$; that is, $M\in \Ac^{X +} (Z)$ if and
only if there exists $R>0$ such that $M_{(x,t)}\neq 0$ implies $d(x, Z) <
R$.  When $X$ is clear from context, we will simply write $\Ac^+ (Z)$ rather than $\Ac^{X +} (Z)$.
\end{definition}

\begin{remark} In Weiss~\cite[Section 2]{Weiss}, a slightly different support condition for modules is used to define an analogue of our category $\Ac (X)$: namely 
the support of each module is required to be a discrete, closed subset of $X\cross [0,1)$.  This condition is equivalent to our local finiteness condition when $X$ is a proper metric space, so that our category $\Ac (X)$ is the same as Weiss's $\A \left(X\cross [0,1], X\cross [0,1)\right)$.  In this paper, we only need to consider $\Ac (X)$ for proper metric spaces $X$.
\end{remark}

The spaces whose controlled $K$--theory appears in this paper will all be simplicial complexes.
We will assume all our simplices have diameter one.  More specifically, we identify the simplex with vertices $x_1, \ldots, x_n$ with the convex hull of the points $\frac{\sqrt{2}}{2}e_i \in \bbR^n$, where the $e_i$ are the standard basis vectors.

Given a simplicial complex $K$, the simplicial metric $d_\Delta$ on $P$ is the unique path-length metric which restricts to the standard Euclidean metric on each simplex.  Explicitly, 
$$d_{\Delta} (x, y) = \inf \sum_{i = 0}^{N-1} \sd(p_i, p_{i+1})$$
where the infimum is taken over all sequences $x = p_0, p_1, \ldots, p_N = y$ (with $N$ arbitrary) such that $p_i$ and $p_{i+1}$ lie in the same simplex of $K$, and $\sd(p_i, p_{i+1})$ is the Euclidean metric on a simplex containing both points.  When $x$ and $y$ lie in different path components of $K$, we set $d_{\Delta} (x, y) = \infty$.
Note that locally finite simplicial complexes are always proper with respect to their simplicial metrics (this follows, for example, from the argument in Lemma~\ref{metric-comp} below, which can be used to show that each ball contains finitely many vertices).  All simplicial complexes in this paper will be equipped with the simplicial metric (possibly restricted from some larger complex).

We will need a lemma regarding the functoriality of controlled $K$--theory for maps between metric spaces.  Versions of the following result are stated (without proof) in~\cite{Bartels-Rosenthal, Bartels, Weiss}; an equivariant version is proven in~\cite[Lemma 3.3]{BFJR}.  For completeness, we sketch the argument. 

\begin{lemma} $\label{functoriality}$  
Let $f: X\to Y$ be a continuous map of proper metric spaces which is proper 
(that is, $f^{-1} (C)$ is compact in $X$ for all compact sets $C\subset Y$) and metrically coarse (that is, for each $R>0$ there exists $S>0$ such that $d_X (x_1, x_2) < R$ implies $d_Y (f(x_1), f(x_2))< S$).  

Then $f$ induces a functor 
\begin{equation*}  f_*\co \Ac (X) \to \Ac (Y).\end{equation*}

Moreover, if $X' \subset X$ and $Y'\subset Y$ are closed subspaces with $f(X') \subset N_t (Y')$ for some $t>0$, then $f$ induces a functor 
\begin{equation*} \label{ii}f_*\co \Acp (X') \to \Acp (Y').\end{equation*}

In particular, given a commutative diagram
$$\xymatrix{ P' \ar@{^{(}->}[r] \ar@{^{(}->}[d] & P \ar@{^{(}->}[r] \ar@{^{(}->}[d] & P'' \ar[d]\\
		    Q' \ar@{^{(}->}[r] & Q \ar@{^{(}->}[r] & Q'' 
		    }
$$		 
 of simplicial maps between locally finite simplicial complexes (with all but the right-hand vertical map injective) there is an induced functor
\begin{equation*}  \label{iii} \Ac^{P +} (P') \to \Ac^{Q +} (Q'),\end{equation*}
where $P'$ and $P$ are given the subspace metrics inherited from the simplicial metric on $P''$, while $Q'$ and $Q$ are given the subspace metrics inherited from the simplicial metric on $Q''$.
\end{lemma}
\begin{proof}  We will construct the functor  $f_*\co \Acp (X') \to \Acp (Y')$; the other functors are special cases (note that simplicial maps decrease distances).

Let $M$ be a geometric module in $\Acp(X')$.  If $f$ is injective, we set $f_* (M)_{(y,t)} = M_{f^{-1} (y,t)}$.
If $f$ is not injective, one needs to redefine the category $\Acp(-)$ so that setting 
$$f_* (M)_{(y,t)} = \bigoplus_{x\in f^{-1}(y)} M_{(x,t)}$$
is well-defined.   We will ignore this technicality in what follows; see~\cite[Section 2]{Weiss} for details.
Since $M$ is supported on a neighborhood of $X'$, $f$ is metrically coarse, and $f(X') \subset N_t (Y')$, the module $f(M)$ will be supported on a neighborhood of $Y'$.  The behavior of $f$ on morphisms is defined similarly; since $f$ is metrically coarse, we see that $f(\phi)$ has finite propagation for each $\phi\in \Acp(X')$.  Finally, we must check that for each $\phi\in \Acp(X')$, $f(\phi)$ is continuously controlled.  
Fix $y\in Y$, and consider a neighborhood $U$ of $(y,1)$ in $Y\cross I$.  Replacing $U$ with a small ball around $(y,1)$ if necessary, we may assume that the closure $\overline{U}$ is compact.  Let $U' = (f\cross \Id_I)^{-1} (U)$.
For each $x\in f^{-1} (y)$, $U'$ is a neighborhood of $(x,1)$, so there exists a smaller neighborhood $V_x$ of $(x,1)$ such that $\phi_{z', z} = 0$ if $z\in V_x$, $z'\notin U'$ or $z\notin U'$, $z'\in V_x$.  Since $f$ is proper, $f^{-1} (y)$ is compact, so we may cover $(f\cross \Id_I)^{-1} (y,1)$ by finitely many of the sets $V_x$, say $V_{x_1}, \ldots,V_{x_n}$.
Since $f$ is proper and continuous and $\overline{U}$ is compact, it follows that 
$$C := (f\cross \Id_I) \left( (f\cross \Id_I)^{-1} (\overline{U}) \setminus \bigcup_{i=1}^n V_{x_i} \right)$$
is compact.  Now $V = U \setminus C$ is a neighborhood of $(y,1)$, and one may check that $\phi$ does not 
cross $U\setminus V = C\cap U$.
\end{proof}

\begin{remark} Most uses of Lemma~\ref{functoriality} in the sequel will only require the statement regarding simplicial complexes.  Note that by setting $P' = P$ and $Q' = Q$, we obtain a statement about the categories $\Ac (-)$.

The functors constructed in Lemma~\ref{functoriality} combine to yield a functor from the category of proper metric spaces and continuous, metrically coarse \e{injections} into the category of small categories.  This makes the various colimits of categories considered later in the paper well-defined.  (For non-injective maps, one needs to be careful in order to make composition strictly associative at the categorical level; this is achieved by Weiss's construction~\cite[Section 2]{Weiss}.  Until Section~\ref{assembly-sec}, all the maps we consider are injective.)
\end{remark}

\section{Karoubi Filtrations}$\label{K-filt}$

We will use the notion of a Karoubi filtration to produce various Mayer--Vietoris sequences in controlled $K$--theory.  Algebraically, a Karoubi filtration is a tool for collapsing a full subcategory of an additive category; geometrically it is a method for producing fibrations of $K$--theory spectra. 

By abuse of notation we will write $A\in \A$ to mean that $A$ is an object in $\A$.  Furthermore, we will write $A = A_1\oplus A_2$ to mean that there exist maps $i_j: A_j \to A$ making $A$ the categorical direct sum of $A_1$ and $A_2$.  We will always implicitly choose particular maps $i_j$, and we will denote the corresponding projections $A\to A_j$ by $\pi_j$.

\begin{definition} Let $\mS \subset \A$ be a full additive subcategory of a small additive category $\A$.  A \emph{Karoubi filtration} on the pair $(\A, \mS)$ consists of an index set $I$ and for each $A\in \A$ and each $i\in I$, a direct sum decomposition $A = A_i \oplus A'_i$ with $A_i\in \mS$.  These data must satisfy the following conditions:
\begin{enumerate}
\item For each morphism $A\srt{f} S$ (with $S\in \mS$) there exists $i\in I$ such that $f$ factors as 
$$A = A_i \oplus A'_i \stackrel{\pi_1}{\maps} A_i \maps S$$
\item For each morphism $S\srt{g} A$ (with $S\in \mS$)  there exists $i\in I$ such that $g$ factors as
$$S \maps A_i \stackrel{i_1}{\maps} A_i \oplus A'_i = A$$
\item The index set $I$ is a directed poset under the relation $i \leqs j \iff$ for all $A\in \A$, $A_i$ is a direct summand of $A_j$ and $A'_j$ is a direct summand of $A'_i$.  (Here \emph{directed} means that for each $i, j\in I$, there exists $k\in I$ such that $i, j\leqs k$.)
\item For each $A, B\in \A$ and each $i\in I$, we have $(A\oplus B)_i = A_i \oplus B_i$ and $(A\oplus B)'_i = A'_i \oplus B'_i$.
\end{enumerate}
\end{definition}

\begin{remark} In the literature on Karoubi quotients, the term ``filtered" is often used instead of ``directed."  In category theory, the term ``directed" is standard.
\end{remark} 

For any full additive subcategory $\mS \subset \A$, the Karoubi quotient $\A/\mS$ is the category with the same objects as $\A$, but with two morphisms identified if their difference factors through an object of $\mS$.  The following lemma is surely well-known, but seems not to have been made explicit previously.

\begin{lemma}$\label{sums}$ If $\mS$ is a full additive subcategory of the additive category $\A$, then $\A/\mS$
is an additive category, and if $A_1\srm{i_1} A \stackrel{i_2}{\longleftarrow} A_2$ is a direct sum diagram in $\A$, then $A_1\srm{[i_1]} A \stackrel{[i_2]}{\longleftarrow} A_2$ is a direct sum diagram in $\A/\mS$.
\end{lemma}
\begin{proof}  It is elementary to check that $\A/\mS$ is a category.  The addition on morphisms is given by $[\phi]+[\psi] = [\phi + \psi]$.  This is well-defined because if $\phi\co A\to B$ factors through $S\in \mS$ and $\psi\co A\to B$ factors through $S'\in \mS$, then $\phi + \psi$ 
factors through $S\oplus S' \in \mS$.

Now, say $A_1\srm{i_1} A \stackrel{i_2}{\longleftarrow} A_2$ is a direct sum diagram in $A$.  Given a diagram
\begin{equation}\label{phi} \xymatrix{
	A_1 \ar[r]^{[i_1]} \ar[dr]_{[f]} & A \ar@{.>}[d] & A_2 \ar[l]_{[i_2]} \ar[dl]^{[g]}\\
	& C
		}
\end{equation}
in $\A/\mS$ we must show that 
there is a unique morphism $A\to C$ making  (\ref{phi}) commute.  For any direct sum $f\oplus g$, the map  $[f\oplus g]$ makes (\ref{phi})  commute, so it suffices
to show that if $A\stackrel{[\phi]}{\maps} C$ makes (\ref{phi}) commute, then $f\oplus g - \phi$ factors through an object of $\mS$. 
Writing  $\phi = (\phi \circ i_1) \oplus (\phi \circ i_2)$, we have 
\begin{equation} \label{f+g} f\oplus g - \phi = (f - \phi \circ i_1) \oplus (g - \phi \circ i_2).\end{equation}
If $[\phi]$ makes (\ref{phi}) commute, then $f - \phi \circ i_1$ and $g - \phi \circ i_2$ factor through objects $S_1$ and $S_2$ in $\mS$ (respectively), so  $f - \phi \circ i_1$ and $g - \phi \circ i_2$ are the composites
$$A_1 \stackrel{\alpha_1}{\maps} S_1 \stackrel{\beta_1}{\maps} C \,\,\,
\textrm{  and  }\,\,\,
 A_2 \stackrel{\alpha_2}{\maps} S_2 \stackrel{\beta_2}{\maps} C,$$
(respectively) for some morphisms $\alpha_k$ and $\beta_k$  in $\A$ ($k = 1, 2$). 
We now see that $(f - \phi \circ i_1) \oplus (g - \phi \circ i_2)$ factors through $S_1 \oplus S_2 \in \mS$ since (letting $j_1$ and $j_2$ denote the inclusions of the summands into $S_1 \oplus S_2$), the diagram
$$\xymatrix{A_1 \ar@/_3pc/[rrdd]_{f-\phi\circ i_1}  \ar[rr]^{i_1} \ar[dr]_{\alpha_1} && A \ar[d]_{j_1 \alpha_1}^{\oplus j_2 \alpha_2} & & A_2 \ar[ll]_{i_2} \ar[dl]^{\alpha_2} \ar@/^3pc/[lldd]^{g-\phi\circ i_2}\\
& S_1 \ar[r]^(.35){j_1} \ar[dr]_{\beta_1} 
& S_1 \oplus S_2 \ar[d]_{\beta_1}^{\oplus \beta_2} 
& S_2 \ar[l]_(.35){j_2} \ar[dl]^{\beta_2}\\
		  &&C
 }
$$
commutes in $\A$ by construction.   
\end{proof}

The utility of Karoubi filtrations comes from the following result due to Pedersen--Weibel~\cite{Pedersen-Weibel};  see also C{\'a}rdenas--Pedersen~\cite[Section 8]{Cardenas-Pedersen}.

\begin{theorem}  $\label{LES}$ If $\mS \subset \A$ is a full additive subcategory of a small additive category $\A$ and $(\A, \mS)$ admits a Karoubi filtration, then then there is a long exact sequence in non-connective algebraic $K$--theory
$$\cdots \maps K_*\mS \maps K_* \A \maps K_* \A/\mS \stackrel{\partial}{\maps} K_{*-1} \mS \maps \cdots$$
\end{theorem}

For our purposes, the key examples of Karoubi filtrations arise from restricting the support of geometric modules.

\begin{definition} Given any family of subspaces $\Y$ of a proper metric space $X$ we may
consider the full subcategory $\Ac(\Y)\subset \Ac(X)$ on those
modules supported on $Y\cross [0,1)$ for some $Y\in \Y$.  Note that $\Ac^{X +} (Z) = \Ac (\{N_r (Z) \, : \, r\in \bbN\}) \subset \A(X)$.  
The category $\Ac(\Y)$ is unchanged if we enlarge $\Y$ be adding subspaces of elements in $\Y$, so we may always assume that our families are closed under taking subspaces.
\end{definition}

The following lemma is a special
case of Bartels and Rosenthal~\cite[(5.7)]{Bartels-Rosenthal}.

\begin{lemma}$\label{filtrations}$ Let $\Y$ and $\Z$ be families of
  subspaces of a proper metric space $X$, and assume $\Y$ and $\Z$ are closed under finite unions.  Then 
  $\Ac(\Y), \Ac(\Z)\subset \Ac(X)$ are additive subcategories, and if
 for all $Y\in \Y$ there exists $Z\in \Z$ 
  with $Y\subset Z$, then $\Ac (\Y)$ is a full (additive) subcategory of $\Ac (\Z)$.
If, in addition, for each $Y\in \Y$
  and each
  $r\in \bbN$
  there exists $Y'\in \Y$ such that $N_r (Y) \subset Y'$, then the
  inclusion
$\Ac(\Y) \subset \Ac(\Z)$
admits a Karoubi filtration.

In particular, for any subspace $Z\subset X$, the pair
$\Ac^{+} (Z) \subset \A(X)$
admits a Karoubi filtration.   
\end{lemma}

The direct sum decompositions making up these Karoubi filtrations come from the following construction, applied
to the subspaces $Y\in \Y$.  

\begin{construction}$\label{sum}$
Let $Z\subset X$ be a subspace of the proper metric space $X$.  For
any $M\in \Ac(X)$, the inclusions 
$$M(Z\cross [0,1)) \injects M \textrm{  and  } M\left(  (X\setminus Z) \cross [0,1)\right)\injects M$$
 yield a direct sum decomposition
$$M = M(Z\cross [0,1)) \oplus M\left(  (X\setminus Z) \cross [0,1)\right).$$
\end{construction}

This follows from the fact that a morphism $M\to N$ is defined as a family of morphisms $M_{(x,t)} \to N_{(y,s)}$ for all $(x,t), (y,s)\in X\cross [0,1)$.

\begin{definition} 
Let $\A_1$ and $\A_2$ be full subcategories of an additive category $\A$.
Let $\Ai$ (the intersection of $\A_1$ and $\A_2$) be the full subcategory generated by those objects lying in both $\A_1$ and $\A_2$.  
\end{definition}

\begin{remark}$\label{filtration-rmk}$
Note that if $\Y$ and $\Z$ are families of
  subspaces of a metric space $X$, then the intersection category
 $\Ac(\Y)\cap \Ac (\Z)$ is simply $\Ac (\{Y\cap Z \, :\, Y\in \Y, Z\in \Z\})$. 
 In particular, if $X_1, X_2\subset X$, then 
 $$\Ac^+(X_1) \cap \Ac^+ (X_2) = \Ac(\{N_r (X_1) \cap N_s (X_2)\}_{r, s \in \bbN}),$$
and Lemma~\ref{filtrations} shows that $\Ac^+(X_1) \cap \Ac^+ (X_2) \subset \Ac(X)$
admits a Karoubi filtration.  With the exception of the inclusion $\Ac(X)_{<1} \subset \Ac(X)$ discussed in Section~\ref{assembly-sec}, all the Karoubi filtrations in this paper follow from Lemma~\ref{filtrations} by similar arguments.
\end{remark}

We need another technical condition for some of our arguments.

\begin{definition} 
Let $\A_1$ and $\A_2$ be full, additive subcategories of the additive category $\A$.
We say that $(\A_1, \A_2)$ is \emph{dispersed} if every morphism $\phi: A_1 \to A_2$ (with $A_i\in \A_i$) factors through an object in $\Ai$.
\end{definition}

The dispersion conditions encountered in this paper are all special cases of the following observation.

\begin{lemma}$\label{disp}$
Let $X$ be a metric space, and consider a family of subspaces $\{X_i\}_{i\in I}$.  Assume that for each $i\in I$ and each $r\in \bbN$, there exists $j\in I$ such that $N_r (X_i) \subset X_j$.  If $\mS\subset \Ac (X)$ is a full additive subcategory that is closed under restriction of modules (meaning that for all $S\in \Ob (\mS)$ and for all $Z\subset X$, $S(Z\cross [0,1)) \in \Ob (\mS)$), then the pair $(\Ac (\{X_i\}_i), \mS)$ is dispersed.
\end{lemma}
\begin{proof} Consider a morphism $\phi\co M\to S$ in $\Ac (X)$, with $M\in \Ob (\Ac (\{X_i\}_i))$ and $S\in \Ob (\mS)$.  Then $\supp (M) \subset X_i\cross [0,1)$ for some $i\in I$, and 
$$\left\{(z,t) : \phi_{(z,t), (z', t')} \neq 0 \textrm{ for some } (z',t')\right\}\subset N_r (X_i)\cross [0,1) \subset X_j\cross [0,1)$$
for some $r>0$ and some $j\in J$.  Now $\phi$ factors through $S(X_j\cross [0,1))$.
\end{proof}

In Section~\ref{MV-sec} we will build Mayer--Vietoris sequences in continuously controlled $K$--theory.  These sequences will be applied in Section~\ref{FDC} to spaces of the form $\coprod_{r=1}^\infty Z^r$, covered by subspaces  $\coprod_{r=1}^\infty U^r$ and  $\coprod_{r=1}^\infty V^r$.  These decompositions $Z^r = U^r \cup V^r$ will become finer (in a sense) as $r$ increases, and we will want to ignore the subcategory 
$\Ac\left( \left\{\coprod_{r=1}^R Z^r\right\}_{R\geqs 1}\right)$.  This will be done through the use of Karoubi quotients, and in the remainder of this section we discuss the necessary categorical set-up. 

\begin{lemma} $\label{quot}$
Let $\mS, \B \subset \A$ be full additive subcategories of the additive category $\A$.  Assume that:
\begin{enumerate}
\item the pairs $(\B, \mS\cap \B)$, $(\A, \mS)$, and $(\A, \B)$ admit Karoubi filtrations;  
\item $(\B, \mS)$ is dispersed.
\end{enumerate}
Then the full
subcategory of $\A/\mS$ on the objects of $\B$ is precisely $\B/(\mS \cap \B)$,  
and the inclusion $\B/(\mS\cap \B) \subset \A/\mS$ admits a Karoubi filtration.
\end{lemma}
\begin{proof}  We begin by examining the full subcategory of $\A/\mS$ on the objects of $\B$.  This category is formed by identifying two morphisms $\phi, \psi: B_1\to B_2$ ($B_i\in \B$) if $\phi - \psi$ factors as
$$B_1 \stackrel{\alpha}{\maps} S \stackrel{\beta}{\maps} B_2$$
for some $S\in \mS$.  Since $(\B, \mS)$ is dispersed, $\alpha$ factors through an object of $\B\cap \mS$, so $\phi \equiv \psi$ (modulo $\B\cap \mS$).     Hence the full subcategory of $\A/\mS$ on the objects of $\B$ is precisely $\B/(\B\cap \mS)$.

Next, we must show that the inclusion $\B/(\B\cap \mS) \subset \A/\mS$ admits a Karoubi filtration.  The filtration on $\B/(\B\cap \mS) \subset \A/\mS$ is exactly the same as the filtration on $\B \subset \A$: for $A\in \A$,
let $I$ denote the indexing set for the latter filtration.  Then for each $i\in I$ we have a decomposition $A = B_i \oplus B'_i$ in $\A$ (with $B_i\in \B$), and this remains a direct sum decomposition in the category $\A/\mS$ by Lemma~\ref{sums}.  It now follows from the definitions that these decompositions give a Karoubi filtration on $\B/(\B\cap \mS) \subset \A/\mS$.
\end{proof}

We record the universal property of Karoubi quotients, which we will use several times.  The proof is an elementary exercise.

\begin{lemma}$\label{induced-map}$ Say $G: \A \to \B$ is a functor between additive categories and $\mS\subset \A$ is a full additive  subcategory admitting a Karoubi filtration.  If $G(\phi) = 0$ whenever $\phi\equiv 0$ (mod $\mS$), then there is a unique additive functor
$$\overline{G} \co \A/\mS \maps \B$$
such that the composite $\A \to \A/\mS \xmaps{\overline{G}} \B$ equals $G$.
\end{lemma}

Our Mayer--Vietoris sequences will be built using the following version of the Third Isomorphism Theorem from elementary abstract algebra.

\begin{proposition}$\label{3rd}$ Let $\A$ be a small additive category with full additive subcategories $\mS$, $\A_1$ and $\A_2$, and assume that $\A_1$ and $\A_2$ generate $\A$ in the sense that every $A\in \A$ admits a direct sum decomposition $A = A_1 \oplus A_2$ with $A_i\in \A_i$.  Set $\A_{12} = \A_1 \cap \A_2$, and similarly set $\mS_1 = \mS \cap \A_1$, $\mS_2 = \mS\cap \A_2$, and $\mS_{12} = \mS\cap \A_{12}$.

If the triples $\mS, \A_1 \subset \A$ and $\mS_2, \A_{12} \subset \A_2$ satisfy the conditions of Lemma~\ref{quot} (in other words, if $(\A_1, \mS)$ and $(\A_{12}, \mS_2)$ are dispersed, and all the relevant inclusions admit Karoubi filtrations), and if the pairs $(\A_1, \A_2)$ and $(\A_2, \mS)$ are dispersed, then the inclusion $\A_2 \injects \A$ induces an equivalence of categories between Karoubi quotients as follows:
$$ \frac{\A_2/\mS_2}{\A_{12}/\mS_{12}} \stackrel{\heq}{\maps} \frac{\A/\mS}{\A_1/\mS_1}$$
\end{proposition}
\begin{proof}  Lemma~\ref{quot} guarantees that the displayed Karoubi quotients are well-defined.  We begin by checking that the composite
\begin{equation} \label{Q} \A_2 \injects \A \maps \A/\mS \maps \frac{\A/\mS}{\A_1/\mS_1}
\end{equation}
factors through the identifications in  $\frac{\A_2/\mS_2}{\A_{12}/\mS_{12}}$, so that Lemma~\ref{induced-map} yields a well-defined functor 
$F\co  \frac{\A_2/\mS_2}{\A_{12}/\mS_{12}} \to  \frac{\A/\mS}{\A_1/\mS_1}$.   

If $A_2, A'_2\in \A_2$ and $\phi:A_2\to A'_2$ is a morphism in $\A$ which is equivalent to zero in $\frac{\A_2/\mS_2}{\A_{12}/\mS_{12}}$, then $\phi$ must factor through an object in either $\mS_2$ or $\A_{12}$.  These are subcategories of $\mS$ and $\A_1$ (respectively), so such morphisms certainly map to zero under the composite (\ref{Q}).  Applying Lemma~\ref{induced-map} twice yields the desired functor $F$.

We must show that, up to isomorphism, every object is in the image of $F$.
Every object $A\in \A$ can be written in the form $A = A_1 \oplus A_2$ with $A_i \in \A_i$, and we claim that in the Karoubi quotient $\frac{\A/\mS}{\A_1/\mS_1}$, the objects $A$ and $A_2$ are isomorphic.  Indeed, the inclusion $i_2: A_2 \to A$ and the corresponding projection $\pi_2 : A \to A_2$ are inverses in this Karoubi quotient (the composite $\pi_2 i_2$ is the identity on $A_2$ by definition, and $\Id_A = i_1 \pi_1 + i_2 \pi_2$, so $\Id_{A} - i_2 \pi_2 = i_1 \pi_1$, which factors through $A_1$).  This shows that up to isomorphism, every object is in the image of the map $F$.

To complete the proof, we must check that $F$ is full and faithful.  Fullness follows from the fact that $\A_2$ is a full subcategory of $\A$.  
Next, if $\phi_1$ and $\phi_2$ are morphisms in $\A_2$ that are equivalent in $(\A/\mS)/(\A_1/\mS_1)$, then $\phi_1 - \phi_2$ factors through an object in either $\mS$ or $\A_1$.  Dispersion implies that $\phi_1 - \phi_2$ actually factors through an object of $\mS_2$ or $\A_{12}$, so $\phi_1$ and $\phi_2$ are equivalent in the domain of $F$.  Hence $F$ is faithful.
\end{proof}


\section{Mayer--Vietoris sequences in continuously controlled $K$--theory}$\label{MV-sec}$

In this section we build Mayer--Vietoris sequences in continuously controlled $K$--theory analogous to the controlled Mayer--Vietoris sequences of Ranicki--Yamasaki~\cite{R-Y-K-theory, R-Y-L-theory} (see also~\cite[Appendix B]{GTY-rigid}).
First we produce a general Mayer--Vietoris sequence for metric spaces, and then we specialize the construction to the Rips complexes that will be used in later sections.

\subsection{A Mayer--Vietoris sequence for the continuously controlled $K$--theory of metric spaces}

\begin{proposition} $\label{MV}$ Let $X$ be a proper metric space with subspaces $X_1, X_2\subset X$ and assume that for some $r>0$, $N_r (X_1)\cup N_r (X_2) = X$. 
Consider a family $\{S_i\}_{i\in I}$ of subspaces of $X$ such that for each $t\in \bbN$ and each $i\in I$, there exists $j\in I$ such that $N_t (S_i)\subset S_j$.
Let $\mS = \Ac (\{S_i\}_{i\in I})$.
  We denote the
  intersection of $\mS$ with $\Ac^+(X_i)$ by $\mS_i$, and we denote
  the intersection of $\mS$ with $\Ac^+(X_1)\cap \Ac^+(X_2)$ by $\mS_{12}$.

Then the natural maps from $\Acp(X_1)/\mS_1$ and $\Acp(X_2)/\mS_2$ to $\Ac(X)/\mS$ are isomorphisms onto their images, and the natural map
$$\frac{\Acp (X_1) \cap \Acp (X_2)}{\mS_{12}} \maps \Ac (X)/\mS$$
is an isomorphism onto the intersection of the images of $\Acp(X_1)/\mS_1$ and $\Acp(X_2)/\mS_2$.
Moreover, there is a long-exact Mayer--Vietoris sequence in non-connective $K$--theory
\begin{eqnarray*}\cdots\maps \K_{*+1} \left(\Ac (X)/\mS\right)\srm{\partial}
 \K_* \left(\frac{\Ac^+ (X_1) \cap \Ac^+ (X_2)}{\mS_{12}}\right)\hspace{.5in}\\
 \xrightarrow{((i_1)_*, (i_2)_*)} \K_*(\Ac^+ (X_1)/\mS_1)\oplus \K_* (\Ac^+ (X_2)/\mS_2)\hspace{.2in}\\
\xrightarrow{(j_1)_* - (j_2)_*} \K_* (\Ac (X)/\mS) \stackrel{\partial}{\maps} \cdots, 
\end{eqnarray*}
in which $i_1, i_2, j_1$, and $j_2$ are induced by the relevant
inclusions of categories.
\end{proposition}

\begin{proof}   
To construct the Mayer--Vietoris sequence, we will consider the diagram of additive categories
\begin{equation}\label{MV-square}
\xymatrix{
\frac{\Ac^+ (X_1) \cap \Ac^+ (X_2)}{\mS_{12}} \ar[rr]^{i_1} \ar[d]^{i_2} 
      & &\Ac^+ (X_1)/\mS_1 \ar[d]^{j_1} \\
\Ac^+ (X_2)/\mS_2 \ar[rr]^{j_2} \ar[d]^{q_2} 
      & &\Ac  (X)/\mS \ar[d]^{q_1}\\
\frac{\Ac^+ (X_2)/\mS_2}{\left(\Ac^+ (X_1) \cap \Ac^+ (X_2)\right)/\mS_{12}} \ar[rr]^(.61)F_(.61){\isom}
& &\frac{\Ac (X)/\mS}{\Ac^+(X_1)/\mS_1},
}
\end{equation}
where the $q_i$ are the Karoubi projections guaranteed by Lemma~\ref{quot} and (as we will check)
the induced map $F$ is an equivalence of categories by
Proposition~\ref{3rd}.
Applying $K$--theory  produces two vertical long-exact sequences (Theorem~\ref{LES}),
which can be weaved together using the isomorphism $F_*$ to form the
desired Mayer--Vietoris sequence (see, for example, Hatcher~\cite[Section 2.2, Exercise 38]{Hatcher}).
The facts that $\Acp(X_1)/\mS_1$, $\Acp(X_2)/\mS_2$, and $\frac{\Acp (X_1) \cap \Acp (X_2)}{\mS_{12}}$ are isomorphic to their images in $\Acp (X)/\mS$ will also follow from Lemma~\ref{quot}.  
The Karoubi filtrations needed to apply Lemma~\ref{quot} come from Lemma~\ref{filtrations} (see Remark~\ref{filtration-rmk}), while the necessary dispersion conditions can be checked using Lemma~\ref{disp}.

To complete the proof, we must check that the conditions of
Proposition~\ref{3rd} are satisfied, so that we obtain a well-defined equivalence of categories
$$\frac{\Ac^+ (X_2)/\mS_2}{\left(\Ac^+ (X_1) \cap \Ac^+ (X_2)\right)/\mS_{12}} \srm{F} \frac{\Ac (X)/\mS}{\Ac^+(X_1)/\mS_1}.$$
The necessary Karoubi filtrations and the dispersion conditions are checked using Lemmas~\ref{filtrations} and~\ref{disp}.
To check that $\Ac^+(X_1)$ and $\Ac^+(X_2)$ generate $\Ac^+(X)$, recall that Construction~\ref{sum} guarantees a direct sum decomposition 
$$M = M\left(N_{r}(X_1) \cross [0,1)\right) \oplus M\left( (X\cross [0,1))\setminus (N_{r} (X_1)\cross [0,1))\right).$$ 
Since $N_r (X_1) \cup N_r (X_2) = X$, we have 
$$(X\cross [0,1))\setminus (N_{r} (X_1)\cross [0,1))\subset N_r (X_2) \cross [0,1),$$
and hence $M\left( (X\cross [0,1))\setminus (N_{r} (X_1)\cross [0,1))\right)\in \Acp (X_2)$. 
\end{proof}

\subsection{Decomposed sequences and Rips complexes}

We will apply our general Mayer--Vietoris sequence (Proposition~\ref{MV}) to decompositions of Rips complexes arising from decompositions of the underlying metric space.  For the proof of our vanishing result for continuously controlled $K$--theory, it will be necessary to consider an infinite sequence of increasingly refined decompositions of our space.  In fact, we will need to consider such sequences all at once by forming an infinite disjoint union of the spaces involved in the decompositions, and we will need to iterate this process (by further decomposing each space in the initial decomposition).  Such considerations lead to the notion of \e{decomposed sequence} introduced below.

We begin by recalling the construction of the Rips complex.

\begin{definition} Given a metric space $X$ and a number $s>0$, the Rips complex $P_s (X)$ is the simplicial complex with vertex set $X$ and with a simplex $\langle x_0, \ldots, x_n\rangle$ whenever $d(x_i, x_j) \leqs s$ for all $i, j\in \{0, \ldots, n\}$.
\end{definition} 

We will often view $X$ as a subset of $P_s (X)$ by identifying $X$ with the vertices of  $P_s (X)$.

Note that if $X$ is a metric space with bounded geometry (i.e. if for each $r>0$ there exists $N>0$ such that for all $x\in X$, the ball $B_r (x)$ contains at most $N$ points), then the Rips complex $P_s (X)$ is finite dimensional and locally finite.  When forming Rips complexes, we will always assume that the underlying metric space has bounded geometry.  Note that a finitely generated group, with the word metric arising from a finite generating set, always has bounded geometry.  This is our main source of examples.  

\begin{definition}
Let $X$ be a bounded geometry metric space and consider a sequence of subspaces $\Z = (Z^1, Z^2, \ldots )$, $Z^i \subset X$, equipped with decompositions 
\begin{equation}\label{decomp}Z^r = \bigcup_{\alpha\in A_r} Z^r_\alpha\end{equation}
 of each $Z^r$ ($r = 1, 2, \ldots$).
 We will call this data (the sequence $\Z$ together with the families $\{Z^r_\alpha\}_{\alpha\in A_r}$) a \emph{decomposed sequence} in $X$.  Note that  the $Z^r_\alpha$ need not be disjoint.

Let $\Seq$ denote the partially ordered set consisting of all non-decreasing sequences of (strictly) positive real numbers, with the ordering $(s_1, s_2, \ldots) \leqs (s_1', s_2', \ldots)$ if $s_i \leqs s_i'$ for all $i$.  Note that $\Seq$ is directed.

Given a decomposed sequence $\Z$ in $X$ and a sequence $\bs\in \Seq$, the Rips complex $P_\bs (\Z)$ is the simplicial complex
$$P_{\bs} (\Z) = \coprodmo_{r=1}^\infty \coprod_{\alpha \in A_r} P_{s_r} (Z^r_\alpha).$$
Note that $Z^r_\alpha$ and $Z^{s}_\beta$ may overlap inside of $X$, so to be precise, points in $P_{\bs} (\Z)$ have the form $(x, r, \alpha)$ where $\alpha\in A_r$ and $x\in P_{s_r} (Z^r_\alpha)$.  
Each simplicial complex $P_{s_r} (Z^r_\alpha)$ is equipped with the metric induced by the simplicial metric on $P_{s_r} (X)$ (see Section~\ref{modules} for a definition of the simplicial metric), and the distance between $(x,r,\alpha)$ and $(y, r', \beta)$ is set to infinity unless $r=r'$ and $\alpha = \beta$.  
In other words, we consider $P_{\bs} (\Z)$ to be a subset of the infinite disjoint union
\begin{equation}\label{psx} \coprodmo_{r=1}^\infty \coprod_{\alpha \in A_r} P_{s_r} (X),\end{equation}
with the induced metric.  
\end{definition}

\begin{remark} $\label{triv-seq-rmk}$
Given a decomposed sequence 
$\Z = (Z^1, Z^2, \ldots)$ with decompositions $Z^r = \bigcup_{\alpha\in A_r} Z^r_\alpha$  ($r\geqs 1)$,
we let $\X_\Z$ denote the decomposed sequence $\X_\Z = (X, X, \ldots)$ with decompositions 
 $X = \bigcup_{\alpha\in A_r} X$ for each $r\geqs 1$.
Then $P_\bs (\X_\Z)$ is the metric space (\ref{psx}).
\end{remark}

We will need to consider coverings of one decomposed sequence by two subsequences.  In the applications, the subsequences will have lower ``decomposition complexity" than the original sequence, in a sense that will be explained in Section~\ref{FDC}.    

\begin{definition} Let $\Z = (Z^1, Z^2, \ldots)$ be a decomposed sequence inside the metric space $X$, with decompositions $Z^r = \bigcup_{\alpha\in A_r} Z^r_\alpha$.  
We write $\Z = \U \cup \V$ if
$\U$ and $\V$ are decomposed sequences in $X$ whose decompositions are indexed over the same sets $A_r$ ($r\geqs 1$) and for each $r\geqs 1$ and each $\alpha\in A_r$ we have
$$Z^r_\alpha = U^r_\alpha \cup V^r_\alpha.$$

Similarly,  we write $\U\subset \Z$ if $\U$ is a decomposed sequence in $X$ with the same indexing sets as $\Z$, and for each $r\geqs 1$ and each $\alpha\in A_r$ we have
$$U^r_\alpha \subset Z^r_\alpha.$$

Given a sequence $\bs\in \Seq$, we define
$$\Ac^{\Z +} (P_\bs (\U)) := \Ac^{P_\bs (\Z) +} (P_\bs (\U))$$
as in Definition~\ref{controlled-mod}.  Note that both $P_\bs (\Z)$ and $P_\bs (\U)$ have the metric induced by the simplicial metric on $P_\bs (\X_\Z)$.  We will sometimes drop $\Z$ from the superscript when it is clear from context.
\end{definition}

In the proof of our vanishing result for continuously controlled $K$--theory (Theorem~\ref{vanishing-thm}), it will be important to ignore the initial portion of a decomposed sequence.  This is done via the following constructions.

\begin{definition}$\label{decomposed}$
Given proper metric spaces $Y_1, Y_2, \ldots$ and a subcategory
$$\A \subset \Ac \left(\coprod_{r = 1}^\infty Y_r\right),$$
we define $\mS = \mS(\A)$ to be the full subcategory of $\A$ consisting of those geometric modules supported on
$\coprod_{r = 1}^R Y_r \cross [0,1)$
for some $R>0$.  Note that 
$$\mS = \colim_{R>0} \left(\A \cap \Ac\left(\coprod_{r = 1}^R Y_r\right)\right).$$
We then define
$\overline{\A} = \A/\mS$.

Given decomposed sequences $\U\subset \Z$ in $X$, we set $\bAc (\Z) = \overline{\Ac (\Z)}$ and 
$\bAcp (\U) = \bAc^{\Z +} (\U) = \overline{\Ac^{\Z+} (\U)}$.
\end{definition}

\begin{remark}$\label{functoriality-rmk}$ The constructions $\bAc$ and $\bAcp$ enjoy the same sort of functoriality as $\Ac$ and $\Acp$.  The statements in Lemma~\ref{functoriality} regarding  functoriality of $\Ac$ and $\Acp$ for inclusions of simplicial complexes   apply to inclusions of Rips complexes associated to inclusions of decomposed sequences, and Lemma~\ref{induced-map} yields corresponding statements for $\bAc$ and $\bAcp$.  

In the sequel, we will simply refer to Lemma~\ref{functoriality} when constructing functors between categories $\bAc(-)$ and $\bAcp(-)$.
\end{remark}

\subsection{Mayer--Vietoris for Rips complexes}$\label{MVRips-sec}$

\begin{theorem} $\label{Rips-MV}$
Let $\Z$, $\U$, and $\V$ be decomposed sequences in a bounded geometry metric space $X$, with $\Z = \U \cup \V$, and choose $\bs\in \Seq$.
Then there is a long exact sequence in non-connective $K$--theory of the form
\begin{equation}\label{Rips-MV-seq}
\cdots \maps \K_*\left(\I_\bs (\U, \V) \right)
\xrightarrow{(i_1, i_2)}
\K_*\bAc^{\Z +} (P_\bs (\U)) \oplus \K_* \bAc^{\Z +} (P_\bs (\V))
\end{equation}
\begin{equation*}\xrightarrow{(j_1)_* - (j_2)_*} \K_* (\bAc  (P_\bs (\Z))) \stackrel{\partial}{\maps}
\K_{*-1} \left(\I_\bs (\U, \V) \right) \maps \cdots,
\end{equation*}
where $\I_\bs (\U, \V)$
denotes the intersection in $\bAc (P_\bs (\Z))$ of  $\bAc^{\Z +} (P_\bs (\U))$ and $\bAc^{\Z +} P_\bs (\V)$.  The maps 
 $i_1$ and $i_2$ are  induced by the relevant inclusions of categories and the maps
$j_1$ and $j_2$ are the functors associated to the inclusions of simplicial complexes $P_\bs (\U) \injects P_\bs (\Z)$ and $P_\bs (\V) \injects P_\bs (\Z)$.
\end{theorem}
\begin{proof}  By Proposition~\ref{MV} it suffices to check that
$$N_1 \left( P_\bs (\U) \right) \cup N_1 \left( P_\bs (\V)\right) = P_\bs (\Z).$$
Given a simplex $\sigma = \langle x_0, \ldots, x_n \rangle$ in $P_\bs (\Z)$, we either have $x_0\in U^r_\alpha$ for some $r\geqs 1$ and some $\alpha\in A_r$, or we have $x_0\in V^r_\alpha$ for some $r\geqs 1$ and some $\alpha\in A_r$.  In the former case, $\langle x_0 \rangle$ is a $0$--simplex in $P_\bs (\U)$, and $\sigma \subset N_1 \left(\langle x_0 \rangle\right) \subset N_1 \left(P_\bs (\U)\right)$.  In the latter case, 
$\sigma \subset N_1 \left(P_\bs (\V)\right)$.
\end{proof}

\subsection{Mayer--Vietoris for relative Rips complexes}$\label{rel-sec}$

We will need another Mayer--Vietoris sequence for the proof of our vanishing theorem (Theorem~\ref{vanishing-thm}), involving the relative Rips complexes introduced by Guentner--Tessera--Yu~\cite[Appendix A]{GTY-rigid}.   

\begin{definition}$\label{rel-Rips}$
Consider a bounded geometry metric space $X$, along with a subspace $Z\subset X$ and a family $\bbW$ of subspaces of $X$. Given $0<s<s'$, the \emph{relative Rips complex} $P_{s, s'} (Z, \bbW)$ is  the subcomplex of $P_{s'} (X)$ consisting of those simplices $\langle x_0, \ldots, x_n\rangle$ satisfying at least one of the following conditions:
\begin{enumerate}
\item $x_0, \ldots, x_n \in Z$ and $d(x_i, x_j)\leqs s$ for all $i, j$; 
\item $x_0, \ldots, x_n \in W$ for some $W\in \bbW$.  
\end{enumerate}
Note that in the second case, $d(x_i, x_j)\leqs s'$ for all $i,j$ since we are defining a subcomplex of $P_{s'} (X)$.
We equip $P_{s, s'} (Z, \bbW)$ with the metric induced by the simplicial metric on $P_{s, s'} (X, \bbW)$.    (It will be crucial for our arguments that we \e{do not} use the metric inherited from the simplicial metric on $P_{s'} (X)$; see in particular Lemmas~\ref{rel-nbhd-comp} and~\ref{rel-metric-comp}.)  Note that in this definition, we do not require that the subspaces $W\in \bbW$ satisfy $W\subset Z$.

Given a  decomposed sequence $\Z = (Z^1, Z^2, \ldots)$ in $X$ with decompositions $Z^r = \bigcup_{\alpha\in A_r} Z^r_\alpha$, a set  $\bbW = \{\bbW^r_\alpha: r\geqs 1, \alpha\in A_r\}$ of families of subspaces of $X$,
and $\bs, \bs' \in \Seq$ satisfying $\bs \leqs \bs'$,
we define the relative Rips complexes
$$P_{\bs, \bs'} (\Z, \bbW) := \coprod_{r=1}^\infty \coprod_{\alpha\in A_r} P_{s_r, s_r'} (Z^r_\alpha,  \bbW^r_\alpha)\hspace{1.5in}$$
$$\hspace{1.5in} \subset \coprod_{r=1}^\infty \coprod_{\alpha\in A_r} P_{s_r, s_r'} (X, \bbW^r_\alpha) =: P_{\bs, \bs'} (X, \bbW),
$$
and we give $P_{\bs, \bs'} (\Z, \bbW)$ the metric induced by the simplicial metric on 
$P_{\bs, \bs'} (X, \bbW)$.   
\end{definition}

Given a covering $\Z = \U\cup \V$ of a decomposed sequence by two subsequences, we will need to consider a relative Rips complex in which the ``larger" simplices are constrained to lie near both $\U$ and $\V$.

\begin{definition}$\label{W_t-def}$
Consider decomposed sequences  $\Z   = (Z^1, Z^2, \ldots)$, $\U =  (U^1, U^2, \ldots)$, and $\V = (V^1, V^2, \ldots)$ in a metric space $X$,
with decompositions $Z^r = \bigcup_{\alpha\in A_r} Z^r_\alpha$,  $U^r = \bigcup_{\alpha\in A_r} U^r_\alpha$, and  $V^r = \bigcup_{\alpha\in A_r} V^r_\alpha$.  Assume that
$\Z= \U \cup \V$, and say that we are given additional decompositions
\begin{equation} \label{addl-decomps} U^r_\alpha = \coprod_{i\in I(r, \alpha)} U^r_{\alpha i}  \, \textrm{ and } \, V^r_\beta = \coprod_{j\in J(r, \beta)}V^r_{\beta j}\end{equation}
for each $r\geqs 1$, and each $\alpha, \beta \in A_r$.
Given $T> 0$, $r\geqs 1$, and $\alpha\in A_r$, we define
$$\bbW^r_{T, \alpha} = \{N_{T} (U^r_{\alpha i}) \cap N_{T} (V^r_{\alpha j})\cap Z^r_\alpha : i\in I(r,\alpha),\, j\in J(r, \beta) \},$$
and given  $\bT\in \Seq$, we define the set of metric families $\bbW_\bT  (\U, \V, \Z)$ to be
$$\bbW_\bT  (\U, \V, \Z) =\bbW_\bT =  \{\bbW^r_{T_r, \alpha}: r\geqs 1, \alpha\in A_r\}.$$
(We will suppress the dependence of $\bbW_\bT  (\U, \V, \Z)$ on the chosen additional decompositions (\ref{addl-decomps}).)

For any $\bs, \bs'\in \Seq$, we can now form the relative Rips complexes
$$P_{\bs, \bs'} (\Z, \bbW_\bT),\,\, P_{\bs, \bs'} (\U, \bbW_\bT)\,\, {\textrm and }\,\,P_{\bs, \bs'} (\V, \bbW_\bT)$$
as in Definition~\ref{rel-Rips}.  Following Definition~\ref{decomposed}, we set
$$\bAc \left(P_{\bs, \bs'} (\Z, \bbW_\bT)\right) := \Ac \left(P_{\bs, \bs'} (\Z, \bbW_\bT)\right)/\mS.$$ 
We define
$\bAc^{\Z +} \left(P_{\bs, \bs'} (\U, \bbW_\bT)\right)$ and $\bAc^{\Z +} \left(P_{\bs, \bs'} (\V, \bbW_\bT)\right)$ similarly, by allowing modules supported on neighborhoods of $P_{\bs, \bs'} (\U, \bbW_\bT)$ (or, respectively, $P_{\bs, \bs'} (\V, \bbW_\bT)$) inside $P_{\bs, \bs'} (\Z, \bbW_\bT)$ (recall that these complexes are given the metrics  inherited from the simplicial metric on $P_{\bs, \bs'} (X, \bbW_\bT)$).
\end{definition}

\begin{theorem}$\label{rel-MV}$
Let $\Z$, $\U$, $\V$ and $X$ be as in Theorem~\ref{Rips-MV}, and say that for each $r\geqs 1$ and each $\alpha, \beta \in A_r$ we are given additional decompositions
$$U^r_\alpha = \coprod_{i\in I (r, \alpha)} U^r_{\alpha i} \, \textrm{ and } \,  U^r_\beta = \coprod_{j\in J(r, \beta)}V^r_{\beta j}.$$

Then for any sequences  $\bs, \bs', \bT\in \Seq$, we can form the sequence $\bbW_\bT = \bbW_\bT (\U, \V, \Z)$ as in Definition~\ref{W_t-def}, and
 there is a long exact Mayer--Vietoris sequence in non-connective $K$--theory of the form
\begin{eqnarray*}
\cdots \maps K_{*+1} \bAc P_{\bs, \bs'} (\Z,   \bbW_\bT) 
\stackrel{\partial}{\maps}  K_{*}  \I'_{\bs, \bs', \bT} (\U, \V)\hspace{.7in}\\
\maps   K_*  \bAc^{\Z +} P_{\bs, \bs'} (\U, \bbW_\bT) 
		\oplus
	   K_* \bAc^{\Z +} P_{\bs, \bs'} (\V, \bbW_\bT) \\
\maps  K_* \bAc  P_{\bs, \bs'} (\Z,   \bbW_\bT) 
\stackrel{\partial}{\maps} \cdots,\hspace{.87in}
\end{eqnarray*}
where $\I'_{\bs, \bs', \bT} (\U, \V)$ is the intersection of the subcategories $\bAc^{\Z +} P_{\bs, \bs'} (\U, \bbW_\bT)$ and 
$\bAc^{\Z +} P_{\bs, \bs'} (\V,  \bbW_\bT)$ inside $\bAc P_{\bs, \bs'} (\Z, \bbW_\bT)$.
\end{theorem}
\begin{proof}  We apply Proposition~\ref{MV}.  The conditions are checked just as in the proof of Theorem~\ref{Rips-MV}: for each $r\geqs 1$ and each $\alpha\in A_r$, each point in 
 $P_{s_r, s_r'} \left(Z^r_\alpha, \bbW^r_{T_r, \alpha} \right)$  
is within distance 1 of
$$P_{s_r, s_r'} \left(U^r_\alpha, \bbW^r_{T_r, \alpha} \right)
\cup P_{s_r, s_r'} \left(V^r_\alpha, \bbW^r_{T_r, \alpha}\right).$$
\end{proof}

\subsection{A comparison of Mayer--Vietoris sequences}

For the arguments in Section~\ref{FDC}, we will need to compare the absolute and relative Mayer--Vietoris sequences from Sections~\ref{MVRips-sec} and~\ref{rel-sec}.

\begin{theorem}$\label{MV-diagram}$
Let $\Z, \U, \V,$ and $X$ be as in Theorem~\ref{Rips-MV}.

Then for any $\bs, \bs', \bT\in \Seq$ there are functors 
$$\bAc^{\Z +} P_{\bs} (\U) \xmaps{i_\U} \bAc^{\Z +} P_{\bs, \bs'} (\U, \bbW_\bT),\,\,\,\,\, \bAc^{\Z +} P_{\bs} (\V) \xmaps{i_\V} \bAc^{\Z +} P_{\bs, \bs'} (\V, \bbW_\bT),$$
$$\bAc  P_{\bs} (\Z) \srt{\gamma} \bAc  P_{\bs, \bs'} (\Z, \bbW_\bT), \textrm{ and }\,\,\, \I_{\bs} (\U, \V) \srt{\rho} \I'_{\bs, \bs', \bT} (\U, \V)$$
such that the diagram of Mayer--Vietoris sequences
\begin{equation}\label{MV-commutes}
\xymatrix{ 
K_{*} \left(\I_{\bs} (\U, \V) \right)\ar[d]
\ar[rr]^-{\rho_*}
	& & K_{*} \left(\I'_{\bs, \bs', \bT} (\U, \V) \right)\ar[d]\\
K^+_* \left( P_\bs \U  \right) \oplus K^+_* \left( P_\bs \V  \right) \ar[d] \ar[rr]^-{(i_\U)_* \oplus (i_\V)_*}
	& & K^+_*  ( P_{\bs, \bs'} (\U, \bbW_\bT)  )
		 \oplus K^+_* ( P_{\bs, \bs'} (\V, \bbW_\bT) )   \ar[d]  \\
K_* \left( P_{\bs} (\Z)  \right) \ar[rr]^-{\gamma_*} \ar[d]^-{\partial} 
	& &K_*  \left(P_{\bs, \bs'} (\Z, \bbW_\bT) \right) \ar[d]^-{\partial}  \\
K_{*-1} \left(\I_{\bs} (\U, \V) \right)
\ar[rr]^-{\rho_*}
	&   &K_{*-1} \left(\I'_{\bs, \bs', t} (\U, \V) \right).\\
}
\end{equation}
is commutative; note that in the middle two rows we have used $K_*$ and $K^+_*$ as  shorthand for $K_* (\bAc (-))$ and $K_* (\bAc^{\Z +} (-))$ (respectively).
\end{theorem}
\begin{proof}  The Mayer--Vietoris sequences are produced by Theorem~\ref{Rips-MV} and Theorem~\ref{rel-MV}.  The functors $i_\U$, $i_\V$, and $\gamma$ are induced by the relevant inclusions of simplicial complexes, which satisfy the hypotheses of Lemma~\ref{functoriality}.  There is then an induced functor 
$\I_{\bs} (\U, \V) \srt{\rho} \I'_{\bs, \bs', \bT}$
between the intersection categories.  By Lemma~\ref{induced-map}, these functors produce a commutative diagram $\D$ of categories consisting of two diagrams of the form (\ref{MV-square}), with one mapping to the other.
After taking $K$--theory spectra, we obtain a morphism between two homotopy (co-)cartesian squares of spectra, together with maps between the homotopy cofibers of the vertical maps in these squares. It is a general fact that maps between homotopy (co-)cartesian squares of spectra yield commutative diagrams of Mayer--Vietoris sequences.
\end{proof}


\section{Metric properties of Rips complexes}$\label{Rips-sec}$

In this section we record some basic geometric results about Rips complexes, some of which may be found in Guentner--Tessera--Yu~\cite[Appendix A]{GTY-rigid}.  For completeness, we provide detailed proofs.

\begin{definition}$\label{C}$ If $X$ is a bounded geometry metric space and $s$ is a positive real number, we let $C(s, X) = (2\sqrt{2} + 1)^{N-1}$, where $N$ is the dimension of $P_s (X)$ (if $P_s (X)$ is zero-dimensional, we set $C(s, X) = 1$).
\end{definition}

\begin{lemma}$\label{metric-comp}$ Let $(X,d)$ be a metric space with bounded geometry and let $d_\Delta$ denote the simplicial metric on $P_s (X)$.  Then for all $x, y\in X \subset P_s (X)$,
$$d (x,y) \leqs s C(s,X) \sd(x,y).$$
\end{lemma}

\begin{proof} Given a sequence $\gamma = (p_0, p_1, \ldots, p_k)$  of points in $P_s (X)$, let $l(\gamma) = \sum_{i=0}^{k-1} d_\Delta (p_i, p_{i+1})$. 
By definition of the simplicial metric, we must show that 
$$d(x,y) \leqs s (2\sqrt{2} + 1)^{N-1}  \, l(\gamma)$$
for all sequences $\gamma = (p_0, p_1, \ldots, p_k)$ such that $p_0 = x$, $p_k = y$, and for $i = 1, \ldots k$, $p_i$ and $p_{i-1}$ lie in a common simplex $\sigma_i$ (which we may assume is the smallest simplex containing  $p_i$ and $p_{i-1}$).
Let $\dim(\gamma) = \max_i \dim (\sigma_i)$, and note that $\dim(\gamma)\leqs N$.
We will show by induction on $\dim(\gamma)$ that
$d(x,y) \leqs s (2\sqrt{2} + 1)^{\dim(\gamma)-1} l(\gamma)$.  Note that if $\sigma_i \subset \sigma_{i+1}$ or  $\sigma_{i+1} \subset \sigma_{i}$, then we may shorten $\gamma$ by removing $p_{i}$, so we may assume without loss of generality that $\sigma_i\cap \sigma_{i+1}$ is a proper face of both $\sigma_i$ and $\sigma_{i+1}$.  This implies that $p_i$ and $p_{i+1}$ lie in the boundary of $\sigma_{i+1}$ for $i = 0, \ldots, k-1$.

If $\dim(\gamma)=1$, then $p_i\in X$ for each $i$, and we have
$$d(x,y) \leqs \sum_{i=0}^{k-1} d(p_i, p_{i+1}) \leqs sk = s (2\sqrt{2} + 1)^0  \, l(\gamma).$$
Now assume the result for paths of dimension at most $n-1$, and say $\dim(\gamma) = n$.  We will replace $\gamma$ by a nearby path of lower dimension.  By assumption, there exists $i\in \{0, \ldots, k-1\}$ such that $\sigma_{i+1} = \langle x_0, \ldots, x_n \rangle$ for some $x_0,\ldots, x_n\in X$.  Reordering the $x_j$ if necessary, we may further assume that $p_i \in \langle x_0, \ldots, x_{n-1} \rangle$ and $p_{i+1} \in \langle x_1, \ldots, x_{n} \rangle$.   Letting $\overline{p_i}$ and $\overline{p_{i+1}}$ denote the orthogonal projections of these points to the affine $(n-2)$--plane containing $\langle x_1, \ldots, x_{n-1} \rangle$ (note that these orthogonal projections necessarily lie inside $\langle x_1, \ldots, x_{n-1} \rangle$), we will replace $\gamma$ by the piecewise geodesic path $\gamma' = (p_0, \ldots, p_i, \overline{p_i}, \overline{p_{i+1}}, p_{i+1}, \ldots, p_k)$.   

We claim that $\sd(p_i, \overline{p_i})$  and   
$\sd(\overline{p_{i+1}}, p_{i+1})$ are  at most $\sqrt{2} \sd(p_i, p_{i+1})$.
In barycentric coordinates, we may write  $p_i = \sum_{i=0}^n a_i x_i$ 
(with $a_n = 0$)
and $p_{i+1} =  \sum_{i=0}^n b_i x_i,$ 
(with $b_0 = 0$).
Setting $w = (a_0 + a_1) x_1 +  \sum_{i=2}^{n-1} a_i x_i$
we have $\sd (p_i, w) = \sqrt{2}a_0$ and
$\sd(p_i, p_{i+1}) \geqs a_0$,
so  $\sd(p_i, w)   \leqs \sqrt{2} \sd (p_i,p_{i+1})$.
Hence 
$$\sd(p_i, \overline{p_i}) \leqs \sd(p_i, w) \leqs  \sqrt{2}  \sd (p_i,p_{i+1}),$$ 
as desired.  Similarly,
$\sd(\overline{p_{i+1}}, p_{i+1}) \leqs \sqrt{2}\sd (p_i,p_{i+1})$.  

Since orthogonal projections decrease distances, we also have
$$\sd(\overline{p_i}, \overline{p_{i+1}}) \leqs \sd(p_i, p_{i+1})\leqs 1,$$
and hence   $l(p_i, \overline{p_i}, \overline{p_{i+1}}, p_{i+1})\leqs (2\sqrt{2} + 1) \sd(p_i, p_{i+1})$.  Repeating this procedure for each $n$--simplex among the $\sigma_i$, we obtain a new path $\gamma'$ (from $x$ to $y$) which lies entirely in the $(n-1)$--skeleton of $P_{s} (X)$ (meaning that $\dim(\gamma')\leqs n-1$) and satisfies 
$l(\gamma')\leqs (2\sqrt{2} + 1) l(\gamma)$.  By induction, we know that $d(x,y)\leqs (2\sqrt{2} + 1)^{n-2} l(\gamma')$, so  
$d(x,y) \leqs (2\sqrt{2} + 1)^{n-1} l(\gamma)$, completing the proof.
\end{proof}

It is important to note that no bound exists in the opposite direction: if $d(x,y) > s$, then $x$ and $y$ may lie in different connected components of $P_s (X)$, in which case $\sd (x,y) = \infty$.

\vspace{.2in}
The following result will allow us to compare distances in  relative Rips complexes.  For this result to hold, it is crucial that we give the relative Rips complex $P_{s, s'} (Z, W)$ the metric inherited from the simplicial metric on $P_{s, s'} (X, W)$ rather than $P_{s'} (X)$.

Note that each point $x$ in a simplicial complex $K$ can be written uniquely, in barycentric coordinates, in the form $x = \sum c_{v_i} (x) v_i$ with $c_{v_i} (x) > 0$ for each $i$.  We will refer to the vertices $v_i$ as the \e{barycentric vertices} of $x$.
Given a vertex $v\in K$, we can extend $c_v$ to a continuous function from $K$ to $[0,1]$ by setting $c_v (x) = 0$ if $v$ is not a barycentric vertex of $x$.

\begin{lemma} $\label{rel-nbhd-comp}$ Let $W\subset X$ be metric spaces, and assume $X$ has bounded geometry.  Given  $s'\geqs s>0$, let $N_t (P_{s'} (W))$ denote a $t$--neighborhood of $P_{s'} (W)$ inside $P_{s, s'} (X, W)$.
Then for all $x\in X\cap N_t (P_{s'} (W))$ (where $X$ is viewed as the $0$--skeleton of $P_{s, s'} (X, W)$), we have 
\begin{equation}\label{dist-est}d (x, W) \leqs (t+1) C(s, X) s.\end{equation}

It follows that inside the simplicial complex $P_{s'} (X)$, we have  inclusions
\begin{equation}\label{rnc1}N_t (P_{s'} (W)) \subset   P_{s, s'} (N_{(t+2) C(s, X) s} (W), W) 
\subset P_{s'} (N_{(t+2) C(s, X) s} (W)),\end{equation}
where on the left, the neighborhood is still taken with respect to the simplicial metric on $P_{s, s'} (X, W)$.
Additionally, for any $U\subset X$, we have inclusions
\begin{equation}\label{rnc2}N_t (P_{s, s'} (U, W)) \subset 
 N_t (P_{s'} (U \cup W)) \subset P_{s'} (N_{(t+2) C(s, X) s} (U\cup W)),\end{equation}
 where the first  neighborhood is taken inside $P_{s, s'} (X, W)$ and the second is taken inside $P_{s, s'} (X, U\cup W)$.
\end{lemma}
\begin{proof}  Say $x\in X\cap N_t (P_{s'} (W))$.  Then there exists a piecewise geodesic path $\gamma$ in $P_{s, s'} (X, W)$, starting at  $x$ and ending at a point in $P_{s'} (W)$, such that $l(\gamma) < t$, where $l( \gamma)$ is the sum of the lengths of the geodesics making up $\gamma$.  

Since $X$ has bounded geometry, the path $\gamma \co [0,1]\to P_{s, s'} (X, W)$ meets only finitely many (closed) simplices $\sigma_1, \ldots, \sigma_m$.  Let $J\subset \{1, \ldots, m\}$ be the subset of those $j$ such that $\sigma_j$ has a vertex lying in $W$; note that $J\neq \emptyset$ since $\gamma$ ends in $P_{s'} (W)$.  Let $r\in [0,1]$ be the minimum element of the compact set $\bigcup_{j\in J} \gamma^{-1} (\sigma_j)$.  If $r=0$, then $d(x, W)\leqs s$ and we are done, so we assume $r>0$.
For $r'< r$, the barycentric vertices of $\gamma(r')$ all lie in $X\setminus W$, so $\gamma(r') \in P_s (X)$.
Continuity of the barycentric coordinate functions implies that the barycentric vertices $x_0, \ldots, x_n$ of $\gamma(r)$ all lie in 
$X\setminus W$ as well. 
By choice of $r$, we know that $\gamma(r)$ lies in a simplex $\sigma$ having a vertex $w\in W$.  This simplex must contain $\langle x_0, \ldots, x_n\rangle$, and since $x_i \notin W$ we conclude (from the definition of the relative Rips complex) $\sigma \subset P_s (X)$.
Concatenating $\gamma|_{[0,r]}$ with a geodesic in $\sigma$ connecting $\gamma(r)$ and $w$ yields a piecewise geodesic path, inside $P_s (X)$, of length at most $t+1$.  Hence the simplicial distance, in $P_s (X)$, from $x$ to $w$ is at most $t+1$, and Lemma~\ref{metric-comp} tells us that $d(x, w)\leqs (t+1) C(s, X) s$.   This proves (\ref{dist-est}).

The first containment in (\ref{rnc1}) follows from the distance estimate (\ref{dist-est}), since if $z \in N_t (P_{s'} (W))$ lies in a simplex $\langle x_0, \ldots, x_n\rangle \subset P_{s, s'} (X, W)$, then for each $i$, the simplicial distance (in $P_{s, s'} (X, W))$ from $x_i$ to $P_{s'} (W)$ is at most $t+1$, so (\ref{dist-est}) shows that $x_i \in N_{(t+2) C(s, X) s} (W)$.   The second containment in (\ref{rnc1}) is immediate from the definitions.

The first containment in (\ref{rnc2}) follows from the fact that the simplicial metric on $P_{s, s'} (X, U\cup W)$ is smaller than the simplicial metric on the subcomplex $P_{s, s'} (X, W)$, while the second follows from (\ref{rnc1}), with $U\cup W$ playing the role of $W$.
\end{proof}

The following result, which generalizes (\ref{dist-est}), will be used in the proof of Lemma~\ref{rel-refinement}.

\begin{lemma}$\label{rel-metric-comp}$ Let $X$ be a bounded geometry metric space, with subspaces $X_1, X_2 \subset X$, and let $\bbW_1$ and $\bbW_2$ be families of subspaces of $X$.  
For $i=1,2$, let 
$$W_i = \bigcup \bbW_i = \{x\in X : x\in W  \textrm{ for some } W\in \bbW_i\}$$ 
denote the union of the subspaces in $\bbW_i$.
Set $\bbW = \bbW_1 \cup \bbW_2$, and let $d_\Delta$ denote the simplicial metric on $P_{s, s'} (X, \bbW)$ for some fixed $s, s' > 0$.  Setting $V_i = X_i \cup  W_i$ and $P_i = P_{s, s'} (X_i, \bbW_i)$ ($i=1,2$),
we have
\begin{equation}\label{rel-dist-est} d (V_1, V_2) \leqs (d_\Delta (P_1, P_2)+2) s C(s, X).\end{equation}
\end{lemma}

\begin{proof}

Consider a piecewise geodesic path $\gamma\co [0,1] \to P_{s, s'} (X, \bbW)$ with $\gamma(0) \in P_1$ and $\gamma(1)\in P_2$.  It will suffice to show that 
$$d (V_1, V_2) \leqs (l(\gamma)+2) s C(s, X).$$
Arguing as in the proof of Lemma~\ref{rel-nbhd-comp}, 
let $t_1 \in [0,1]$ denote the maximum time at which $\gamma(t)$ lies in a simplex with a vertex in $V_1$, and let $t_2 \in [t_1, 1]$ denote the minimum time (in the interval $[t_1, 1]$) at which $\gamma(t)$  lies in a simplex with a vertex in $V_2$.

If $t_1 = t_2$, then there is a simplex in $P_{s, s'} (X, \bbW)$ containing vertices from both $V_1$ and $V_2$.  If this simplex lies outside $P_s (X)$, then its vertices must lie entirely inside some set in $\bbW = \bbW_1 \cup \bbW_2$, and we find that $V_1 \cap V_2\neq \emptyset$.  If this simplex lies in $P_s (X)$, then we have $d(V_1, V_2) \leqs s$.  In either case, (\ref{rel-dist-est}) is trivially satisfied.

We now assume that $t_1 < t_2$.  Then for $t\in (t_1, t_2)$, if $\sigma\subset P_{s, s'} (X, \bbW)$ is a simplex containing $\gamma(t)$, then $\sigma$ has no vertex in $V_1\cup V_2$, and in particular no vertex in $W_1 \cup W_2$.  Thus $\gamma(t_1, t_2) \subset P_s (X)$.    Furthermore, by considering barycentric coordinates as in the proof of Lemma~\ref{rel-nbhd-comp}, one may check that for $i=1, 2$, $\gamma (t_i)$ lies in  a simplex $\sigma_i \subset P_s (X)$ such that at least one vertex $v_i\in \sigma_i$ satisfies  $v_i \in V_i$.  
Concatenating $\gamma|_{[t_1, t_2]}$ with geodesic paths inside $\sigma_i$ from $\gamma(t_i)$ to $v_i$, we obtain a path inside $P_s (X)$, of simplicial length at most $l(\gamma) +2$, connecting $V_1$ and $V_2$.  The result now follows from Lemma~\ref{metric-comp}.
\end{proof}


\section{Controlled $K$--theory for spaces of finite decomposition complexity}$\label{FDC}$

We now apply the results of Sections~\ref{MV-sec} and~\ref{Rips-sec} to  the continuously controlled  $K$--theory of spaces with \e{finite decomposition complexity}.  

We begin by reviewing some definitions from Guentner--Tessera--Yu~\cite{GTY-FDC, GTY-rigid}, where the notion of decomposition complexity was first introduced.   A set of metric spaces will be called a \e{metric family}.  Let $\fB$ denote the class of uniformly bounded metric families; that is, a family $\mF$  lies in $\fB$ if there exists $R>0$ such that  $\diam(F) < R$ for all $F\in \mF$.  Given a class $\fD$ of metric families, we say that a metric family $\mF = \{F_\alpha\}_{\alpha \in A}$ \emph{decomposes} over $\fD$ if for every $r>0$ and every $\alpha \in A$ there exists a decomposition $F_\alpha = U^r_\alpha \cup V^r_\alpha$ and $r$--disjoint decompositions
$$U^r_\alpha = \coprod^{r\textrm{--disjoint}}_{i\in I (r, \alpha)} U^r_{\alpha i} \, \textrm{ and } \, V^r_\alpha =   \coprod^{r\textrm{--disjoint}}_{j\in J (r, \alpha)} V^r_{\alpha j}$$
such that the families 
$$\{U^r_{\alpha i} : \alpha\in A,\,  i\in I(r, \alpha)\}\,  \textrm{ and }\, \{V^r_{\alpha j}: \alpha\in A,  j\in J(r, \alpha)\}$$ 
lie in $\fD$.\footnote{In the original definition in~\cite[Section 2]{GTY-rigid}, one assumes instead that there exists a family $\mF'\in \fD$ such that 
$\{U^r_{\alpha i}: \alpha \in A, \,i\in I(r, \alpha)\}\cup\{V^r_{\alpha j}: \alpha \in A, \,j\in J(r, \alpha)\} \subset \mF'$. 
However, since the collections of families $\D_\gamma$ defined here, and the analogous families defined in~\cite{GTY-rigid}, are closed under forming finite unions of families and under subfamilies, the two definitions of $\D_\gamma$ agree.  (With our definition of $D_\gamma$, closure under finite unions is checked by transfinite induction; closure under subfamilies follows, for example, from Lemma~\ref{nbhds}.)}
Here $r$--disjoint simply means that if $i_1, i_2\in I (r, \alpha)$ for some $\alpha\in A$, and  $i_1 \neq i_2$, then $d(U^r_{\alpha i_1}, U^r_{\alpha i_2}) > r$ (and similarly for $V$ in place of $U$).
We set $\fD_0 = \fB$, and given a successor ordinal $\gamma+1$ we define
$\fD_{\gamma+1}$ to be the class of all metric spaces which decompose over $\fD_\gamma$.
If $\gamma$ is a limit ordinal, we define 
$$\fD_{\gamma} = \bigcup_{\beta < \gamma} \fD_{\beta}.$$
(This definition will make the limit ordinal cases of all our transfinite induction arguments trivial.)

\begin{definition}$\label{fdc}$
We say that a metric space $X$ has finite decomposition complexity if the single-element family $\{X\}$ lies in $\fD_\gamma$ for some ordinal $\gamma$.  (We often write $X\in \fD_\gamma$ rather than $\{X\}\in \fD_\gamma$.)
\end{definition}

\begin{remark} If $X\in \fD_\gamma$ for some ordinal $\gamma$, then in fact there exists a \emph{countable} ordinal $\gamma'$ such that $X\in \fD_{\gamma'}$.  This is proven in Guentner--Tessera--Yu~\cite[Theorem 2.2.2]{GTY-FDC}.
\end{remark}

Given a metric space $X$, we use the term \e{metric family in $X$} to mean a metric family $\mF$ such that each $F\in \mF$ is a subspace of $X$ (with the induced metric).

\begin{lemma}$\label{nbhds}$ Let $X$ be a metric space, and let $\{Z_\alpha\}_{\alpha\in A}$ and $\{Y_\beta\}_{\beta \in B}$ be metric families in  $X$.  Say $\{Z_\alpha\}_{\alpha\in A} \in \fD_\gamma$ for some ordinal $\gamma$.  Assume further that there exists $t>0$ such that for all $\beta\in B$, there exists $\alpha\in A$ with $Y_\beta \subset N_t (Z_\alpha)$.  Then $\{Y_\beta\}_{\beta\in B} \in \fD_\gamma$ as well.  (Note here that the parameter $t$ is \e{independent} of $\beta\in B$.)
\end{lemma} 
\begin{proof}  We use transfinite induction.  In the base case, we have a uniform bound $D$ on the diameter of the $Z_\alpha$, and $D+t$ gives a uniform bound on the diameter of the $Y_\beta$, so $\{Y_\beta\}_{\beta } \in \fD_0$.  Now say 
$\gamma = \delta + 1$ is a successor ordinal, and assume the result for $\fD_\delta$.   If $\{Z_\alpha\}_\alpha \in \fD_\gamma$, then for each $r>0$ and each $\alpha\in A$ there exist $U^r_\alpha$ and $V^r_\alpha$ such that $Z_\alpha = U^r_\alpha \cup V^r_\alpha$, and there exist decompositions
$$U^r_\alpha = \coprodmo^{r\textrm{--disjoint}}_{i\in I(r, \alpha)} U^r_{\alpha i}, \hspace{.4in} V^r_\alpha = \coprodmo^{r\textrm{--disjoint}}_{j\in J(r, \alpha)} V^r_{\alpha j}$$
such that the families
$$\{U^r_{\alpha i} : \alpha\in A,\, i\in I(r, \alpha)\} \,\textrm{ and }\, \{V^r_{\alpha j}: \alpha\in A,  j\in J(r, \alpha)\}$$
lie in $\fD_\delta$.   For each $\beta\in B$, we know there exists $\alpha = \alpha(\beta) \in A$ such that $Y_\beta\subset N_t (Z_\alpha)$.  We now have decompositions
$$Y_\beta = \left(N_t (U^r_{\alpha}) \cap Y_\beta\right) \cup \left(N_t (V^r_{\alpha}) \cap Y_\beta\right),$$
and $(r-2t)$--disjoint decompositions
$$N_t (U^r_{\alpha}) \cap Y_\beta = \coprodmo_{i\in I(r, \beta)} N_t (U^r_{\alpha i}) \cap Y_\beta$$
and
$$N_t (V^r_{\alpha}) \cap Y_\beta = \coprodmo_{j\in J(r, \beta)} N_t (V^r_{\alpha j}) \cap Y_\beta.$$
By induction we know that 
the families 
$$\{N_t (U^r_{\beta i}) \cap Y_\beta : \beta\in B,\, i\in I(r, \beta)\} \,\,\, \textrm{  and  }\,\,\,  \{N_t (V^r_{ \beta j}) \cap Y_\beta\ : \beta\in B,\, j\in J(r, \beta)\}$$
 lie in $\fD_\delta$.   Since $r-2t$ tends to infinity with $r$, we see that the family $\{Y_\beta\}_\beta$ decomposes over $\fD_\delta$, as desired.  The case of limit ordinals is trivial.
\end{proof}
 
We now come to the main result of this section.

\begin{theorem}$\label{vanishing-thm}$
If $X$ is a bounded geometry metric space with finite decomposition complexity, then for each $*\in \bbZ$ we have
$$\mocolim_{s\to \infty} K_* \left( \Ac (P_s X) \right) = 0,$$
where the colimit is taken with respect to the maps 
$$K_*\left(\Ac (P_s X)\right) \xmaps{\eta_{s,s'}} K_*\left(\Ac (P_{s'} X)\right)$$
induced by applying Lemma~\ref{functoriality} to the inclusions $P_s X \injects P_{s'} X$.
\end{theorem}

We will deduce Theorem~\ref{vanishing-thm} from a closely related vanishing result for the constant and trivially
 decomposed sequence 
\begin{equation} \label{triv-seq} \X = (X, X, X, \ldots),\end{equation}
 where at each level  $X$ is decomposed into the one-element family $\{X\}$.  
 
\begin{definition} Let $\Z$ be a decomposed sequence in $X$.
For each $\bs \leqs \bs'\in \Seq$, we define 
\begin{equation}\label{vanishing}\eta_{\bs, \bs'} = \eta_{\bs, \bs'} (\Z)  \co K_* \left(\bAc (P_\bs (\Z))\right) \maps K_* \left(\bAc (P_{\bs'} (\Z))\right)\end{equation}
to be the map  induced by the inclusion  $P_\bs (\Z) \subset P_{\bs'} (\Z)$.
 \end{definition}
 
\begin{proposition}$\label{vanishing-prop2}$
If $X$ is a bounded geometry metric space with finite decomposition complexity, then for each  
$\bs\in \Seq$ and each element $x\in K_*(\bAc (P_\bs (\X)))$
there exists $\bs'\in \Seq$, with $\bs'\geqs \bs$, such that  $\eta_{\bs, \bs'} (x) = 0$.

\end{proposition}

We will see  in the proof  that $\bs'$ may depend on $x$.

\begin{remark} $\label{colim-rmk}$ Note that since $K$--theory commutes with directed colimits of additive categories (see Quillen~\cite[Section 2]{Quillen}), Proposition~\ref{vanishing-prop2} is equivalent to the statement that 
\begin{equation*}\mocolim_{\bs \in \Seq} K_* \Big(\bAc (P_\bs \X)\Big) = 0.\end{equation*}
\end{remark}

\vspace{.2in}
\noindent {\bf Proof of Theorem~\ref{vanishing-thm} assuming Proposition~\ref{vanishing-prop2}.}
We apply Proposition~\ref{vanishing-prop2} with $\bs = (s, s, \ldots)$.  Given $\bs'\geqs \bs$ and $m\geqs 1$, consider the diagram
\begin{equation}\label{prod-diag}
\xymatrix{ 	&	&	\mocolim_n \Ac \left(\coprod\limits_{r=1}^n  P_{s'_r} (X)\right) \ar[d]^{\mocolim j_n}  \\
	\Ac (P_s X) \ar[r]^\mu & \Ac (P_\bs \X) \ar[r]^{i} \ar[d]^{\pi_\bs} & \Ac (P_{\bs'} \X) \ar[r]^{q_{m}} \ar[d]^{\pi_{\bs'}}  
				&    \Ac P_{{s'_m}} (X)	\\
				&		\bAc (P_\bs \X) \ar[r]^{\overline{i}} & \bAc (P_{\bs'} \X).
}
\end{equation}
Here the maps $i$, $\bar{i}$, and $j_n$ are induced by inclusions of simplicial complexes, $\pi_\bs$ and 
$\pi_{\bs'}$ are the Karoubi projections,
 the functor $\mu$ sends a geometric module $M$ on $P_s (X)\cross [0,1)$ to the constant sequence $(M, M, \ldots)$ (and similarly for morphisms), and $q_{m}$ is the functor which restricts a geometric module to the subspace  $P_{s'_m} (X)\cross [0,1)\subset \coprod_{r=1}^\infty P_{s'_r} (X)\cross [0,1)$.  

Let $x\in K_* \Ac (P_s X)$ be given.
 In $K$--theory, $\bar{i}_*$ is the map (\ref{vanishing}), so Proposition~\ref{vanishing-prop2} implies that we can choose $\bs'\geqs \bs$ such that $\bar{i}_*\left(\pi_\bs\circ \mu (x)\right) = 0$.  For $m>n$ the composite $q_{m} \circ j_n$ is the constant functor mapping all objects to $0$, so $(q_{m})_* (j_n)_* = 0$ in $K$--theory.  However, for any $m$, the composite $q_{m} \circ i \circ\mu$ is simply the functor induced by the inclusion $P_s (X) \injects P_{{s'_m}} (X)$, so $(q_{m} \circ i \circ\mu)_* = \eta_{s, s_m'}$.  Since the third column of Diagram (\ref{prod-diag})
 is a Karoubi sequence, chasing the diagram and applying Remark~\ref{colim-rmk} shows that for some $N\geqs 0$ and some $y\in K_* \Ac \left(\coprod_{r=1}^N  P_{{s'_r}} (X)\right)$, we have $i_* \mu_* (x) = (j_N)_* (y)$, so 
$$\eta_{s, s'} (x) = (q_{N+1})_* \circ i_* \circ\mu_* (x)
= (q_{N+1})_* (j_N)_* (y) = 0.$$
The result now follows, since the colimit in Theorem~\ref{vanishing-thm} is defined in terms of the maps $\eta_{s, s'}$. 
$\hfill \Box$

\vspace{.2in}
To prove the desired vanishing result for the map (\ref{vanishing}), we will proceed through an induction for decomposed sequences inside $X$. 

\begin{definition}$\label{vanishing-def}$ Let  $\Z = (Z^1, Z^1, \ldots)$ be a decomposed sequence in $X$ with decompositions 
$Z^r = \bigcup_{\alpha\in A_r} Z^r_\alpha$. 
We say that $\Z$ is a \emph{vanishing sequence} (or more briefly, $\Z$ is vanishing) if for each $\bs\in \Seq$ and each $x\in K_* (\bAc (P_\bs \Z))$, there exists $\bs'\geqs \bs$  such that $x$ maps to zero under
\begin{equation*}  K_* (\bAc (P_{ \bs} \Z)) \xmaps{\eta_{\bs, \bs'}} K_* (\bAc (P_{ \bs'}  \Z)).\end{equation*}
For each ordinal $\gamma$, let $\fD_\gamma (X)$ denote the set of $\mF \in \fD_\gamma$ such that $\mF$ is a metric family in $X$.
By abuse of notation we write $\Z\in \fD_\gamma (X)$ if 
$\{Z^r_\alpha\}_{\alpha\in A_r} \in \fD_\gamma (X)$ for each $r\geqs 1$.
We say that $\fD_\gamma (X)$ is vanishing if all decomposed sequences $\Z\in \fD_\gamma (X)$ are vanishing.  

Finally, given a sequence $\bs\in \Seq$, we say that $\Z$ is \e{vanishing at} $\bs$ if for each  $x\in K_* (\bAc (P_\bs \Z))$, there exists $\bs'\geqs \bs$  such that $\eta_{\bs, \bs'} (x) = 0.$
\end{definition}

\begin{definition} Given a sequence $\bT\in \Seq$ and a decomposed sequence $\Z = (Z^1, Z^2, \ldots)$ in $X$ with decompositions
$Z^r = \bigcup_{\alpha\in A_r} Z^r_\alpha$, we define
$N_\bT (\Z)$ to be the decomposed sequence 
$(N_{T_1} (Z^1), N_{T_2} (Z^2), \ldots)$, with decompositions
$N_{T_r} (Z^r) = \bigcup_{\alpha\in A_r} N_{T_r} (Z^r_\alpha)$.
\end{definition}

The next lemma is an immediate consequence of Lemma~\ref{nbhds}.

\begin{lemma}$\label{nbhds2}$
Let $\Z$ be a decomposed sequence in $X$, and say $\Z\in \fD_\gamma (X)$ for some ordinal $\gamma$.
If $\Y$ is another decomposed sequence in $X$, and $\Y\subset N_\bT (\Z)$ for some sequence $\bT$ of positive real numbers, then $\Y\in \fD_\gamma (X)$ as well.
\end{lemma}

Note that a metric space $X$ has finite decomposition complexity if and only if the constant and trivially  decomposed sequence $\X = (X, X, \ldots)$ (see (\ref{triv-seq})) lies in $\fD_\gamma (X)$ for some ordinal $\gamma$, so Proposition~\ref{vanishing-prop2} is an immediate consequence of the next result.

\begin{proposition}$\label{vanishing-prop4}$
If $X$ is a bounded geometry metric space,
then $\fD_\gamma (X)$ is vanishing for every ordinal $\gamma$.  
\end{proposition}

The proof of Proposition~\ref{vanishing-prop4} will be by transfinite induction on the ordinal $\gamma$, and will fill the remainder of the section.  

For the rest of the section, we fix a bounded geometry metric space $X$. 
We first consider the base case of our induction, $\Z \in \D_0(X)$.
This means $\Z$ is a decomposed sequence in $X$ for which each family $\{Z^r_\alpha\}_{\alpha\in A_r}$ is uniformly bounded.  Hence for each $r\geqs 1$, there exists $N(r)$ such that for all $\alpha\in A_r$, the diameter of $Z^r_\alpha$ is at most $N(r)$.  This means that if $\bs' \geqs \bN := (N(1), N(2), \ldots)$, the simplicial complex
$$P_{\bs'} (\Z) = \coprod_{r=1}^{\infty} \coprod_{\alpha\in A_r} P_{s'_r} (Z^r_\alpha)$$
is a disjoint union of simplices, one for each pair $r\geqs 1$, $\alpha\in A_r$.  The following lemma will now establish the base case of our induction.

\begin{lemma} Say $\Z = (Z^1, Z^2, \ldots)$ is a decomposed sequence in $X$ with decompositions $Z^r = \bigcup_{\alpha\in A_r} Z^r_\alpha$.  Assume that there exists a sequence $\bN = (N_1, N_2, \ldots)\in\Seq$, such that for all $r\geqs 1$ and for all $\alpha\in A_r$, the diameter of $Z^r_\alpha$ is at most $N_r$.  

Then if $\bs\geqs \bN$, we have
$K_* \bAc (P_\bs (\Z)) = 0$ for all $*\in \bbZ$.  
\end{lemma}
\begin{proof}  We have already observed that for $\bs\geqs \bN$, $P_{\bs} (\Z)$ is a disjoint union of simplices, all at infinite distance from one another.  We claim that the controlled $K$--theory of such a metric space vanishes.  Let $W = \coprod_{i\in I} W_i$ be such a metric space, meaning that for each $i$ we have $W_i\isom \Delta^{k_i}$ for some $k_i$ and the distance between $W_i$ and $W_j$ is infinite if $i\neq j$.  Choose inclusions $\{*\}\injects W_i$, where $\{*\}$ denotes the one-point space, and let $j$ denote the resulting map $\coprod_{i\in I} \{*\} \injects W$.  Also, let $\pi$ denote the natural projection $W\to \coprod_{i\in I} \{*\}$.  By Bartels~\cite[Corollary 3.19]{Bartels}, the induced maps $\pi_* \co K_* \Ac (W) \to K_* \Ac (\coprod_{i\in I} \{*\})$ and $j_*\co K_* \Ac (\coprod_{i\in I} \{*\})\to K_* \Ac (W)$ are inverse isomorphisms:  $\pi\circ j$ is the identity and $j\circ \pi$ is continuously Lipschitz homotopic to the identity (as defined in~\cite[Definition 3.16]{Bartels}).
The category $\Ac (\coprod \{*\})$ has trivial $K$--theory, because it admits an Eilenberg swindle (this is analogous to Bartels~\cite[Remark 3.20]{Bartels}, which treats the case of a single point).
Thus we conclude that $\Ac(P_\bs (\Z))$ has trivial $K$--theory.

A similar argument shows that the subcategory 
$$\mS = \colim_n \Ac \left(\coprod_{r=1}^n \coprod_{\alpha\in A_r} P_{s_r} (Z^r_\alpha)\right) \subset \Ac (P_{\bs } ( \Z))$$
has trivial $K$--theory.
We  conclude that $K_* \bAc P_{\bs}  (\Z) = 0$ for all $*$ by examining the long exact sequence in $K$--theory associated to the Karoubi sequence
$\mS\injects  \Ac P_{\bs} (\Z) \to \bAc P_{\bs }(\Z)$.
 \end{proof}

\vspace{.2in}
If $\gamma$ is a limit ordinal and Proposition~\ref{vanishing-prop4} holds for all $\beta< \gamma$, it follows immediately from the definitions\footnote{This is an oversimplification. See the Addendum (Section~\ref{Addendum}) for full details of the limit ordinal case.} that Proposition~\ref{vanishing-prop4} also holds for $\gamma$.

Next, consider a successor ordinal $\gamma = \beta + 1$ and assume that $\fD_\beta (X)$ is vanishing.  For the rest of the section, we fix  a decomposed sequence $\Z  = (Z^1, Z^2, \ldots) \in \fD_\gamma (X)$, with decompositions  $Z^r = \bigcup_{\alpha\in A_r} Z^r_\alpha$,
and we fix a sequence $\bs\in \Seq$.  We will show that $\Z$ is vanishing at $\bs$.   

Let $C_r = (2\sqrt{2} +1)^{\dim (P_{s_r} (X)) - 1}$ be the sequence of constants from Definition~\ref{C}, and let $\bC = (C_1, C_2, \ldots)$.
Since $\Z\in \fD_\gamma (X)$ and $\gamma = \beta + 1$, for each $r\geqs 1$ and each $\alpha \in A_r$ we may choose decompositions $Z^r_\alpha = U^r_\alpha (\bs) \cup V^r_\alpha (\bs)$
and $(C_r   s_r  r)$--disjoint decompositions
\begin{equation}\label{decomps}U^r_\alpha (\bs) = \coprodmo_{i\in I(r, \alpha) }^{(C_r   s_r   r)\textrm{--disjoint}} U^r_{\alpha i} (\bs)
\,\,\,\,\,\,\,\,\,\textrm{    and    }\,\,\,\, \,\,\,\,\,V^r_\alpha = \coprodmo_{j\in J(r, \alpha) }^{(C_r   s_r  r)\textrm{--disjoint}} V^r_{\alpha j} (\bs) \end{equation}
such that  
\begin{equation}\label{db}\{U^r_{\alpha i} (\bs) : \alpha\in A_r,\, i\in I(r, \alpha)\},  \, \{V^r_{\alpha j} (\bs) : \alpha\in A_r,\, j\in J(r, \alpha)\} \in \fD_\beta.\end{equation}
Setting 
\begin{equation}\label{s-decomp}U^r (\bs) = \bigcup_{\alpha\in A_r} U^r_\alpha (\bs) \, \textrm{    and    } \,V^r (\bs) = \bigcup_{\alpha\in A_r} V^r_\alpha (\bs),\end{equation}
we have decomposed sequences 
$$\U_\bs = (U^1 (\bs), U^2 (\bs), \ldots) \, \textrm{    and    } \, \V_\bs  (V^1 (\bs), V^2 (\bs), \ldots),$$ 
with decompositions given by (\ref{s-decomp}); note that $\Z = \U(\bs) \cup \V(\bs)$.  On the other hand, we can also consider $\U_\bs$ and $\V_\bs$ as decomposed sequences under the finer decompositions
$$U^r(\bs) = \bigcup_{\alpha\in A_r} \bigcup_{i\in I(r,\alpha)} U^r_{\alpha i} (\bs) \, \textrm{    and    } \, 
V^r (\bs)= \bigcup_{\alpha\in A_r} \bigcup_{j\in J(r,\alpha)} V^r_{\alpha j}(\bs).$$
We will denote these more finely decomposed sequences by $\U'_\bs$ and $\V'_\bs$.  Note that by (\ref{db}), we have $\U_\bs', \V_\bs'\in \D_\beta (X)$.  We will use this observation in the proofs of Lemmas~\ref{rho-lemma} and~\ref{mu-lemma}.  In the sequel, we will often write $\U = \U_{\bs}$, $\V= \V_{\bs}$, $\U'=\U'_{\bs}$, and $\V'=\V'_{\bs}$, suppressing the dependence of these sequences and their underlying data on  $\bs$ (and similarly for $U^r_{\alpha i} (\bs)$ and $V^r_{\alpha j} (\bs)$).

For any $\bs', \bs''\in \Seq$ satisfying $\bs\leqs \bs' \leqs \bs''$, Theorems~\ref{Rips-MV},~\ref{rel-MV}, and~\ref{MV-diagram} imply that there is a commutative diagram as follows, in which the first column comes from the Mayer--Vietoris sequence in Theorem~\ref{Rips-MV} and the
second column is the colimit, over $t>0$, of the Mayer--Vietoris sequences from Theorem~\ref{rel-MV} (we write $\bigcup_t$ rather than $\mocolim_t$ to save space):
\begin{equation}\label{MV-diag-eq}
\xymatrix{
	& *\txt{$\bigcup\limits_t K^{+}_*  \left(P_{\bs, \bs'} (\U, \bbW_{t\bC \bs}) \right)$\\
		$\bigoplus$ \\  $\bigcup\limits_t K^{+}_*  \left(P_{ \bs, \bs'} (\V,  \bbW_{t\bC \bs}) \right)$}  \ar[d]^(.6){i_\U + i_\V} \ar[r]^-{\mu_{\bs, \bs', \bs''}}
              & *\txt{$\bigcup\limits_t K^{+}_* \left(   P_{ \bs''}(N^\Z_{t \bC \bs} \U  )\right)$ \\ $\bigoplus$ \\ $\bigcup\limits_t K^{+}_*\left( P_{\bs''}(N^\Z_{t \bC \bs} \V )\right)$} \ar[d]\\
K_* (P_{\bs} (\Z)) \ar[r]^-{\gamma_{\bs, \bs'}} \ar[d]^-{\partial} 
	& \bigcup\limits_t K_* \left(P_{ \bs, \bs'} (\Z,  \bbW_{t\bC \bs}) \right) \ar[d]^-{\partial} \ar[r]^-{\zeta_{\bs, \bs', \bs''}}
        & K_*  (P_{ \bs''} ( \Z))\\
K_{*-1} \left(\I_{\bs} (\U, \V) \right)
\ar[r]^-{\rho_{\bs, \bs'}}
	& \bigcup\limits_t K_{*-1} \left(\I'_{\bs, \bs', t} (\U, \V) \right).\\
}
\end{equation}
The importance of Diagram (\ref{MV-diag-eq}) stems from the fact  (which follows easily from the definitions below) that the composite $\zeta_{\bs, \bs', \bs''} \circ\gamma_{\bs, \bs'}$ is  the natural map 
$$\eta_{\bs, \bs'}\co K_* \left(\bAc (P_\bs (\Z))\right) \xmaps{\eta_{\bs, \bs'}} K_* \left(\bAc (P_{\bs'} (\Z))\right).$$
We now explain the various terms in Diagram (\ref{MV-diag-eq}).
\begin{itemize}
\item The functor $K_*$ is shorthand for $K_* \bAc$.

\item $t\bC\bs$ is the product sequence with $r\ts{th}$ term $t C_r s_r$, and $C_r = C_r (s_r, X)$ is the constant from Definition~\ref{C}.

\item The sequence $\bbW_{t\bC\bs} = \bbW_{t\bC\bs} (\U, \V, \Z)$ was defined in Definition~\ref{W_t-def}.

\item  The functor $K^{+}_*$ is shorthand for $K_* \bAc^{\Z +}$ (Definitions~\ref{decomposed} and~\ref{W_t-def}).

\item The maps $\gamma_{\bs, \bs'}$ and $\rho_{\bs, \bs'}$ are simply the compositions of the maps appearing in Theorem~\ref{MV-diagram} (for any chosen $t>0$) with the natural maps to the colimits.  Theorem~\ref{MV-diagram} implies that the left-hand square in Diagram (\ref{MV-diag-eq}) commutes.

\item In the third column, $N^\Z_{t \bC \bs} \U$ is the decomposed sequence with $r$th term 
\begin{equation}\label{NZU}\bigcup_{\alpha\in A_r} Z^r_\alpha \cap N_{tC_r s_r} U^r_\alpha\end{equation}
and with decompositions exactly as shown in (\ref{NZU}),
and similarly for $\V$ in place of $\U$.

\item The vertical map in the third column arises from the inclusions
$$P_{\bs''} \left(N^\Z_{t\bC\bs} \U\right) \subset P_{\bs''} (\Z) \textrm{  and  }  P_{\bs''} \left(N^\Z_{t\bC\bs} \V\right) \subset  P_{\bs''} (\Z).$$

\item To describe the horizontal map $\zeta = \zeta_{\bs, \bs', \bs''}$, note that for each $t>0$, the inclusion  of simplicial complexes
$$P_{\bs, \bs'} (\Z, \bbW_{t\bC\bs}) \subset P_{\bs''} (\Z)$$
induces a functor after applying $\bAc (-)$ (Lemma~\ref{functoriality}).  These maps are compatible as $t$ increases, and $\zeta$ is the induced map from the colimit.  

\item The map $\mu_{\bs, \bs', \bs''}$ is the direct sum of  maps
 $\mu_{\bs, \bs', \bs''} (\U)$ and  $\mu_{\bs, \bs', \bs''} (\V)$ 
induced by the inclusions
$$\hspace{.45in} P_{\bs, \bs'} (\U, \bbW_{t\bC\bs}) \subset P_{\bs''}(N^\Z_{t \bC \bs} \U ) \,\, \textrm{ and } \,\,
P_{\bs, \bs'} (\V, \bbW_{t\bC\bs}) \subset P_{\bs''}(N^\Z_{t \bC \bs} \V ).$$
\end{itemize}
Note that the term-wise colimit of a (directed) sequence of exact sequences is exact, so the second column of Diagram (\ref{MV-diag-eq}) is exact.  Commutativity of the right-hand square in Diagram (\ref{MV-diag-eq}) is immediate from the definitions of the functors inducing the maps.

We will prove the following two lemmas, which will allow us to deduce that $\Z$ is vanishing by chasing Diagram (\ref{MV-diag-eq}).

\begin{lemma}$\label{rho-lemma}$
For each
$x \in K_{*-1} (\I_\bs (\U_\bs, \V_\bs))$, there exists $\bs'\geqs \bs$ such that $\rho_{\bs, \bs'} (x) = 0$.  
\end{lemma}

\begin{lemma}$\label{mu-lemma}$
For each sequence  $\bs'\geqs \bs$, and for each element 
$$x\in \left(\bigcup_t K^{+}_*  \left(P_{\bs, \bs'} (\U_\bs, \bbW_{t\bC \bs}) \right)\right) \oplus \left(\bigcup_t K^{+}_*  \left(P_{ \bs, \bs'} (\V_\bs,  \bbW_{t\bC \bs}) \right)\right),$$ 
there exists $\bs''\geqs\bs'$ such that $\mu_{\bs, \bs', \bs''} (x) = 0$.  
\end{lemma}

\noindent {\bf Proof of Proposition~\ref{vanishing-prop4} assuming Lemmas~\ref{rho-lemma} and~\ref{mu-lemma}.}  For simplicity, we drop most subscripts from the maps in Diagram (\ref{MV-diag-eq}).
For each element $x\in K_* \left(\bAc \left(P_\bs (\Z)\right)\right)$, we have $\partial(\gamma x) = \rho (\partial x)$.  By Lemma~\ref{rho-lemma}, we can choose $\bs'$ large enough so that $\rho (\partial x) = 0$.  Exactness of the second column in Diagram (\ref{MV-diag-eq}) then shows that $\gamma (x) = (i_\U + i_\V) (x_1, x_2)$ for some $x_1, x_2$.  Now Lemma~\ref{mu-lemma} tells us that for $\bs''$ large enough, we have $\mu (x_1, x_2) = 0$, and it follows from commutativity of the right-hand square of Diagram~(\ref{MV-diag-eq}) that $\zeta (\gamma x) = \zeta ( (i_\U + i_\V) (x_1, x_2)) = 0$.  However, as mentioned above the composite $\zeta \circ\gamma$ is simply the natural map 
$$\eta_{\bs, \bs'}\co K_* \left(\bAc (P_\bs (\Z))\right) \xmaps{\eta_{\bs, \bs'}} K_* \left(\bAc (P_{\bs'} (\Z))\right).$$
Hence $\Z$ is vanishing at $\bs$.  Since $\bs\in \Seq$ was arbitrary, $\Z$ is in fact a vanishing sequence, and our induction is complete.
$\hfill \Box$

\vspace{.2in}

To prove Lemma~\ref{rho-lemma}, we need to compare two versions of the category of controlled modules on $P_\bq (\W)$, where $\bq\in \Seq$ and $\W = (W^1, W^2, \ldots)$ is a decomposed sequence in $X$ with decompositions $W^r = \bigcup_{\alpha\in A_r} W^r_\alpha$.  
We may give $P_\bq (\W)$ either its intrinsic simplicial metric or the simplicial metric inherited from $P_\bq (\X_\W)$ (where $\X_\W$ is the decomposed sequence defined in Remark~\ref{triv-seq-rmk}).   The   category corresponding to the first metric will be denoted $\bAc^\W (P_\bq (\W))$.   The latter metric is the one used to define the category $\Ac(P_\bq (\W))$, and we will sometimes write $\bAc^X (P_\bq (\W)) = \bAc (P_\bq (\W))$  simply to emphasize the chosen metric on $P_\bq (\W)$.

\begin{lemma}$\label{metric-ind}$ Let $\W$ be a decomposed sequence in $X$ and let $\bq\in \Seq$ be any sequence.  Then there exist functors
\begin{equation}\label{Phi_s}\Phi_\bq \co \bAc^{ X} (P_\bq (\W)) \maps \mocolim_{\bn\in \Seq} \bAc^\W (P_\bn (\W))
\end{equation}
that
make the diagram
\begin{equation}\label{compatible}\xymatrix{ \bAc^ X (P_\bq \W) \ar[d] \ar[r]^-{\Phi_\bq} & \mocolim\limits_{\bn\in\Seq} \bAc^\W (P_\bn \W)\\
		\bAc^ X (P_{\bq'} \W) \ar[ur]_-{\Phi_{\bq'}}
		}
\end{equation}
commute whenever $\bq\leqs \bq'$.
\end{lemma}

\begin{proof}

To construct the functors $\Phi_\bq$, first note that objects in $\bAc^ X (P_\bq \W)$ are also objects in  $\bAc^\W (P_\bq \W)$ because the change of metrics does not affect which sets are compact, and hence the locally finiteness condition is the same in both cases.  Letting 
$$i_\bq\co \bAc^\W (P_\bq (\W))\maps \colim_{\bn\in\Seq} \bAc^\W (P_\bn (\W))$$
denote the structure map for the colimit, we can now define  $\Phi_\bq$ to be the identity on objects by setting $\Phi_\bq (M) = i_\bq (M)$.
More care is required to define $\Phi_\bq$ on morphisms.  

A morphism in $\psi \co M\to N$ in $\Ac^ X (P_\bq \W)$ is a bounded map of geometric modules on $P_\bq (\W) \cross [0,1)$ (with metric induced from $P_\bq (\X_\W) \cross [0,1)$) which is controlled at 1.  Let $d<\infty$ denote a bound on the propagation of $\psi$, and let $\bq'$ be the sequence with $r\ts{th}$ term $q'_r = q_r C_r (d+2)$ (where $C_r = C_r (q_r, X)$ is the constant from Definition~\ref{C}).  We claim that $\psi\in \Ac^{\W} (P_{\bq'} (\W))$.
First we check that $\psi$ is bounded as a morphism on $P_{\bq'} (\W) \cross [0,1)$, where $P_{\bq'} (\W)$ has its intrinsic simplicial metric.  Let $(a_1, t_1), (a_2,t_2)\in P_\bq (\W) \cross [0,1)$ be points such that 
$\psi_{(a_1,t_1),(a_2,t_2)}\neq 0$.  Note that this implies that $a_1, a_2 \in P_{q_r} (W_r)$ for some $r$.  Choose barycentric vertices $v_1, v_2\in W_r$ for $a_1$ and $a_2$ (respectively).   
 Lemma~\ref{metric-comp} implies that $d(v_1,v_2) \leqs q_r C_r (d+2)$, so $v_1$ and $v_2$ lie in a common simplex in $P_{q'_r} (W_r)$ (by choice of $q'_r$).  It follows that $a_1$ and $a_2$ are at most distance $3$ apart in the simplicial metric on $P_{q'_r} (W_r)$, so $(a_1,t_1)$  and $(a_2,t_2)$ are at most distance $4$ apart in the corresponding metric on $P_{q'_r} (W_r)\cross [0,1)$.  Hence $\psi$ has propagation at most $4$ in this metric.

To check that $\psi$ is controlled as a morphism on $P_{\bq'} (\W)\cross [0,1)$, 
note that (continuous) control is a \e{topological} condition: it does not refer to the metric on the complex in question.  Since $\psi$ is controlled as a morphism on $P_{\bq} (\W)\cross [0,1)$, Lemma~\ref{functoriality} implies that $\psi$ is also controlled on $P_{\bq'} (\W)\cross [0,1)$ (for the purpose of applying Lemma~\ref{functoriality}, we can give these complexes the metrics inherited from $P_\bq (\X_\W)$ and $P_{\bq'} (\X_\W)$, respectively).

We can now define $\Phi_\bq ([\psi])$ to be the morphism $\Phi_\bq (M) \to \Phi_\bq (N)$ represented by $\psi$.
Since $\Phi_\bq$ does not change the underlying data of either geometric modules or morphisms, it follows that $\Phi_\bq$ is a functor and that Diagram  (\ref{compatible}) commutes.  
\end{proof}

The key result behind the proofs of Lemmas~\ref{rho-lemma} and~\ref{mu-lemma} is a comparison between the categories of controlled modules associated to a decomposed sequence and to a sufficiently good refinement of that sequence.  First we record a lemma regarding the construction $\bAc (-)$.

\begin{lemma}$\label{>R}$
Let $ K_1 \subset K_1', K_2 \subset K_2', \ldots$ be locally finite simplicial complexes.  Then for each $R>0$, the inclusion
$$\coprod_{r\geqs R} K_r \injects \coprod_{r \geqs 1} K_r$$
induces an equivalence of categories 
$$i_R \co \bAc \left(\coprod_{r\geqs R} K_r\right) \srm{\isom} \bAc\left(\coprod_{r\geqs 1} K_r\right),$$
where on the left we use the simplicial metric from $ \coprod_{r\geqs R} K_r' $ and on the right  we use the simplicial metric from $ \coprod_{r\geqs 1} K_r' $.

In particular, for each decomposed sequence $\Y = (Y^1, Y^2, \ldots)$  in $X$ 
 and each $\bq \in \Seq$, there is an equivalence of categories
$$i_R \co \bAc \left(P_{\bq(R)} \left(\Y(R)\right)\right) \srm{\isom} \bAc \left(  P_\bq \left(\Y\right)\right),$$
where $\bq(R) = (q_{R}, q_{R+1}, \ldots)$ and $\Y(R)$ denotes the decomposed sequence $(Y^{R}, Y^{R+1}, \ldots)$, with the same decompositions as in $\Y$.
\end{lemma}
\begin{proof}  The functor $i_R$ exists by Lemma~\ref{functoriality}, and it follows  from the definitions that $i_R$ is full and faithful.  Each module $M\in \Ac \left(\coprod_{r\geqs 1} K_r\right)$ is isomorphic, in  the Karoubi quotient  $\bAc \left( \coprod_{r\geqs 1} K_r\right)$, to its restriction 
$$M\left(\coprod_{r\geqs R}K_r\cross[0,1) \right).$$  
This restriction is in the image of $i_R$, completing the proof.
\end{proof}

\begin{lemma}$\label{refinement}$ Let $\Y = (Y^1, Y^2, \ldots)$ be a decomposed sequence in $X$, with decompositions $Y^r = \bigcup_{\alpha\in A_r} Y^r_\alpha$.  Consider sequences $\bq$, $\bbf \in \Seq$ satisfying $\lim_{r\to \infty} f_r/ C_r q_r = \infty$, where $C_r = C(q_r, X)$ is the constant from Definition~\ref{C}.  Assume that for each sufficiently large $r$ and each $\alpha\in A_r$ we have a decomposition 
$$Y^r_\alpha = \coprodmo_{i\in I(r, \alpha)}^{f_r\textrm{--disjoint}} Y^r_{\alpha i}.$$
Let $\Y'$ be the decomposed sequence $\Y' = (Y^1, Y^2, \ldots)$ with decompositions
$$Y^r = \bigcup_{\alpha\in A_r} \bigcup_{i\in I(r, \alpha)} Y^r_{\alpha i}.$$

Then there are maps
$$\Psi_\bt \co K_* \left(\bAc \left(P_\bt (\Y')  \right) \right) \maps K_* \left( \bAc  \left(P_\bt (\Y) \right)\right),$$
natural with respect to $\bt\in \Seq$, and $\Psi_\bq$ is an \e{isomorphism} in all dimensions.
Consequently, if $\Y'$ is  vanishing at $\bq$, then so is $\Y$.
\end{lemma}
\begin{proof} Naturality of the maps $\Psi_\bt$ means that for $\bt'\geqs \bt$, we will construct a commutative diagram
\begin{equation}\label{Phi-diag}\xymatrix{ K_* \left(\bAc \left( P_\bt (\Y)\right)\right) \ar[rr]^-{\eta_{\bt, \bt'}(\Y)} 
	& &K_* \left(\bAc \left( P_{\bt'} (\Y)\right)\right)  \\
K_* \left(\bAc \left( P_{\bt} (\Y')\right)\right) \ar[rr]^-{\eta_{\bt, \bt'} (\Y') } \ar[u]^{\Psi_\bt} && K_* \left(\bAc \left( P_{\bt'} (\Y')\right)\right).  \ar[u]^{\Psi_{\bt'}}
}
\end{equation}
The final statement of the lemma will  follow from commutativity of Diagram (\ref{Phi-diag}) together with the definition  of vanishing, once we establish that $\Psi_\bq$ is an isomorphism.

We now define the desired homomorphisms $\Psi_{\bt}$ for each sequence $\bt\in \Seq$.
Our hypotheses imply that  there exists $R>0$ such that if $r\geqs R$ then $f_r > q_r>0$
and $d(Y^r_{\alpha i}, Y^r_{\alpha j})\geqs f_r$ for all $\alpha \in A_r$ and all $i, j\in I(r, \alpha)$ with $i\neq j$.
In particular, for such $r, \alpha, i$ and $j$ we have $Y^r_{\alpha i} \cap Y^r_{\alpha j} = \emptyset$, so there is an injective map of simplicial complexes
$$P_{\bt(R )} (\Y' (R )) \injects P_{\bt(R )} (\Y (R )).$$
(Note that for $r<R$, we allow for the possibility that  $Y^r_{\alpha i} \cap Y^r_{\alpha j} \neq \emptyset$ for some $i\neq j$; this will be important when we apply the Lemma~\ref{refinement} in the proof of Lemma~\ref{rho-lemma}.)
This map is proper and decreases distances, so by Lemma~\ref{functoriality} we have an induced functor
$$\Phi_{\bt} \co \bAc\left( P_{\bt(R )} \left(\Y' (R )\right)\right) \maps \bAc\left( P_{\bt(R )} \left(\Y (R )\right)\right).$$
Lemma~\ref{>R}  yields a diagram
\begin{equation}\label{zz}\xymatrix{
 \bAc \left(P_\bt (\Y')  \right)   & \bAc  \left(P_\bt (\Y) \right) \\
\bAc\left(  P_{\bt(R )} \left(\Y'(R )\right) \right) \ar[r]^-{\Phi_\bt} \ar[u]^-{i'}_-\isom & \bAc \left( P_{\bt(R )}  \left(\Y(R )\right) \right)  \ar[u]^-{i}_-\isom,}
\end{equation}
and we define
$$K_* \left(\bAc \left(P_\bt (\Y')  \right) \right) \srm{\Psi_\bt}  K_* \left( \bAc  \left(P_\bt (\Y) \right)\right)$$
to be the homomorphism obtained from Diagram (\ref{zz}) by inverting $(i')_*$.
Given $\bt'\geqs \bt$, we obtain a commutative diagram linking the zig-zag (\ref{zz}) to the corresponding zig-zag for $\bt'$; this yields commutativity of Diagram (\ref{Phi-diag}).

We need to check that $\Psi_\bq$ is an isomorphism.  We will show that $\Phi_\bq$ is an \e{isomorphism} of categories (not just an equivalence).  
For $r\geqs R$, $\alpha\in A_r$, and $i,j\in I(r, \alpha)$ with $i\neq j$, we have $d(Y^r_{\alpha i}, Y^r_{\alpha j})> q_r$,
so the simplicial complex $P_{q_r} (Y^r_\alpha)$ is the disjoint union of the subcomplexes $P_{q_r} (Y^r_{\alpha i})$.  This shows that $\Phi_\bq$ is bijective on objects: each module on $P_\bq (Y^r_\alpha)\cross [0,1)$ is the direct sum of its restrictions to the disjoint subspaces 
$P_\bq (Y^r_{\alpha i})\cross [0,1)$.  Next, we check that $\Phi_\bq$ is surjective on morphisms.  Given $[\psi] \in\bAc P_{\bq(R)} (\Y(R))$, let $T$ be a bound on the propagation of $\psi$.  Since $f_r/C_r q_r\to \infty$, there exists $R_T$ such that $f_r > (T+1)q_r C_r$ for $r\geqs R_T$.
Now $[\psi] = [\psi(R_T)]$, where $\psi(R_T)$ denotes the morphism 
$$\psi(R_T)_{ a,b } = \left\{ \begin{array}{ll}
					\psi_{a,b}, \,\,\, a, b\in P_{q_r} (Y^r_\alpha) \cross [0,1) \textrm{ for some } r\geqs R_T, \alpha\in A_r,\\
					0, \,\,\,\,\,\,\,\,\,\, else.
				      \end{array}
					\right.
$$  
Lemma~\ref{metric-comp} and our choice of $R_T$ imply that $\psi(R_T)$  is a direct sum, over $r>R$, $\alpha\in A_r$ and $i\in I(r, \alpha)$, of controlled morphisms $\psi(R_T)_{r \alpha i}$ supported on $P_{q_r} (Y^r_{\alpha i})\cross [0,1)$.  Hence $[\psi] = [\psi(R_T)]$ is in the image of $\Phi_\bq$.  Note here that both $P_{q_r} (Y^r_{\alpha i})$ and $P_{q_r} (Y^r_{\alpha})$ have the metric inherited from $P_{q_r} (X)$, so the propagation of 
$$\psi(R_T) = \bigoplus_{r \geqs R_T} \bigoplus_{\alpha\in A_r} \bigoplus_{i\in I(r, \alpha)} \psi(R_T)_{r \alpha i}$$ 
is the same whether we consider it as a morphism between  modules on $P_\bq (\Y)$ or on $P_\bq (\Y')$.

Finally, we must check that $\Phi$ is faithful.  If $\Phi_\bq ([\psi_1]) = \Phi_\bq ([\psi_2])$, then for sufficiently large $R_0$, the restrictions of $\psi_1$ and $\psi_2$ to  
$$\coprod_{r>R_0} \coprod_{\alpha\in A_r} P_{q_r} (Y^r_{\alpha})\cross [0,1)$$
 are identical, and hence $[\psi_1] = [\psi_2]$.
\end{proof}

\noindent {\bf Proof of Lemma~\ref{rho-lemma}.}
Given $t>0$, $r\geqs 1$, and $\alpha\in A_r$, we define
\begin{equation}\label{wrta}W^r_{t \alpha} = N_{tC_r s_r} (U^r_\alpha) \cap N_{tC_r s_r} (V^r_\alpha)\cap Z^r_\alpha.\end{equation}
For each $t > 0$ we define the decomposed sequence
$\W_t = \W_t (\U, \V, \Z)$, whose $r\ts{th}$ term is 
\begin{equation}\label{W^r_t} W^r_t = \bigcup_{\alpha\in A_r} W^r_{t \alpha};
\end{equation}
the decomposition of $W^r_t$ is exactly that displayed in (\ref{W^r_t}).
We claim that $\W_t$ is vanishing at $\bs$.  Consider the decomposed sequence $\W_t'$, whose $r\ts{th}$ term is the same as that of $\W_t$, but with the 
finer decomposition
\begin{equation}\label{fine}
W^r_t = \bigcup_{\alpha\in A_r} \bigcup_{(i,j) \in I(r, \alpha) \cross J(r, \alpha)} W^r_{t\alpha i j},
\end{equation}
where
$$W^r_{t\alpha i j} = N_{tC_r s_r} (U^r_{\alpha i}) \cap N_{tC_r s_r} (V^r_{\alpha j})\cap Z^r_\alpha.$$
By (\ref{db}), the families
$$\{U^r_{\alpha i}   : \alpha\in A_r,\, i\in I(r, \alpha)\} \, \textrm{    and    } \, \{V^r_{\alpha j}   : \alpha\in A_r,\, j\in J(r, \alpha)\}$$ 
lie in $\fD_\beta$, so
the induction hypothesis and Lemma~\ref{nbhds} tell us that $\W_t'$ is a vanishing sequence.   
Our disjointness hypotheses (\ref{decomps}) imply that 
$$W^r_{t \alpha} = \coprod^{C_r s_r (r - 2t)-\textrm{disjoint}}_{(i,j) \in I(r, \alpha) \cross J(r, \alpha)} W^r_{t\alpha i j},$$
and $(C_r s_r (r - 2t))/C_r s_r = r - 2t$ tends to infinity with $r$.
By Lemma~\ref{refinement}, $\W_t$ is  vanishing at $\bs$.

For each $t>0$, let $\bbW_t$ denote the set of metric families
$$\bbW_t = \{ \{W^r_{t \alpha}\}_{\alpha \in A_r} \}_{r\geqs 1}.$$
Note that  $P_{\bs'} (\W_{t} ) \subset P_{\bs, \bs'} (X, \bbW_t)$, where the latter complex was introduced in Definition~\ref{rel-Rips}.
For each $\bs'\geqs \bs$, we will show that the map $\rho_{\bs, \bs'}$ factors through a map
\begin{equation}\label{eta-factorization} \mocolim_{t\to \infty} K_{*-1} \left(\bAc^X  P_{ \bs} (\W_{t })\right) 
\xmaps{(\xi_{\bs, \bs'})_*}  \mocolim_{t\to \infty}  K_{*-1} \left(\bAc^\textrm{rel}  P_{\bs'} (\W_{t} )\right),
\end{equation}
where the superscripts indicate that we give these Rips complexes the  metrics induced from the simplicial metrics on 
$P_\bs (\X_{\W_t})$ and $P_{\bs, \bs'} (X, \bbW_t)$ (respectively).   (This choice of metrics will be important in obtaining the desired factorization of $\rho_{\bs, \bs'}$; in particular, if we used the metric on $P_{\bs'} (\W_t)$ inherited from the simplicial metric on $P_{\bs'} (X)$  to define the codomain of $(\xi_{\bs, \bs'})_*$, we would not be able to define the map $l$ in (\ref{eta}) below.)   Lemma~\ref{functoriality} shows that the inclusion of simplicial complexes
$P_{ \bs} (\W_{t }) \subset P_{\bs'} (\W_{t} )$
induces a functor 
$$\xi_{\bs, \bs', t} \co \bAc^X  P_{ \bs} (\W_{t })\maps  \bAc^\textrm{rel}  P_{\bs'} (\W_{t} ).$$ 
We set $\xi_{\bs, \bs'} = \mocolim_t \xi_{\bs, \bs', t}$, and (\ref{eta-factorization}) is the induced map on $K$--theory.

We now  show that for every class on the left-hand side of (\ref{eta-factorization}), there exists $\bs'\geqs \bs$ such that $(\xi_{\bs, \bs'})_* (x) = 0$.  (The corresponding result for $\rho_{\bs, \bs'}$ will follow immediately once we establish the claimed factorization.)
It suffices to show that the functor $\xi_\bs = \mocolim_{\bs' \in \Seq} \xi_{\bs, \bs'}$ induces the zero-map on $K$--theory.  In Lemma~\ref{metric-ind}, we constructed functors
$$\Phi_{\bs, t} \co \bAc^{X} (P_\bs (\W_t)) \maps \mocolim_{\bs'\in \Seq} \bAc^{\W_t} (P_{\bs'} (\W_t)),$$
where again the superscripts indicate the chosen metrics on the Rips complexes (see the discussion preceding Lemma~\ref{metric-ind}).  Lemma~\ref{functoriality} yields functors $\bAc^{\W_t} (P_{\bs'} (\W_t))\xmaps{\Psi_{\bs', t}} \bAc^\textrm{rel}  (P_{\bs'} (\W_{t}))$ for each $\bs' \in \Seq$ and each $t>0$, and now $\xi_\bs= \mocolim_{\bs' \in \Seq} \xi_{\bs, \bs'}$   factors as
\begin{eqnarray*}  \mocolim_{t\to \infty} \bAc^X  \left(P_{ \bs} (\W_{t })\right) \xmaps{\mocolim_{t} \Phi_{\bs, t}}  \mocolim_{t\to \infty} \mocolim_{\bs'\in \Seq} \bAc^{\W_t} (P_{\bs'} (\W_t))\hspace{1.2in}\\
 \hspace{.2in} \xmaps{\colim_{t, \bs'}  \Psi_{\bs', t}} \mocolim_{t\to \infty} \mocolim_{\bs'\in \Seq}   \bAc^\textrm{rel}  P_{\bs'} (\W_{t} )  \isom
\mocolim_{\bs'\in \Seq}   \mocolim_{t\to \infty}    \bAc^\textrm{rel}  P_{\bs'} (\W_{t} ).
 \end{eqnarray*}
 It will suffice to show that the maps $\Phi_{\bs, t}$ induce the zero map on $K$--theory for all $t$.  
Recall (see Diagram (\ref{compatible})) that for each $t$ and each  $\bs'\geqs \bs$, the map $\Phi_{\bs, t}$ factors through the natural map 
$$ \bAc^X  P_{ \bs} (\W_{t }) \xmaps{ \eta_{\bs, \bs', t}}  \bAc^X  P_{ \bs'} (\W_{t }),$$
and hence $\mocolim_{t} \Phi_{\bs, t}$ factors through $\colim_{t, \bs'}  \eta_{\bs, \bs', t}$. 
As established above, $\W_t$ is vanishing at $\bs$ for every $t$, so $\colim_{\bs'}  \eta_{\bs, \bs', t}$ induces the trivial map on $K$--theory for each $t$.  Hence the first map $\mocolim_{t} \Phi_{\bs, t}$ in the above composition induces the  trivial map on $K$--theory, and we conclude that the same is true of $\xi_\bs$ (as desired).

The desired factorization of $\rho_{\bs, \bs'}$ comes from a sequence of functors
\begin{eqnarray}\label{eta}  
\I_\bs (\U, \V) \srm{i} \mocolim_{t\to \infty} \bAc^X \left(P_\bs (\Z) \cap (N_t P_\bs \U)\cap (N_t P_\bs \V)\right)\hspace{1in} \notag\\
 \srm{j}  \mocolim_{t\to \infty} \bAc^X \left(  P_\bs ( \W_t)\right)
\xmaps{\colim_t \xi_{\bs, \bs', t}}\mocolim_{t\to \infty} \bAc^\textrm{rel} \left(  P_{\bs'} ( \W_t) \right)\\
\srm{l} \mocolim_{t\to \infty} \I'_{\bs, \bs', t} (\U, \V).\notag
\end{eqnarray}
The functors $i$ and $l$ are inclusions of categories that exist by the definitions of the intersection terms in the Mayer--Vietoris sequences.  (In the case of $l$, note that both of these categories are defined using the simplicial metric on $P_{\bs, \bs'} (X, \bbW_t)$, and use the fact that for metric spaces $A\subset B$, a continuously controlled morphism between geometric modules on $A\cross [0,1)$ is also continuously controlled on $B\cross [0,1)$.  This latter fact was shown in the proof of Lemma~\ref{metric-ind}, and also follows from Lemma~\ref{functoriality}.)
  The functor $j$ exists by Equation (\ref{rnc1}) in Lemma~\ref{rel-nbhd-comp} (which may be applied to non-relative Rips complexes simply by setting the two parameters $s, s'$ appearing in the Lemma to be equal), and it is immediate from the definitions that the composite of these functors is the functor inducing $\rho_{\bs, \bs'}$ on $K$--theory.  
$\hfill \Box$

\vspace{.2in}
For the proof of Lemma~\ref{mu-lemma}, we need a relative version of (one part of) Lemma~\ref{refinement}.

\begin{lemma}$\label{rel-refinement}$ Let $\Q = (Q^1, Q^2, \ldots)$ and $\Y = (Y^1, Y^2, \ldots)$ be decomposed sequences in $X$ satisfying $\Q\subset  \Y$, and say the decompositions of these sequences are $Q^r = \bigcup_{\alpha\in A_r} Q^r_\alpha$ and $Y^r = \bigcup_{\alpha\in A_r} Y^r_\alpha$.  Let $\bq$, $\bbg \in \Seq$ satisfy 
$$\lim_{r\to \infty} g_r/ C_r q_r = \infty,$$ 
where $C_r = C(q_r, X)$ is the constant from Definition~\ref{C}.  Assume that for each $r$ and each $\alpha\in A_r$ we have decompositions 
$$Q^r_\alpha = \bigcup_{i\in I(r, \alpha)} Q^r_{\alpha i}.$$
Let $\Q'$ and $\Y'$ (respectively) be the decomposed sequences $\Q' = (Q^1, Q^2, \ldots)$  and 
$\Y' = (Y^1, Y^2, \ldots)$, with decompositions
$$Q^r = \bigcup_{\alpha\in A_r} \bigcup_{i\in I(r, \alpha)} Q^r_{\alpha i} \,\,\,  \textrm{and} \,\,\, Y^r = \bigcup_{\alpha\in A_r} \bigcup_{i\in I(r, \alpha)} Y^r_{\alpha i},$$
where $Y^r_{\alpha i} = Y^r_\alpha$ for each $r\geqs 1$ and each $i\in I(r, \alpha)$.  Note that  $\Q'\subset \Y'$.

Assume further that we are given a set
$$\bbW = \{\bbW^r_\alpha : r\geqs 1,\, \alpha\in A_r\}$$ 
of metric families in $X$.
Let 
$$\bbW' = \{\bbW^r_{\alpha i} : r\geqs 1,\, \alpha\in A_r,\, i\in I(r, \alpha)\}$$
 be a refinement of $\bbW$, in the sense that for each $r\geqs 1$ and each $\alpha\in A_r$, we have
$$\bbW^r_\alpha = \{S  : S\in \bbW^r_{\alpha i} \textrm{ for some } i\in I(r,\alpha)\}.$$  
Given $r\geqs 1$, $\alpha\in A_r$, and $i\in I(r,\alpha)$, let 
\begin{equation}\label{W}W^r_{\alpha i} = \bigcup \bbW^r_{\alpha i} = \{x\in X : x\in S \textrm{ for some } S\in \bbW^r_{\alpha i}\}
\end{equation}
denote the union of all the sets in the family 
$\bbW^r_{\alpha i}$.
Assume that there exists $R_0>0$ such that 
 for each $r\geqs R_0$ and each $\alpha\in A_r$, the family 
$$\{Q^r_{\alpha i} \cup  W^r_{\alpha i}\}_{i\in I(r, \alpha)}$$
 is $g_r$--disjoint, meaning that  
\begin{equation}\label{g}d(Q^r_{\alpha i} \cup  W^r_{\alpha i}, Q^r_{\alpha j} \cup  W^r_{\alpha j}) > g_r\end{equation}
 for $i\neq j$.

Then for each $\bq'\in \Seq$ and each $*\in \bbZ$, there is an isomorphism
$$\Psi_{\bq, \bq'} \co K_* \left(\bAc^{\Y'+} \left(P_{\bq, \bq'} (\Q', \bbW')  \right) \right) \srm{\isom} K_* \left( \bAc^{\Y+}  \left(P_{\bq, \bq'} (\Q, \bbW) \right)\right).$$
\end{lemma}

\begin{proof}
For each $T>0$, define
$$N_{T, d_\Delta'}^{Y^r_\alpha} \left(P_{q_r, q'_r} \left(Q^r_{\alpha i}, \bbW^r_{\alpha i}\right)\right) :=  P_{q_r, q'_r} \left(Y^r_\alpha, \bbW^r_{\alpha i}\right) \cap N_{T, d_\Delta'} \left(P_{q_r, q'_r} \left(Q^r_{\alpha i}, \bbW^r_{\alpha i}\right)\right),$$
where on the right, the neighborhood is taken inside the larger complex $P_{q_r, q'_r} \left(X, \bbW^r_{\alpha i}\right)$, with its simplicial metric $d_\Delta'$.
Similarly, let 
$$N_{T, d_\Delta }^{Y^r_\alpha} \left(P_{q_r, q'_r}  \left(Q^r_\alpha, \bbW^r_\alpha\right)\right)
 :=  P_{q_r, q'_r} \left(Y^r_\alpha, \bbW^r_{\alpha}\right) \cap N_{T, d_\Delta  } \left(P_{q_r, q'_r} \left(Q^r_{\alpha}, \bbW^r_{\alpha}\right)\right),$$
where on the right, the neighborhood is taken inside $P_{q_r, q'_r} \left(X, \bbW^r_{\alpha}\right)$ with its simplicial metric $d_\Delta = d_\Delta (r, \alpha)$.

Our hypotheses imply that for each $T>0$, there exists $R_T \geqs R_0$ such that if $r\geqs R_T$ then 
\begin{equation}\label{g2}g_r>(2T+2) C_rq_r.\end{equation} 
Set
$$K'_{T, d'_\Delta} (R_T) := \coprodmo_{r\geqs R_T} \coprodmo_{\alpha\in A_r} \coprodmo_{i\in I(r, \alpha)} N_{T, d'_\Delta}^{Y^r_\alpha} \left(P_{q_r, q'_r} (Q^r_{\alpha i}, \bbW^r_{\alpha i})\right)$$
and
$$K_T(R_T) :=
\coprodmo_{r\geqs R_T} \coprodmo_{\alpha\in A_r} N_{T, d_\Delta}^{Y^r_\alpha} \left(P_{q_r, q'_r} (Q^r_{\alpha  }, \bbW^r_{\alpha })\right).$$
Let $\bbW(R_T)$ and $\bbW'(R_T)$
denote the sets of metric families
$$\{\bbW^{r}_\alpha : r\geqs R_T,\, \alpha\in A_r\} \,\,\, \textrm{ and } \,\,\, \{\bbW^{r}_{\alpha i} : r\geqs R_T,\, \alpha\in A_r,\, i\in I(r, \alpha)\},$$
respectively.
We have inclusions of simplicial complexes
$$K'_{T, d'_\Delta} (R_T) \subset P_{\bq(R_T), \bq'(R_T)} \left(X, \bbW'(R_T)\right)$$
and
$$K_T (R_T) \subset P_{\bq(R_T), \bq'(R_T)} \left(X, \bbW(R_T)\right)$$
(recall that these relative Rips complexes were introduced in Definition~\ref{rel-Rips}), and 
we give $K'_{T, d'_\Delta} (R_T)$ and $K_T (R_T)$ the metrics induced from the simplicial metrics on  these relative Rips complexes.

The inclusion maps  
$$N_{T, d'_\Delta}^{Y^r_\alpha} \left(P_{q_r, q'_r} (Q^r_{\alpha i}, \bbW^r_{\alpha i})\right) \injects N_{T, d_\Delta}^{Y^r_\alpha} \left(P_{q_r, q'_r}  \left(Q^r_\alpha, \bbW^r_\alpha\right)\right)$$
combine to yield a simplicial map
$$K'_{T, d'_\Delta} (R_T)  \xmaps{\phi_{\bq, \bq'}(T)}  K_T (R_T),$$
which decreases distances (by our choice of metrics).
We claim that $\phi_{\bq, \bq'} (T)$ is injective as well.  If not, we would have
\begin{eqnarray*}d_\Delta \left(P_{q_r, q'_r} \left(Q^r_{\alpha i}, \bbW^r_{\alpha i}\right), P_{q_r,  q'_r} \left(Q^r_{\alpha j}, \bbW^r_{\alpha j}\right) \right) \hspace{1in}\\
\hspace{1in} \leqs d'_\Delta \left(P_{q_r, q'_r} \left(Q^r_{\alpha i}, \bbW^r_{\alpha i}\right), P_{q_r,  q'_r} \left(Q^r_{\alpha j}, \bbW^r_{\alpha j}\right) \right) < 2T
\end{eqnarray*}
for some $r>R_T$, $\alpha\in A_r$, and $i, j\in I(r, \alpha)$ with $i\neq j$, and then Lemma~\ref{rel-metric-comp} would yield
$$d\left(Q^r_{\alpha i} \cup W^r_{\alpha i}, \bigcup_{k\in I(r, \alpha),\, k\neq i} \left(Q^r_{\alpha k} \cup  W^r_{\alpha k}\right)\right)
\leqs (2T+2)C_r q_r,$$
contradicting (\ref{g}) and (\ref{g2}).

Since $\phi_{\bq, \bq'} (T)$ is injective and decreases distances, by Lemma~\ref{functoriality} it induces a functor
$$\bAc(K'_{T, d'_\Delta}(R_T)) \xmaps{\Phi_{\bq, \bq'}(T)} \bAc(K_T(R_T)).$$

Define
$$
K'_{T, d'_\Delta} := \coprodmo_{r\geqs 1 }  \coprodmo_{\alpha\in A_r} \coprodmo_{i\in I(r, \alpha)}  N_{T, d'_\Delta }^{Y^r_\alpha} \left(P_{q_r, q'_r} (Q^r_{\alpha i}, \bbW^r_{\alpha i})\right)$$
and
$$K_T := \coprodmo_{r\geqs 1}  \coprodmo_{\alpha\in A_r} N_{T, d_\Delta }^{Y^r_\alpha} \left(P_{q_r, q'_r}  \left(Q^r_\alpha, \bbW^r_\alpha\right)\right),$$
and give these complexes the metrics induced by the simplicial metrics on $P_{\bq, \bq'} (X, \bbW')$ and $P_{\bq, \bq'} (X, \bbW)$, respectively.
By Lemma~\ref{>R}, we have a diagram
\begin{equation}\label{zz-rel}\xymatrix{
\bAc(K'_{T, d'_\Delta})   & & \bAc(K_T)  \\
\bAc(K'_{T, d'_\Delta}(R_T) ) \ar[rr]^-{\Phi_{\bq, \bq'} (T)} \ar[u]^-{i'}_-\isom & &\bAc (K_T(R_T) )  \ar[u]^-{i}_-\isom,}
\end{equation}
and we define
$$K_* \left(\bAc(K'_T) \right) \xmaps{\Psi_{\bq, \bq'} (T)}  K_* \left( \bAc(K_T)\right)$$
by the equation  
$$\Psi_{\bq, \bq'} (T) = i_* \circ (\Phi_{\bq, \bq'} (T))_* \circ(i')_*^{-1}.$$

By definition, we have
$$\bAc^{\Y+} P_{\bq, \bq'} (\Q, \bbW) = \mocolim_{T>0}  \bAc (K_T),$$
and  
\begin{equation} \label{prime}\bAc^{\Y'+} P_{\bq, \bq'} (\Q', \bbW') = \mocolim_{T>0}  \bAc (K'_{T, d'_\Delta}).\end{equation}
The maps $\Psi_{\bq, \bq'} (T)$ are natural with respect to $T$, 
so we obtain the desired map  $\Psi_{\bq, \bq'} = \mocolim_{T>0} \Psi_{\bq, \bq'}(T)$:
$$K_* \left(\bAc^{\Y'+} \left(P_{\bq, \bq'}  (\Q', \bbW')  \right) \right) \xmaps{\Psi_{\bq, \bq'}}  K_* \left( \bAc^{\Y+}  \left(P_{\bq, \bq'}  (\Q, \bbW) \right)\right).$$
We need to check that $\Psi_{\bq, \bq'}$ is an isomorphism.  We will show that $\Phi_{\bq, \bq'}(T)$ is an isomorphism of categories for each $T>0$.

We claim that the maps $\phi_{\bq, \bq'}(T)$ are actually \e{bijections}.  We have already shown that $\phi_{\bq, \bq'}(T)$ is injective, so we need only consider surjectivity.  

Set
$$N_{T, d_\Delta}^{Y^r_\alpha} \left(P_{q_r, q'_r} \left(Q^r_{\alpha i}, \bbW^r_{\alpha i}\right)\right) :=  P_{q_r, q'_r} \left(Y^r_\alpha, \bbW^r_{\alpha}\right) \cap N_{T, d_\Delta } \left(P_{q_r, q'_r} \left(Q^r_{\alpha i}, \bbW^r_{\alpha i}\right)\right),$$
where on the right, the neighborhood is taken inside $P_{q_r, q'_r} \left(X, \bbW^r_{\alpha}\right)$ with its simplicial metric $d_\Delta$.
We claim that for
$r\geqs R_T$,
\begin{equation}\label{N=N'} N_{T, d_\Delta}^{Y^r_\alpha} \left(P_{q_r, q'_r} \left(Q^r_{\alpha i}, \bbW^r_{\alpha i}\right)\right)
= N_{T, d_\Delta'}^{Y^r_\alpha} \left(P_{q_r, q'_r} \left(Q^r_{\alpha i}, \bbW^r_{\alpha i}\right)\right) 
\end{equation}
for all $\alpha\in A_r$, $i\in I(r, \alpha)$.  
It follows easily from the  definitions that 
$$N_{T, d_\Delta'}^{Y^r_\alpha} \left(P_{q_r, q'_r} \left(Q^r_{\alpha i}, \bbW^r_{\alpha i}\right)\right) 
\subset N_{T, d_\Delta}^{Y^r_\alpha} \left(P_{q_r, q'_r} \left(Q^r_{\alpha i}, \bbW^r_{\alpha i}\right)\right).$$
Now say $x\in N_{T, d_\Delta}^{Y^r_\alpha} \left(P_{q_r, q'_r} \left(Q^r_{\alpha i}, \bbW^r_{\alpha i}\right)\right)$.  
Then there is a piecewise geodesic path $\gamma$ in $P_{q_r, q_r'} (X, \bbW^r_\alpha)$, of length less than $T$, from $x$ to $P_{q_r, q'_r} \left(Q^r_{\alpha i}, \bbW^r_{\alpha i}\right)$.  We claim that $\gamma$ lies inside $P_{q_r, q'_r} \left(X, \bbW^r_{\alpha i}\right)$ (which will imply, in particular, that $x\in N_{T, d'_\Delta}  \left(P_{q_r, q'_r} \left(Q^r_{\alpha}, \bbW^r_{\alpha i}\right)\right)$).  If not, then for some $t\in [0,1]$ and some $j\in I(r, \alpha)$ with $i\neq j$, we have $\gamma(t) \in  P_{q'_r} (W)$ for some $W\in \bbW^r_{\alpha j}$.  Then
$$d_\Delta \left(P_{q_r, q'_r} \left(Q^r_{\alpha i}, \bbW^r_{\alpha i}\right), P_{q_r,  q'_r} \left(Q^r_{\alpha j}, \bbW^r_{\alpha j}\right) \right) < T.$$
By Lemma~\ref{rel-metric-comp}, we have
$$d\left(Q^r_{\alpha i} \cup W^r_{\alpha i}, \bigcup_{j\in I(r, \alpha),\, j\neq i} \left(Q^r_{\alpha j} \cup  W^r_{\alpha j}\right)\right)
\leqs (T+2)C_r q_r,$$
contradicting our choice of $R_T$ (note that  $W^r_{\alpha i}$ and $W^r_{\alpha j}$ were defined in (\ref{W})).  A similar argument shows that $x\in P_{q_r, q_r'} (Y^r, \bbW^r_{\alpha i})$, establishing (\ref{N=N'}).  From here on we drop the subscripts $d_\Delta$ and $d_\Delta'$ from the sets in (\ref{N=N'}).

By (\ref{g}) and (\ref{g2}), for $r> R_T$, $\alpha\in A_r$, and $i, j\in I(r, \alpha)$ with $i\neq j$, we have $d(Q^r_{\alpha i}, Q^r_{\alpha j}) > (2T+2) C_rq_r> q_r$.
Hence
$$P_{q_r, q_r'} (Q^r_{\alpha}, \bbW^r_\alpha) = \bigcup_{i\in I(r, \alpha)} P_{q_r, q_r'} (Q^r_{\alpha i}, \bbW^r_{\alpha i}).$$
Together with (\ref{N=N'}), this establishes surjectivity of $\phi_{\bq, \bq'}(T)$.

Bijectivity of $\phi_{\bq, \bq'}(T)$ implies that $\Phi_{\bq, \bq'}(T)$ is bijective on objects: for all $r\geqs R_T$, $\alpha\in A_r$, each module on 
$$N_{T, d_\Delta}^{Y^r_\alpha} \left(P_{q_r, q'_r} (Q^r_\alpha, \bbW^r_\alpha)\right)\cross [0,1) = \coprod_{i\in I(r, \alpha)} N_{T}^{Y^r_\alpha} \left( P_{q_r, q'_r} (Q^r_{\alpha i}, \bbW^r_{\alpha i}) \right)\cross [0,1)$$
 is the direct sum of its restrictions to the disjoint subspaces on the right.

Next we check that $\Phi_{\bq, \bq'} (T) $ is full.  Each morphism $\alpha$  in 
$\bAc^{\Y+} P_{\bq, \bq'} (\Q, \bbW)$
is represented by a morphism $\psi$   in $\Ac (K_T)$ for some $T>0$.
Let $D$ be a bound on the propagation of $\psi$.  Since $g_r/C_r q_r \to \infty$, there exists $S=S(\psi) \geqs R_T$ such that $g_r  > (D+2T+2)C_rq_r $ for $r\geqs S$.
Setting
$$\psi(S)_{ a,b } = \left\{ \begin{array}{ll}
					\psi_{a,b}, \,\,\, a, b\in N_{T}^{Y^r_\alpha} \left(P_{q_r, q'_r} (Q^r_\alpha, \bbW^r_\alpha)\right) \cross [0,1) \textrm{ for some } r\geqs S\\
					 \,\,\,\,\,\,\,\,\,\,  \,\,\,\,\,\,\,\,\,\,  \,\,\,\,\,\,\,\,\,\, \textrm{ and some } \alpha\in A_r,\\
					0, \,\,\,\,\,\,\,\,\,\, else,
				      \end{array}
					\right.
$$  
we have $[\psi] = [\psi(S)]$ as morphisms in $\bAc (K_T)$.
As above, Lemma~\ref{rel-metric-comp} and our choice of $S$ imply that for $r>S$, $\alpha\in A_r$, and $i,j\in I(r, \alpha)$ with $i\neq j$,
\begin{equation}\label{D+2T}d_\Delta \left(P_{q_r, q'_r} \left(Q^r_{\alpha i}, \bbW^r_{\alpha i}\right), P_{q_r,  q'_r} \left(Q^r_{\alpha j}, \bbW^r_{\alpha j}\right)\right) \geqs D+2T.
\end{equation}
Hence $\psi(S)$  is a direct sum, over $r\geqs S$, $\alpha\in A_r$, and $i\in I(r, \alpha)$, of morphisms supported on 
$N_{T}^{Y^r_\alpha} \left(P_{q_r, q_r'} (Q^r_{\alpha i}, \bbW^r_{\alpha i})\right)\cross [0,1).$  When viewed as a morphism between modules on $P_{\bq (S), \bq' (S)} (X, \bbW'(S))\cross [0,1)$, the morphism
$\psi (S)$ still has propagation at most $D$: if  $i\in I(r, \alpha)$ for some $r\geqs S$ and some $\alpha\in A_r$ and there exist points
$$(x,t), (y,s)\in N_{T}^{Y^r_\alpha} \left(P_{q_r, q_r'} (Q^r_{\alpha i}, \bbW^r_{\alpha i})\right)\cross [0,1)$$ with $\psi(S)_{(x,t), (y, s)} \neq 0$, then there exists a simplicial path in $P_{q_r, q_r'} (X, \bbW^r_{\alpha})$ of length at most $D$ connecting $x$ and $y$.  This path must in fact lie in $P_{q_r, q_r'} (X, \bbW^r_{\alpha i})$, since otherwise 
we would have 
$$d_\Delta  \left(P_{q_r, q'_r} \left(Q^r_{\alpha i}, \bbW^r_{\alpha i}\right), P_{q_r,  q'_r} \left(Q^r_{\alpha j}, \bbW^r_{\alpha j}\right)\right) < T+D$$
for some $j\in I(r, \alpha)$ with $j\neq i$, contradicting (\ref{D+2T}).
This shows that $[\psi] = [\psi(S)]$ is in the image of $\Phi_{\bq, \bq'} (T)$.  

Finally,  check that $\Phi_{\bq, \bq'} (T)$ is faithful.  If $\Phi_{\bq, \bq'} (T) ([\psi_1]) = \Phi_{\bq, \bq'} (T) ([\psi_2])$, then for sufficiently large $R_0$, the restrictions of $\psi_1$ and $\psi_2$ to  
$$\coprod_{r>R_0} \coprod_{ \alpha\in A_r} P_{\bq, \bq'} (Q^r_{\alpha}, \bbW^r_\alpha)\cross [0,1)$$ 
are identical, and hence $[\psi_1] = [\psi_2]$.

\end{proof}

\vspace{.1in}
\noindent {\bf Proof of Lemma~\ref{mu-lemma}.}
As in (\ref{decomps}), let $\U'$ and $\V'$ be the decomposed sequences $\U' = (U^1, U^2, \ldots)$  and $\V' = (V^1, V^2, \ldots)$,  with  decompositions
$$U^r = \bigcup_{\alpha\in A_r} \bigcup_{i\in I(r, \alpha)} U^r_{\alpha i} \, \textrm{ and } \, 
V^r = \bigcup_{\alpha\in A_r} \bigcup_{j\in J(r, \alpha)} V^r_{\alpha j},$$ 
respectively.  By  (\ref{db}), we have $\U', \V' \in \D_\beta (X)$. 
 
Given $t>0$, let $N^\Z_{t\bC  \bs} \U'$ denote the decomposed sequence with $r$th term
\begin{equation}\label{NZU'}\bigcup_{\alpha\in A_r} \bigcup_{i\in I(r, \alpha)} Z^r_\alpha \cap N_{tC_r s_r} U^r_{\alpha i}\end{equation}
and with decompositions exactly as shown in (\ref{NZU'}),
and similarly for $\V$ in place of $\U$.
Lemma~\ref{nbhds2} implies that for each $t>0$, $N^\Z_{t\bC  \bs} \U'$ and $N^\Z_{t\bC  \bs} \V'$ are in $\D_\beta (X)$ as well.  By the induction hypothesis, $N^\Z_{t\bC  \bs} \U'$ and $N^\Z_{t\bC  \bs} \V'$ are vanishing sequences.

We will show that for any $\bs', \bs''\in \Seq$ with $\bs \leqs \bs'\leqs \bs''$, the map $\mu_{\bs, \bs', \bs''} = \mu_{\bs, \bs', \bs''} (\U) \oplus \mu_{\bs, \bs', \bs''} (\V)$ factors through the direct sum of the maps
\begin{equation}\label{eta-eq} \mocolim_{t\to \infty} K_* \bAc^X (P_{\bs'} (N^\Z_{t\bC  \bs} \U' )) 
	\xmaps{\eta = \mocolim_t \eta_{\bs', \bs''} (t) } \mocolim_{t\to \infty}   K_* \bAc^X (P_{\bs''} (N^\Z_{t\bC  \bs} \U' )) 
\end{equation}
and
\begin{equation} \label{eta2}  \mocolim_{t\to \infty} K_* \bAc^X (P_{\bs'} (N^\Z_{t\bC  \bs} \V' )) \xmaps{\eta =\mocolim_t \eta_{\bs', \bs''} (t) } \mocolim_{t\to \infty}   K_* \bAc^X (P_{\bs''} (N^\Z_{t\bC  \bs} \V' )).
\end{equation}
Since $N^\Z_{t\bC  \bs} \U'$ and $N^\Z_{t\bC  \bs} \V'$ are vanishing sequences,
the desired result will follow from this factorization.  
We will in fact show that $\mu_{\bs, \bs', \bs''} (\U)$ factors through (\ref{eta-eq}) and $\mu_{\bs, \bs', \bs''} (\V)$ factors through (\ref{eta2}).
From here on we deal only with $\U$; the argument for $\V$ is identical.

For $t> 0$, $r\geqs 1$, $\alpha\in A_r$, $i\in I(r, \alpha)$ and $j\in J(r, \alpha)$, let
$$W^r_{t \alpha i j} = Z^r_\alpha \cap N_{tC_r s_r} (U^r_{\alpha i}) \cap N_{t C_r s_r} (V^r_{\alpha j})$$
and let 
$$W^r_{t \alpha i} = \bigcup_{j\in J(r, \alpha)} W^r_{t \alpha i j}.$$
Furthermore, let $\bbW^r_{t \alpha i }$ denote the metric family
$\{W^r_{t \alpha i j}\}_{j\in J(r, \alpha)}$.

Given $T> 0$, $r\geqs 1$, $\alpha\in A_r$, and $i\in I(r, \alpha)$,   we have
\begin{eqnarray}\label{1} \hspace{.3in}P_{s_r, s_r'} \left(Z^r_\alpha,\bbW^r_{t \alpha i } \right)\cap N_{T} \left( P_{s_r, s'_r} \left(U^r_{\alpha i}, \bbW^r_{t \alpha i } \right)\right)\\
\subseteq
P_{s_r, s_r'} \left(Z^r_\alpha, W^r_{t \alpha i}\right) \cap N_{T} \left( P_{s_r, s'_r} \left(U^r_{\alpha i}, W^r_{t \alpha i} \right)\right),\notag
\end{eqnarray}
where the first neighborhood is taken with respect to the simplicial metric on 
$P_{s_r, s_r'} \left(X, \bbW^r_{t \alpha i } \right)$, and  the second neighborhood is taken with respect to the (smaller) simplicial metric on 
$P_{s_r, s_r'} \left(X, W^r_{t \alpha i}\right)$.
Equation (\ref{rnc2}) in Lemma~\ref{rel-nbhd-comp}, along with the fact that $W^r_{t\alpha i}\subset Z^r_\alpha$, now shows that
\begin{eqnarray} \label{2}P_{s_r, s_r'} \left(Z^r_\alpha, W^r_{t \alpha i} \right) \cap N_{T} \left( P_{s_r, s'_r} \left(U^r_{\alpha i}, W^r_{t \alpha i}\right)\right)\hspace{.4in}\\
\subseteq
P_{s_r'} \left(Z^r_\alpha \cap N_{(T+2) C_r s_r} \left(U^r_{\alpha i} \cup W^r_{t \alpha i}\right)\right),\notag
\end{eqnarray}
where the first neighborhood is taken inside $P_{s_r, s_r'} (X, W^r_{t \alpha i})$.
From the definitions of $W^r_{t \alpha i j}$ and $W^r_{t \alpha i}$, we have 
\begin{equation}\label{3}N_{(T+2) C_r s_r} \left(U^r_{\alpha i} \cup W^r_{t \alpha i} \right) \subset  N_{(T+t+2) C_r s_r} (U^r_{\alpha i} ).
\end{equation}
Combining (\ref{1}), (\ref{2}), and (\ref{3}) yields
\begin{eqnarray}\label{4} P_{s_r, s_r'} \left(Z^r_\alpha, \bbW^r_{t \alpha i } \right)\cap N_{T} \left( P_{s_r, s'_r} \left(U^r_{\alpha i}, \bbW^r_{t \alpha i } \right)\right)\\
 \subset P_{s_r'} \bigg(Z^r_\alpha \cap N_{(T+t+2) C_r s_r} (U^r_{\alpha i} )\bigg)\notag.
 \end{eqnarray}
 
Let $\bbW'_{t \bC \bs}$ denote the set of   metric families 
$$\bbW'_{t \bC \bs} = \{ \bbW^r_{t \alpha i } : r\geqs 1,\, \alpha\in A_r,\, i\in I(r, \alpha)\},$$
and let $\Z'$ denote the decomposed sequence $\Z' = (V^1, V^2, \ldots)$ with  decompositions
$Z^r = \bigcup_{\alpha\in A_r} \bigcup_{ i \in I(r, \alpha)} Z^r_{\alpha i}$, where $Z^r_{\alpha i} = Z^r_\alpha$ for each $r\geqs 1, \alpha\in A_r$, and $i\in I(r, \alpha)$.
For each $t, T>0$,  we define
$$N_T^{\Z}  \left( P_{\bs, \bs'} \left(\U', \bbW'_{t \bC \bs} \right)\right) := P_{\bs, \bs'}  (\Z', \bbW'_{t \bC \bs}) \cap N_T \left(P_{\bs, \bs'} \left(\U', \bbW'_{t \bC \bs} \right)\right),$$
where on the right,  the neighborhood  is taken inside the larger complex $P_{\bs, \bs'} \left(X, \bbW'_{t \bC \bs} \right)$ (with respect to the simplicial metric on $P_{\bs, \bs'} \left(X, \bbW'_{t \bC \bs} \right)$).   We give $N_T^{\Z}  \left( P_{\bs, \bs'} \left(\U', \bbW'_{t \bC \bs} \right)\right)$  the metric  induced by the simplicial metric on $P_{\bs, \bs'} \left(X, \bbW'_{t \bC \bs} \right)$.
Applying Lemma~\ref{functoriality} to the inclusions (\ref{4}) yields functors
\begin{equation}\label{j}\bAc \left(N_T^{\Z}  \left( P_{\bs, \bs'} \left(\U', \bbW'_{t \bC \bs} \right)\right) \right)
\xmaps{j_{t,T}} \bAc^X \left( P_{\bs'} \left(N^\Z_{(t+T+2)\bC\bs} (\U')\right)\right)
\end{equation}

The colimit, over $T>0$, of the categories appearing in the domain of $j_{t,T}$ is precisely
$\bAc^{\Z' +} \left(P_{\bs, \bs'} \left(\U', \bbW'_{t \bC \bs} \right)\right)$.  Hence the functors $j_{t,T}$ combine to yield a functor
$$ \bAc^{\Z' +} \left(P_{\bs, \bs'} \left(\U', \bbW'_{t \bC \bs} \right)\right)
\xmaps{j_t = \colim_{T} j_{t, T}} \mocolim_{t\to \infty}  \bAc^X \left(P_{\bs'} \left(N^\Z_{t\bC\bs} (\U')\right) \right).$$
For each $r\geqs 1$, $\alpha\in A_r$, and $i\in I(r, \alpha)$ we have 
$$U^r_{\alpha i} \cup  W^r_{t  \alpha  i} \subset N_{tC_rs_r} (U^r_{\alpha i}),$$ 
and the families 
$$\{N_{tC_rs_r} (U^r_{\alpha i}) : i\in I(r, \alpha)\}$$ 
are $(C_r s_r r - 2tC_r s_r)$--disjoint (by (\ref{decomps})).  Since $(r-2t)C_r s_r/C_r s_r = r-2t$ tends to infinity with $r$, 
Lemma~\ref{rel-refinement} tells us that for each $t > 0$ there is an isomorphism
$$\Psi_{\bs, \bs'}(t) \co K_* \bAc^{\Z' +} \left(P_{\bs, \bs'} \left(\U', \bbW'_{t \bC \bs} \right)\right)  \srm{\isom} 
K_*  \bAc^{\Z +} \left(P_{\bs, \bs'} \left(\U, \bbW_{t \bC \bs} \right)\right).$$

The desired factorization of $\mu_{\bs, \bs', \bs''} (\U)$ is obtained by composing the isomorphism $\mocolim_{t\to\infty} \Psi_{\bs, \bs'}(t)^{-1}$ 
with the composite 
\begin{eqnarray*} \mocolim_{t\to \infty} K_* \bAc^{\Z' +} \left(P_{\bs, \bs'} \left(\U', \bbW'_{t \bC \bs} \right)\right)
\xmaps{\colim_t (j_{t })_*} \mocolim_{t\to \infty} K_* \bAc^X \left(P_{\bs'} \left(N^\Z_{t\bC\bs} (\U')\right) \right)\\
\srm{\eta_*} \mocolim_{t\to \infty} K_* \bAc^X \left(P_{\bs''} (N^\Z_{t\bC\bs} (\U')) \right)
\srm{\Psi_{\bs''}}  \mocolim_{t\to \infty} K_* \bAc^X \left(P_{\bs''} (N^\Z_{t\bC\bs} (\U))\right)\\
\srm{k_*} \mocolim_{t\to \infty} K_* \bAc^{\Z +} \left(P_{\bs''} (N^\Z_{t\bC\bs} (\U)) \right),
\end{eqnarray*}
where $\eta$ is the functor from (\ref{eta-eq}), $\Psi_{\bs''}$ is the colimit (over $t$) of the isomorphisms from Lemma~\ref{refinement}, and $k_*$ is induced by the colimit of the inclusions
$$\Ac^X \left(P_{\bs''} (N^\Z_{t\bC\bs} (\U)) \right) \subset \Ac^{\Z +} \left(P_{\bs''} (N^\Z_{t\bC\bs} (\U)) \right).$$

To show that this composite agrees with $\mu_{\bs, \bs', \bs''} (\U)$, we examine Diagram (\ref{commutes}), whose terms are explained below. 
The dotted arrows in Diagram (\ref{commutes}) exist only after passing to $K$--theory.  The maps labelled injective are inclusions of one term into a colimit. 
To save space, we have written $\bbW_{t_0}$ and $\bbW'_{t_0}$ rather than $\bbW_{t_0 \bC\bs}$ and $\bbW'_{t_0 \bC\bs}$, and we have written $\bigcup_t$ rather than $\colim_t$.  In the upper left corner, 
$$P^R_{\bs, \bs'} (\U', \bbW'_{t_0}) := P_{\bs(R), \bs'(R)} (\U'(R), \bbW_{t_0}'(R)),$$
and the other superscripts on the Rips complexes should be interpreted similarly.
The isomorphisms labelled $i$ are those from Lemma~\ref{>R}. 
 Furthermore, we have set $\tau = t_0 + T + 2$.   

\begin{equation}\label{commutes}
\xymatrix{ \bAc N^\Z_{T} P^{R}_{\bs, \bs'} (\U', \bbW'_{t_0}) \ar[rr]^{\Phi^{t_0}_{\bs, \bs'} (T)}_\isom \ar[ddd]^{j_{t_0, T} (R)} \ar[dr]^i_\isom
	& & \bAc N^\Z_{T} P^{R}_{\bs, \bs'} (\U, \bbW_{t_0})\ar[d]^i_\isom \\
 & \bAc N^\Z_T P_{\bs, \bs'} (\U', \bbW'_{t_0}) \ar@{^{(}->}[d]  \ar@{.>}[r]^-{\Psi^{t_0}_{\bs, \bs'} (T)}_-\isom  & \bAc N^\Z_{T} P_{\bs, \bs'} (\U, \bbW_{t_0}) \ar@{^{(}->}[d]\\
& \bAc^{\Z' +} P_{\bs, \bs'} (\U', \bbW'_{t_0})\ar[d]^{j_{t_0}} \ar@{.>}[r]^-{ \Psi^{t_0}_{\bs, \bs'}}_-\isom  & \bAc^{\Z +} P_{\bs, \bs'} (\U, \bbW_{t_0})\ar[d]^{\beta_{t_0}}\\
 \bAc P^{R_{\tau}}_{\bs'} N^\Z_{\tau\bC\bs} (\U') \ar[r] \ar@{^{(}->}[d] &\bigcup_t \bAc P_{\bs'}  N^\Z_{t \bC\bs} (\U')  \ar[d]^\eta
 	& \bAc^{\Z +} P_{\bs'} N^\Z_{t_0\bC\bs} (\U) \ar@{^{(}->}[d]^{\gamma_{t_0}}\\
\bigcup_t \bAc P_{\bs''}^{R_t} N^\Z_{t\bC\bs } (\U') \ar[r]^i_\isom \ar[d]^{\Phi_{\bs''}}_\isom & \bigcup_t \bAc P_{\bs''} N^\Z_{t\bC\bs } (\U') 
	\ar@{.>}[d]^{\Psi_{\bs''}} & \bigcup_t \bAc^{\Z +} P_{\bs'} N^\Z_{t\bC\bs } (\U)\ar[d]^\delta\\
\bigcup_t \bAc P_{\bs''}^{R_t} N^\Z_{t\bC\bs } (\U) \ar[r]^i_\isom	& \bigcup_t \bAc P_{\bs''} N^\Z_{t \bC\bs} (\U ) \ar[r]^-k
	&\bigcup_t  \bAc^{\Z +} P_{\bs''} N^\Z_{t\bC\bs } (\U).
}
\end{equation}

Diagram (\ref{commutes}) exists  for each $t_0, T>0$, in the sense that we may choose natural numbers $R = R(T, t_0)$ and $R_t$ (for each $t>0$) such that all the maps exist.  Specifically, for each $t>0$, choose $R_t$ large enough that $\Phi_{\bs''}$ exists (where $\Phi_{\bs''}$ is the colimit over $t$ of the maps constructed in the proof of Lemma~\ref{refinement}) and then choose $R\geqs R_{\tau}$ large enough that   $\Phi^{t_0}_{\bs, \bs'} (T)$ exists, where $\Phi^{t_0}_{\bs, \bs'} (T)$ is the map constructed in the proof of Lemma~\ref{rel-refinement}.  After passing to $K$--theory, the proof of Lemma~\ref{rel-refinement} also gives the maps 
$\Psi^{t_0}_{\bs, \bs'} (T)$ and $ \Psi^{t_0}_{\bs, \bs'}$ appearing on the right-hand side of the diagram.  It follows from the definitions of these maps that the squares having these maps as their horizontal sides are commutative (after passing to $K$--theory).  Similarly, the square in the lower left corner is commutative after passing to $K$--theory.
Since $R\geqs R_\tau$, we can define $j_{t_0, T} (R)$ in analogy with $j_{t_0, T}$, so that the trapezoid on the left of the diagram commutes.  The map $\beta_{t_0}$ is induced by the inclusion
$$P_{\bs, \bs'} (\U, \bbW_{t_0})\subset P_{\bs'} N^\Z_{t_0\bC\bs} (\U),$$
and $\delta = \mocolim_{t\to\infty} \eta_{\bs', \bs''} (N^\Z_{t\bC\bs})$.

Commutativity of the undotted portion of the diagram follows quickly from the definitions of the functors involved.  For instance, the outer square commutes because the maps involved do not change the underlying data of geometric modules or morphisms.  
Commutativity of the lower right rectangle (after passing to $K$--theory) now follows from the fact that in the upper left corner of this rectangle, $K_* \left(\bAc^{\Z +} P_{\bs, \bs'} (\U', \bbW'_{t_0})\right)$ is (isomorphic to) the  colimit over $T>0$ of the $K$--theories of the categories
$\bAc N^\Z_{T} P^{R}_{\bs, \bs'} (\U', \bbW'_{t_0})$ appearing in the upper left-hand corner of the diagram.

By definition, the map $\mu_{\bs, \bs', \bs''} (\U)$ is obtained from $\delta\circ\gamma_{t_0}\circ\beta_{t_0}$
by passing to the colimit (over $t_0$) in the domain (and then applying $K$--theory).  Commutativity of the lower right rectangle (after passing to $K$--theory) shows that 
$$\delta\circ\gamma_{t_0}\circ\beta_{t_0} = (k\circ \Psi_{\bs''} \circ \eta \circ j_{t_0})_* \circ  \Psi_{\bs, \bs'} (t_0)^{-1},$$
and taking colimits over $t_0$ gives the claimed factorization of $\mu_{\bs, \bs', \bs''} (\U)$.
$\hfill \Box$


\section{Assembly for FDC groups}$\label{assembly-sec}$

In this section, we apply our vanishing result for continuously controlled $K$--theory (Theorem~\ref{vanishing-thm}) to study assembly maps.  We first prove a large-scale, bounded version of the Borel Conjecture, analogous to Guentner--Tessera--Yu~\cite[Theorems 4.3.1, 4.4.1]{GTY-rigid}, relating the bounded $K$--theory of the Rips complexes on an FDC metric space to an associated homology theory.  Then we study the classical $K$--theoretic assembly map, using Carlsson's descent argument~\cite{Carlsson-assembly}.

\begin{theorem}$\label{bdd-Borel}$
Let $X$ be a bounded geometry metric space with finite decomposition complexity.  Then there is an isomorphism
$$\mocolim_{s\to \infty}  H_* (P_s(X); \mathbb{K} (\A)) \isom \mocolim_{s\to \infty} K_* (\Ab (P_s (X))).$$
\end{theorem}

  This result may be thought of as excision statement for bounded $K$--theory.  
  Before giving the proof, we need some setup.
For a proper metric space $X$, let $\Ac(X)_{<1}$ denote the 
full additive subcategory of $\Ac (X)$ on those modules $M$ whose support has no limit points at 1; that is, 
$$\overline{\supp(M)} \cap (X\cross 1) = \emptyset,$$ where the closure $\overline{\supp(M)}$ is taken in $X\cross [0,1]$.
By an argument similar to the proof of Lemma~\ref{filtrations}, the inclusion of categories
$$\Ac ( X )_{<1} \subset \Ac ( X )$$
admits a Karoubi filtration.

\begin{definition}
The Karoubi quotient $\Ac (X)/ \Ac ( X )_{<1}$ is denoted $\Ainf (X)$.
\end{definition}

Theorem~\ref{LES} yields a long exact sequence in non-connective $K$--theory 
\begin{eqnarray}\label{assembly-LES}
\cdots \srm{\partial} K_{*} \Ac (X)_{<1} \maps K_* (\Ac (X))\hspace{1.4in}\\
\hspace{1.3in} \maps K_* \Ainf(X)     \srm{\partial} K_{*-1} \Ac (X)_{<1} \maps\cdots.\notag
\end{eqnarray}
As shown by Weiss~\cite{Weiss}, $K_{*} (\Ainf (-))$ is the (Steenrod) homology theory associated to the non-connective algebraic $K$--theory spectrum $\bbK (\A)$, with a dimension shift: in particular, if $X$ is a finite CW complex, there are isomorphisms
\begin{equation}\label{PW} K_{*} (\Ainf (X)) \isom H_{*-1} (X; \bbK (\A))\end{equation}
for each $*\in \bbZ$ (this result was first proven, in a slightly different form,  in Pedersen--Weibel~\cite{Pedersen-Weibel-homology}).   The two key components of Weiss's proof are the facts that the functor $X\goesto  K_{*} (\Ainf (X))$ is homotopy invariant and satisfies excision.  The methods of Weiss and Williams~\cite{Weiss-Williams} then show that $\Ainf (X)\heq X_+ \sm \Ainf (*)$ (at least for $X$ an ENR, and in particular for $X$ a finite CW complex).  One then identifies the coefficients $\Ainf(*)$ by observing that $K_{*} (\Ainf (\{*\}))$ is isomorphic to $K_{*-1} (\Ac(\{*\})_{<1})$, since the other terms in the long exact sequence~(\ref{assembly-LES}) vanish when $X = *$ (see, for example Bartels~\cite[Remark 3.20]{Bartels}), and  $K_{*-1} (\Ac(\{*\})_{<1})\isom \K_{*-1} \A$ by Lemma~\ref{bounded} below.   Details can be found in the above references; see~\cite[Section 5]{Weiss} in particular.  

\begin{remark} Weiss~\cite{Weiss} uses a somewhat different description of the category $\Ainf (X)$.  He describes the morphisms as  ``germs" of morphisms in $\Ac(X)$.  It is easy to check, however, that Weiss's germ category is the same as the Karoubi quotient $\Ainf (X)$.  Additionally, Weiss works with the idempotent completion of his germ category.  This does not affect the results though, since the non-connective $K$--theory spectrum of an additive category $\A$ is weakly equivalent to that for its idempotent completion $\A^\wedge$: this follows from Pedersen--Weibel~\cite[Lemmas 1.4.2 and 2.3]{Pedersen-Weibel}. 
\end{remark}

\begin{lemma} $\label{bounded}$ For every proper metric space $X$ there is an equivalence
$$\Ab (X)\srm{\isom}\Ac(X)_{<1}.$$
\end{lemma}
\begin{proof}  This equivalence is induced by the inclusion of categories 
$$\Ab(X) = \Ab(X\cross\{0\}) \subset \Ac(X)_{<1},$$
which is clearly bijective on Hom sets in the domain.  We need to check that every object in $\Ac(X)_{<1}$ is isomorphic to an object in $\Ab(X)$.  Given a module $M\in \Ac(X)_{<1}$, let 
 $\overline{M} \in \Ab (X)$ be the module
$$\overline{M}_x = \bigoplus_{t\in [0,1)} M_{(x, t)}.$$
Since objects in $\Ab(X\cross [0,1))_{<1}$ stay away from 1, $\overline{M}$ is finitely generated at each point, and properness of $X$ implies that $\overline{M}$ is locally finite.  
We now have an isomorphism $M\to \overline{M}$ sending $M_{(x, t)}$ isomorphically to the corresponding summand of $\overline{M}_x$.  This morphism has propagation at most 1, and is continuously controlled due to the support condition on $M$.
\end{proof}

\noindent {\bf Proof of Theorem~\ref{bdd-Borel}}  For each $s$, the isomorphisms given by (\ref{PW}) and Lemma~\ref{bounded} show that the long exact sequence (\ref{assembly-LES}) has the form
$$\cdots \to K_* (\Ac (P_s (X)))\to H_{*-1} (P_s(X); \mathbb{K} (\A)) \srt{\partial} K_{*-1} \Ab (P_s (X))\to \cdots.$$
 Since directed colimits preserve exact sequences and the $K$--theory of the category
$\mocolim_s \Ac (P_s X)$ vanishes (Theorem~\ref{vanishing-thm}), the colimit (over $s$) of the boundary maps for this sequence yields the desired isomorphism. $\hfill \Box$

\vspace{.2in}

We now begin the preparations for the proof of our main result, Theorem~\ref{assembly}.
The proof requires some preliminaries regarding group actions and the ``forget-control" description of the assembly map.
Let $X$ be a proper metric space with an isometric action of a group $\Gamma$.  Then $\Gamma$ acts on $\Ac (X)$ through additive functors (given by translating modules and morphisms), and this action maps the subcategory $\Ac(X)_{<1}$ into itself.  It follows from the definitions that the inclusion of fixed point categories $\Ac ( X )_{<1}^\Gamma \subset \Ac ( X )^\Gamma$ admits a Karoubi filtration.
We now have  a Karoubi sequence 
\begin{equation} \label{fp} \Ac ( X )_{<1}^\Gamma \subset \Ac ( X )^\Gamma \maps \left(\Ac ( X )^\Gamma \right)/\left(\Ac ( X )_{<1}^\Gamma \right).
\end{equation}
When $\Gamma$ acts freely and cocompactly on $X$, one may check that there is an equivalence of categories 
$$\Ac(X)_{<1}^\Gamma \isom \A[\Gamma]_c (X/\Gamma)_{<1};$$
note that by compactness, modules in $\A[\Gamma]_c (X/\Gamma)_{<1}$ have \e{finite} support and hence all morphisms in $\A[\Gamma]_c (X/\Gamma)_{<1}$ lift to bounded morphisms on $X\cross [0,1)$.
When $\A$ is the category of finitely generated free $R$--modules for some ring $R$, $\A[\Gamma]$ is the category of finitely generated free $R[\Gamma]$--modules. 
If $\Gamma$ acts properly discontinuously, there is also an equivalence of categories 
$$\left(\Ac ( X )^\Gamma \right)/\left(\Ac ( X )_{<1}^\Gamma \right) \isom \Ainf(X/\Gamma)$$ 
(this is essentially Carlsson--Pedersen~\cite[Lemma 2.8]{Carlsson-Pedersen}), and   (\ref{PW}) yields 
$$K_* \left(\left(\Ac ( X )^\Gamma \right)/\left(\Ac ( X )_{<1}^\Gamma \right)\right) \isom H_{*-1} (X/\Gamma ; \bbK \A).$$ 
The boundary map for the long exact sequence in $K$--theory associated to (\ref{fp}) now has the form
\begin{equation}\label{assembly7}H_{*} (X/\Gamma ; \bbK \A) \maps K_* \left(\A[\Gamma]_c (X/\Gamma)_{<1}\right).\end{equation}

The following lemma identifies the codomain of this map in the case of interest to us.

\begin{lemma} $\label{bounded2}$ If $K$ is a compact metric space with $\diam(K)<\infty$ and $\mathcal{E}$ is an additive category, then there are equivalences of categories
\begin{equation}\label{A}\mathcal{E}_c (K)_{<1}\srm{\isom} \mathcal{E}_b (K) \isom\mathcal{E}. \end{equation}
\end{lemma}
\begin{proof}  The first equivalence is given by Lemma~\ref{bounded}.
Given $x_0\in K$, the second equivalence is induced by the inclusion of categories 
$$\mathcal{E}\isom \mathcal{E}_b (\{x_0\}) \subset \mathcal{E}_b (K).$$
This inclusion is an equivalence because compactness implies that any locally finite module $M$ over $K$ is in fact supported on a finite set $S\subset K$, and is isomorphic to the module $\bigoplus_{x\in S} M_x$ considered as a module over $\{x_0\}$ (this isomorphism has finite propagation because $\diam(K) < \infty$).  
\end{proof}

Under the isomorphism induced by (\ref{A}), the map (\ref{assembly7})
agrees with the classical assembly map
$$H_{*} (X/\Gamma ; \bbK \A) \maps K_* \left(\A[\Gamma]\right).$$
(For proofs, see~\cite{Carlsson-Pedersen, Hambleton-Pedersen, Sperber, Weiss}.)
The boundary map for a fibration sequence of spectra can be realized (up to homotopy) as a map of spectra after looping the base spectrum, so we have a map
\begin{equation}\label{bdry-omega}\Omega \bbK \Ainf (X) \maps \bbK \Ac ( X )_{<1}\end{equation}
that induces the assembly map after taking fixed-point spectra and then homotopy groups.
(We are using the fact that if $\C$ is an additive category with an action of a group $G$ by additive functors,
then $\bbK (\C)^G \isom \bbK(\C^G)$.)  

\begin{remark}$\label{hofib}$ To be precise, the domain of (\ref{bdry-omega}) should be replaced by the \e{homotopy fiber} of the map $\bbK \Ac ( X )_{<1} \srt{i} \bbK \Ac (X)$; then the natural map $\hofib(i)\to \bbK \Ac ( X )_{<1}$ is $\Gamma$--equivariant and induces the boundary map on homotopy groups.  Moreover, since we are dealing with $\Omega$--spectra, the homotopy fiber can be formed level-wise and one finds that $\hofib(i)^\Gamma = \hofib(i^\Gamma)$, where $i^\Gamma$ is the restriction of $i$ to the fixed point spectra.
\end{remark}

The key ingredient in the proof of Theorem~\ref{assembly} will be a variation on Theorem~\ref{vanishing-thm}.  First, we need a simple lemma about homotopically finite classifying spaces of groups.  Note that up to homotopy, there is no difference between assuming that a group admits a finite CW model for $B\Gamma$ or a finite simplicial complex model, because every finite CW complex is homotopy equivalent to a finite simplicial complex.  We have the following lemma.

\begin{lemma}$\label{unif-contractible}$
If $E\Gamma\to B\Gamma$ is a universal principal bundle with $B\Gamma$ a finite simplicial complex, then the simplicial metric $d_\Delta$ on $E\Gamma$ (corresponding to the simplicial structure lifted from $B\Gamma$) is proper
and 
$E\Gamma$ is uniformly contractible with respect to $d_\Delta$.
\end{lemma}

The statement about uniform contractibility is a special case of Bartels--Rosenthal~\cite[Lemma 1.5]{Bartels-Rosenthal}.  Properness follows from the fact that $E\Gamma$ is a locally finite simplicial complex (this is similar to the proof of Lemma~\ref{metric-comp}).

\begin{theorem}$\label{controlled-vanishing-EG}$\
Let $\Gamma$ be a group with finite decomposition complexity, and assume that there exists a universal principal $\Gamma$--bundle $E\Gamma\to B\Gamma$ with $B\Gamma$ a finite simplicial complex (this implies, in particular, that $\Gamma$ is finitely generated).  Equip $E\Gamma$ with the simplicial metric corresponding to the simplicial structure lifted from $B\Gamma$.
Then the category $\Ac (E\Gamma)$ has trivial $K$--theory.
\end{theorem}
\begin{proof}  This is  similar to the proofs of \cite[Lemma 4.3.6]{GTY-rigid} and~\cite[Lemma 4.4]{Bartels-Rosenthal}.
We will construct continuous, proper, metrically coarse maps 
$$f_s\co E\Gamma\to P_s \Gamma, \,\,\,\,\,\,\, g_s\co P_s \Gamma \to E\Gamma$$
 for all sufficiently large $s$, having the property that each composition
$$E\Gamma \srm{f_s} P_s \Gamma \stackrel{i}{\injects} P_{s'} \Gamma \srm{g_{s'}} E\Gamma$$
induces the identity on $K_* \Ac (E\Gamma)$.  This suffices, since given any element 
$x\in K_* \Ac (E \Gamma)$, Theorem~\ref{vanishing-thm} guarantees that  we can choose $s'$ large enough that $i_*(f_s)_* (x) = 0$ in $K_* \Ac (P_{s'} \Gamma)$; now $x = (g_{s'})_* i_*(f_s)_* (x) = 0$. 

Fix a vertex $x_0\in E\Gamma$ and consider the embedding $\Gamma \injects E\Gamma$, $\gamma\goesto \gamma\cdot x_0$.  The action of $\Gamma$ on $E\Gamma$ by deck transformations restricts to left multiplication on $\Gamma$, so
the simplicial metric $d_\Delta$ on $E\Gamma$ restricts to a proper, left-invariant metric $d_\Delta$ on $\Gamma$.  If we equip $\Gamma$ with the left-invariant metric $d_w$ associated to a finite generating set, then for each $R > 0$ there exists $S>0$ such that $d_\Delta (\gamma, \gamma') < R$ implies $d_w (\gamma, \gamma') < S$.  
In particular, letting $D$ denote the diameter of $B\Gamma = E\Gamma/\Gamma$, there exists $s>0$ such that $d_\Delta (\gamma, \gamma') < 2(D+1)$ implies $d_w (\gamma, \gamma') < s$. 
By choice of $D$, the sets
$$U_\gamma = B_{D+1} (\gamma\cdot x_0) \setminus \{\gamma'\cdot x_0 \, :\, \gamma'\neq \gamma\}.$$ 
($\gamma\in \Gamma$)  form an open cover of $E\Gamma$.
If $\{\phi_\gamma\}_{\gamma\in \Gamma}$ is a partition of unity subordinate to this cover,
we can define $f_s \co E\Gamma\to P_s \Gamma$ by the formula
$$f_s (x) = \sum_{\gamma\in \Gamma} \phi_\gamma (x) \gamma,$$ 
Note that $f_s (x)$ is a well-defined point in $P_s \Gamma$, by our choice of $s$.
For each $\gamma \in \Gamma$, we have $\phi_\gamma (\gamma\cdot x_0) = 1$ and hence $f_s (\gamma\cdot x_0) = \langle x_0 \rangle$.  

The maps $g_s\co P_s \Gamma\to E\Gamma$ ($s = 0, 1, \ldots$) are defined by induction over the simplices in $P_s \Gamma$.  When $s = 0$, $P_0 \Gamma = \Gamma$ and $g_0$ is just the embedding $\gamma \goesto \gamma\cdot x_0$.  Now assume that $g_{s-1}$ has been defined ($s>0$).  Let $P_s^{(k)} \Gamma$ denote the $k$--skeleton of $P_s \Gamma$.  Viewing $P_{s-1} \Gamma$ as a subcomplex of $P_s \Gamma$, we extend $g_{s-1}$ inductively over the subcomplexes $P_s^{(k)} \Gamma \cup P_{s-1} (\Gamma)$.
Assuming $g_s$ has been defined on the $P_s^{(k-1)} \Gamma \cup P_{s-1} (\Gamma)$ for some $k\geqs 1$, we extend over a $k$--simplex $\sigma \notin P_{s-1} \Gamma$ as follows.  Let $D = \diam (g_s (\partial \sigma))$ and choose $x\in g_s (\partial \sigma)$.  By uniform contractibility of $E\Gamma$ (Lemma~\ref{unif-contractible}) there exists $D'>0$ (depending only on $D$) and a nullhomotopy of $g_s|_{\partial \sigma}$ whose image lies inside $B_{D'} (x)$.  We now extend $g_s$ over $\sigma$ using this nullhomotopy.

One may now check that $f_s$ and $g_s$ are inverse coarse equivalences, hence metrically coarse and proper (since $E\Gamma$ and $P_s \Gamma$ are proper).  

To show that $g_{s} \circ f_s$ induces the identity map on continuously controlled $K$--theory, it suffices to show that this map is Lipschitz homotopic to the identity~\cite[Proposition 3.17]{Bartels}, where a Lipschitz homotopy $H\co X\cross I\to Y$ (with $X$ and $Y$ metric spaces) is simply a continuous, metrically coarse map for which $\{x\in X \, :\, H(x, t) \in C \textrm{ for some } t\in I\}$ is compact for all compact sets $C\subset Y$.  Following Bartels--Rosenthal~\cite[Lemma 4.4]{Bartels-Rosenthal}, one constructs a homotopy $H\co E\Gamma\cross I\to E\Gamma$ connecting $g_s\circ f_s$ to $\Id_{E\Gamma}$ by induction over the skeleta of $E\Gamma\cross I$, again using the uniform contractibility of $E\Gamma$.  (Here it is most convenient to use the cell structure on $E\Gamma\cross I$ in which cells are either of the form $\sigma\cross \{0\}$, $\sigma\cross \{1\}$, or $\sigma \cross I$, with $\sigma$ a simplex in $E\Gamma$.)

To see that $H$ is metrically coarse, note that its restriction to the zero skeleton of $E\Gamma\cross I$ is the disjoint union of $g_s f_s$ and $\Id_{E\Gamma}$, hence is metrically coarse.  Assuming $H$ is metrically coarse on the $k$--skeleton, one checks metric coarseness on the $(k+1)$--skeleton using the fact that there is a uniform bound $D(k)$ on the diameter of $H(\sigma)$ for $\sigma$ a $k$--simplex (note that for $1$--simplices, this follows from the fact that $g_s f_s$ is a bounded distance from the identity).  For the remaining condition, it suffices to check that 
$$\{x \, :\, d(H(x, t), \gamma\cdot x_0) < R \textrm{ for some } t\in I\}$$
is compact  for each $\gamma\in \Gamma$, $R>0$.  This is similar: if $x$ lies in a $k$--simplex, then $d(H(x,t), g_s f_s (x)) \leqs D(k)$ and $d(g_s f_s (x), x)\leqs S$ (for some constant $S$ independent of $x$), so if $d(H(x, t), \gamma\cdot x_0) < R$, we have $d(x, \gamma \cdot x_0) < S+D(k) + R$, which suffices.
\end{proof}

Theorem~\ref{assembly} can now be proven exactly as in Bartels' proof for groups with finite asymptotic dimension~\cite[Theorems 5.3 and 6.5]{Bartels}.  For convenience of the reader, we recall the argument.

\vspace{.2in}
\noindent {\bf Proof of Theorem~\ref{assembly}.} As explained above, the assembly map 
$$H_* \left(B\Gamma; \bbK (\A[\Gamma])\right) \maps K_* \A[\Gamma]$$
can be realized (up to homotopy) as the map of fixed-point spectra
\begin{equation}\label{assembly-map}\left(\Omega \bbK \left(\Ainf (E\Gamma)\right)\right)^\Gamma 
\srm{\partial^\Gamma} \left(\bbK\left(\Ac (E\Gamma)_{<1}\right)\right)^\Gamma
\end{equation}
associated to a map of spectra 
\begin{equation}\label{boundary}\Omega \bbK \left(\Ainf (E\Gamma)\right) \srm{\partial} \bbK\left(\Ac (E\Gamma)_{<1}\right)
\end{equation}
that induces, on homotopy groups, the $K$--theoretic boundary map for the Karoubi sequence 
\begin{equation}\label{assembly-seq}\Ac (E\Gamma)_{<1} \srm{i} \Ac (E\Gamma) \srm{q} \Ainf (E\Gamma).\end{equation}
Given an $\Omega$--spectrum $Y$ with a level-wise action of a group $G$, let $Y^{h\Gamma}$ denote the homotopy fixed point spectrum; that is, the function spectrum $F^G(EG_+, Y)$ consisting of (unbased) equivariant maps from $EG$ to $Y$.
The map (\ref{assembly-map}) sits in a  commutative diagram
\begin{equation} \label{descent-diag}
\xymatrix{  \left(\Omega \bbK\left(\Ainf (E\Gamma)\right)\right)^\Gamma \ar[r]^{ \partial^\Gamma} \ar[d]^i & \left(\bbK \left(\Ab (E\Gamma)\right)\right)^\Gamma \ar[d]^j\\
\left(\Omega\bbK \left(\Ainf (E\Gamma)\right)\right)^{h\Gamma} \ar[r]^{\partial^{h\Gamma}} & \left(\bbK \left(\Ab (E\Gamma)\right)\right)^{h\Gamma}.
}
\end{equation}
The fact that $E\Gamma/\Gamma = B\Gamma$ is a finite CW complex implies that $i$ is a weak equivalence of spectra (see, for example, Carlsson--Pedersen~\cite[Theorem 2.11]{Carlsson-Pedersen}).  
Theorem~\ref{controlled-vanishing-EG}, together with the long exact sequence in homotopy associated to the Karoubi sequence
$$\Ac (E\Gamma) \isom \Ac (E\Gamma)_{<1} \injects \Ac (E\Gamma) \maps \Ainf (E\Gamma),$$
shows that the map  (\ref{boundary}) is a weak equivalence.  It follows that the map $\partial^{h\Gamma}$ in Diagram (\ref{descent-diag}) is also a weak equivalence (every $G$--equivariant map between $\Omega$--spectra with $G$--actions that is a weak equivalence, in the usual non-equivariant sense, induces a weak equivalence on homotopy fixed point spectra). 
Commutativity of (\ref{descent-diag}) implies that the assembly map $\partial^\Gamma$ in (\ref{assembly-map}) is a split injection on homotopy, with splitting given by $(i_*)^{-1} (\partial^{h\Gamma}_*)^{-1} j_*$.
$\hfill \Box$

\vspace{.2in}
As is usually the case in this area (see Bartels~\cite[Section 7]{Bartels}, for example), Theorem~\ref{assembly} has an analogue for Ranicki's ultimate lower quadratic $L$--theory spectrum $\bbL^{-\infty} (\A)$ of an additive category $\A$ with involution.  

\begin{theorem}$\label{L-theory}$ Let $\Gamma$ be a group with finite decomposition complexity, and assume there exists a universal principal $\Gamma$--bundle $E\Gamma\to B\Gamma$ with $B\Gamma$ a finite CW complex.  Let $\A$ be an additive category with involution, and assume that for some $r>0$ we have $K_r (\mathcal{A}) = 0$ for all $*<-r$.  Then the assembly map
$$H_* (B\Gamma; \bbL^{-\infty} (\A)) \maps \bbL^{-\infty}_* (\A [\Gamma]),$$
is a split injection for all $*\in \bbZ$.
\end{theorem}

The proof is analogous to that of Theorem~\ref{assembly}.  The relevant tools for $L$--theory are provided in Carlsson--Pedersen~\cite[Section 4]{Carlsson-Pedersen}.  The additional condition on $\bbL^{-\infty}_* (\A)$ is needed in order to apply the $L$--theoretic analogue of Carlsson--Pedersen~\cite[Theorem 2.11]{Carlsson-Pedersen} (see~\cite[Theorem 5.5]{Carlsson-Pedersen}).  

\section{Addendum}\label{Addendum}

Here we provide details of the limit ordinal step in the inductive proof of Proposition~\ref{vanishing-prop4}. This argument was missing from previous versions of the article.  The  argument is essentially formal, and does not affect the strategy of the proof. We thank Daniel Kasprowski (private communication) for pointing out the omission and for sharing with us the argument presented below.

The proof of Proposition~\ref{vanishing-prop4} proceeds by transfinite induction on $\gamma$.  It is stated above that ``If  $\gamma$ is a limit ordinal and Proposition 6.11 holds for all $\beta < \gamma$, it follows immediately from the definitions that Proposition 6.11 also holds for $\gamma$."  This is an oversimplification: if $\Z \in \fD_\gamma (X)$, we can not immediately conclude that $\Z\in \fD_\beta$ for some $\beta < \gamma$.  The notation $\Z\in \fD_\gamma (X)$ is highly abusive; it means only that each family $\{Z^r_\alpha\}_{\alpha \in A_r}$ is an element of the set $\fD_\gamma (X) := \bigcup_{\beta < \gamma} \fD_\beta (X)$. Thus $\Z\in \fD_\gamma (X)$ means that for each $r$, there exists $\beta_r < \gamma$ such that $\{Z^r_\alpha\}_{\alpha \in A_r} \in \fD_{\beta_r} (X)$, but it is possible that the least upper bound of the ordinals $\beta_r$ is actually $\gamma$.
We now provide a complete discussion of the limit ordinal step in the proof of Proposition~\ref{vanishing-prop4}.   

To avoid further abuse of notation, from here on we will replace the notation $\Z \in \fD_\eta (X)$ by the statement ``$\Z$ has complexity at most $\eta$."  
Let $\gamma$ be a limit ordinal.
Assume that for all $\beta < \gamma$, every decomposed sequence $\W$ in $X$ with complexity at most $\beta$ is a vanishing sequence.  Let $\Z$ be a decomposed sequence in $X$ with complexity at most $\gamma$.  We must prove that $\Z$ is in fact a vanishing sequence, in the sense that  
\begin{equation}\label{Z-vanish}\colim_{\bs\in \Seq} K_* \bAc (P_\bs (\Z))  = 0.\end{equation}

\vspace{.1in}
\noindent {\bf Notation.}
Given a family of metric spaces $\{W_\alpha\}_{\alpha\in A}$ and a number $s>0$, let 
$P_s (\{W_\alpha\}_{\alpha})$ denote the Rips complex $P_s \left( \coprod_\alpha W_\alpha\right)$.

For each $\bs\in \Seq$, we have the Karoubi sequence
\begin{equation}\label{ses2}\mS_\bs \injects \Ac \left( \coprod_r \left(P_{s_r}\left(\{Z^r_\alpha\}_\alpha\right)\right)\right) \maps   
 \bAc (P_\bs (\Z))
\end{equation}
defining $\bAc (P_\bs (\Z))$.
Applying $K$--theory and passing to the colimit along $\bs\in \Seq$ gives a long exact sequence (note that this is a filtered colimit, so it preserves exactness).  
Hence to prove (\ref{Z-vanish}), it suffices to show that
\begin{equation}\label{S-vanish}
\textrm{ For each $*\in \bbZ$, }\,
\colim_{\bs\in \Seq} K_* \left( \mS_\bs\right) = 0  \end{equation}
and 
\begin{equation}\label{coprod-vanish}  
\textrm{ For each $*\in \bbZ$, }\,
\colim_{\bs\in \Seq} K_*\left( \Ac \left( \coprod_r P_{s_r} \left(\{Z^r_\alpha\}_\alpha\right)\right) \right) = 0.
\end{equation}

To prove (\ref{S-vanish}), note that $\mS_\bs$ is the colimit, over $r<R$, of 
$$\displaystyle{\prod_{r<R}} \Ac \left(P_{s_r}\left(\{Z^r_\alpha\}_\alpha\right)\right).$$ 
For each $r$, we know that
$\{Z^r_\alpha\}_\alpha \in \D_{\beta_r}$ for some $\beta_r < \gamma$.  The space $\coprod_\alpha Z^r_\alpha$ decomposes over the family  $\{Z^r_\alpha\}_\alpha$, so $\coprod_\alpha Z^r_\alpha\in \D_{\beta_r + 1}$.  Since $\beta_r < \gamma$ and $\gamma$ is a limit ordinal, we have $\beta_r + 1 < \gamma$, so our induction hypothesis implies that
\begin{equation*}\colim_{\bs\in \Seq} K_*  \Ac \left(P_{s_r}\left(\{Z^r_\alpha\}_\alpha\right)\right) = 0.\end{equation*}
The desired result now follows from 
the fact that $K$--theory commutes with filtered colimits and with (finite) products. 

\begin{remark}\label{IH}
Strictly speaking, our induction hypothesis  states only that every decomposed sequence of complexity at most $\beta$ $($with $\beta < \gamma$$)$ is a vanishing sequence.  
As explained in Section~\ref{FDC} $($see in particular Diagram (\ref{prod-diag})$)$, it follows that for every \e{space} $W\in \D_{\beta} (X)$ $(\beta < \gamma)$ we have
$$\colim_{\bs\in \Seq} K_* \left(P_s W\right) = 0.$$
\end{remark}

Now we turn to the proof of (\ref{coprod-vanish}).
For each $\bs\in \Seq$,  there is a natural inclusion of categories
$$\Ac \left( \coprod_r P_{s_r} \left(\{Z^r_\alpha\}_\alpha\right)\right) \injects \displaystyle{\prod_r} \Ac \left(P_{s_r}\left(\{Z^r_\alpha\}_\alpha\right)\right),$$
and these inclusions induce a functor
$$j\co \colim_{\bs\in \Seq} \Ac \left( \coprod_r P_{s_r} \left(\{Z^r_\alpha\}_\alpha\right)\right) \injects \colim_{\bs\in \Seq} \displaystyle{\prod_r} \Ac \left(P_{s_r}\left(\{Z^r_\alpha\}_\alpha\right)\right).$$  
We will now define a functor in the opposite direction,
$$C\co \colim_{\bs\in \Seq} \displaystyle{\prod_r} \Ac \left(P_{s_r}\left(\{Z^r_\alpha\}_\alpha\right)\right) \maps \colim_{\bs\in \Seq} \Ac \left( \coprod_r P_{s_r} \left(\{Z^r_\alpha\}_\alpha\right)\right),$$
which will be inverse to $j$.
On objects, this functor is simply induced by the inclusions 
$$\Ac \left(P_{s_r}\left(\{Z^r_\alpha\}_\alpha\right)\right) \injects \Ac \left( \coprod_r P_{s_r} \left(\{Z^r_\alpha\}_\alpha\right)\right),$$
which are compatible as $\bs$ increases.  Given a morphism
$$ (M_r)_r \xmaps{(\phi_r)_r} (N_r)_r$$
in the category
$$\displaystyle{\prod_r} \Ac \left(P_{s_r}\left(\{Z^r_\alpha\}_\alpha\right)\right),$$
 let $D_r < \infty$ be the propagation of $\phi_r$, and let $s'_r  = \max(s_r, D_r)$.  Applying the functor
$$\eta_{s_r, s'_r} \co \Ac \left(P_{s_r}\left(\{Z^r_\alpha\}_\alpha\right)\right) \maps \Ac \left(P_{s'_r}\{Z^r_\alpha\}_\alpha\right),$$
we see that $\eta_{s_r, s_r'} (\phi_r)$ now has propagation at most 3, so $(\eta_{s_r, s_r'} (\phi_r))_r$ is a morphism 
$$\eta_{\bs, \bs'}\left((M_r)_r\right) \maps \eta_{\bs, \bs'} \left((N_r)_r\right)$$
in the category 
$$\Ac \left(\coprod_r P_{s_r} (\{Z^r_\alpha\}_\alpha)\right).$$
Since $\eta_{\bs, \bs'} ((M_r)_r)$ represents $C((M_r)_r)$ and $\eta_{\bs, \bs'}((N_r)_r)$ represents $C((N_r)_r)$, we may define $C((\phi_r)_r)$ to be the morphism represented by $(\eta_{s_r, s_r'} (\phi_r))_r$.   It follows from the definitions that $C$ is well-defined on morphisms and functorial, and also that $C$ and $j$ are inverses.   
Hence (\ref{coprod-vanish}) is equivalent to the statement that
\begin{equation}\label{Prod-vanish} 
\textrm{ For each $*\in \bbZ$, }\,
K_*\left(  \colim_{\bs\in \Seq}   
 \displaystyle{\prod_r} \Ac \left(P_{s_r}\left(\{Z^r_\alpha\}_\alpha\right)\right)\right) = 0.\end{equation}

Examining the definitions, one sees there is an isomorphism of categories
$$ \colim_{\bs\in \Seq}  \displaystyle{\prod_r} \Ac \left(P_{s_r}\left(\{Z^r_\alpha\}_\alpha\right)\right)
\isom \displaystyle{\prod_r}  \colim_{\bs\in \Seq}  \Ac \left(P_{s_r}\left(\{Z^r_\alpha\}_\alpha\right)\right).$$
(Note that on the right, the colimit over $\bs\in \Seq$ may be replaced by a colimit over $s\in \bbN$, since $\Ac \left(P_{s_r}\left(\{Z^r_\alpha\}_\alpha\right)\right)$ depends only on the $r$--term of the sequence $\bs$.)
This yields
\begin{eqnarray*} 
K_*\left( \colim_{\bs \in \Seq}    \displaystyle{\prod_r} \Ac \left(P_{s_r}\left(\{Z^r_\alpha\}_\alpha\right)\right)\right) 
\isom K_*\left(\displaystyle{\prod_r}  \colim_{\bs\in \Seq}  \Ac \left(P_{s_r}\left(\{Z^r_\alpha\}_\alpha\right)\right)\right)\\
\isom
 \displaystyle{\prod_r} K_*  \left( \colim_{\bs\in \Seq}  \left(\Ac \left(P_{s_r}\left(\{Z^r_\alpha\}_\alpha\right)\right)\right)\right),
 \end{eqnarray*}
 where we have used  Carlsson's theorem that $K$--theory commutes with infinite products (for connective $K$--theory this is proven in~\cite{Carlsson-prod}; the proof is extended to non-connective $K$--theory in~\cite{Carlsson-assembly}).  
As discussed above, for each $r$ we have $\coprod_\alpha Z^r_\alpha \in \D_{\beta_r+1}$, and $\beta_r +1 < \gamma$,
so each term in the product 
$$ \displaystyle{\prod_r} K_*  \left( \colim_{\bs\in \Seq}  \left(\Ac \left(P_{s_r}\left(\{Z^r_\alpha\}_\alpha\right)\right)\right)\right) \isom  \displaystyle{\prod_r} \colim_{\bs\in \Seq}   K_*  \left(\Ac \left(P_{s_r}\left(\{Z^r_\alpha\}_\alpha\right)\right)\right)$$
vanishes.  This proves (\ref{Prod-vanish}), and completes the proof of (\ref{Z-vanish}).
 
\vspace{.3in}
\noindent{\bf Comments regarding $L$--theory.}
\vspace{.1in}

The argument above applies equally well to algebraic $L$--theory, if one invokes the theorem of Carlsson and Pedersen~\cite{Carlsson-Pedersen} that $L$--theory commutes with infinite products when coefficient category $\mathcal{A}$ satisfies $K_r (\mathcal{A}) = 0$  for  all $r << 0$.  In particular, this  fills the missing step in the proof of Theorem~\ref{L-theory}.  (We also take this opportunity to note that previous versions of this article contained an  error in the statement of that result; we thank Christoph Winges for pointing out the misstatement.)

It should be noted that our proof of the $L$--theoretic analog of the bounded Borel conjecture (that is, the $L$--theoretic version of Theorem~\ref{bdd-Borel} requires the Carlsson--Pedersen result, so one must again assume that  $K_* (\mathcal{A}) = 0$  for  all $*<<0$.  Since Winges has shown that there are in fact additive categories with involution for which $L$--theory does not commute with infinite products~\cite{Winges}, it would be interesting to know if there is a way to prove the $L$--theoretic bounded Borel conjecture for FDC metric spaces without invoking Carlsson's theorem.


\begin{thebibliography}{10}

\bibitem{ACFP}
Douglas~R. Anderson, Francis~X. Connolly, Steven~C. Ferry, and Erik~K.
  Pedersen.
\newblock Algebraic {$K$}-theory with continuous control at infinity.
\newblock {\em J. Pure Appl. Algebra}, 94(1):25--47, 1994.

\bibitem{BFJR}
Arthur Bartels, Tom Farrell, Lowell Jones, and Holger Reich.
\newblock On the isomorphism conjecture in algebraic {$K$}-theory.
\newblock {\em Topology}, 43(1):157--213, 2004.

\bibitem{Bartels-Rosenthal}
Arthur Bartels and David Rosenthal.
\newblock On the {$K$}-theory of groups with finite asymptotic dimension.
\newblock {\em J. Reine Angew. Math.}, 612:35--57, 2007.

\bibitem{Bartels}
Arthur~C. Bartels.
\newblock Squeezing and higher algebraic {$K$}-theory.
\newblock {\em $K$-Theory}, 28(1):19--37, 2003.

\bibitem{Baum-Connes}
Paul Baum and Alain Connes.
\newblock {$K$}-theory for discrete groups.
\newblock In {\em Operator algebras and applications, {V}ol.\ 1}, volume 135 of
  {\em London Math. Soc. Lecture Note Ser.}, pages 1--20. Cambridge Univ.
  Press, Cambridge, 1988.

\bibitem{BHM}
M.~B{\"o}kstedt, W.~C. Hsiang, and I.~Madsen.
\newblock The cyclotomic trace and algebraic {$K$}-theory of spaces.
\newblock {\em Invent. Math.}, 111(3):465--539, 1993.

\bibitem{Cardenas-Pedersen}
M.~C{\'a}rdenas and E.~K. Pedersen.
\newblock On the {K}aroubi filtration of a category.
\newblock {\em $K$-Theory}, 12(2):165--191, 1997.

\bibitem{Carlsson-assembly}
Gunnar Carlsson.
\newblock Bounded {$K$}-theory and the assembly map in algebraic {$K$}-theory.
\newblock In {\em Novikov conjectures, index theorems and rigidity, {V}ol.\ 2
  ({O}berwolfach, 1993)}, volume 227 of {\em London Math. Soc. Lecture Note
  Ser.}, pages 5--127. Cambridge Univ. Press, Cambridge, 1995.

\bibitem{Carlsson-prod}
Gunnar Carlsson.
\newblock On the algebraic {$K$}-theory of infinite product categories.
\newblock {\em $K$-Theory}, 9(4):305--322, 1995.

\bibitem{Carlsson-Goldfarb}
Gunnar Carlsson and Boris Goldfarb.
\newblock The integral {$K$}-theoretic {N}ovikov conjecture for groups with
  finite asymptotic dimension.
\newblock {\em Invent. Math.}, 157(2):405--418, 2004.

\bibitem{Carlsson-Pedersen}
Gunnar Carlsson and Erik~Kj{\ae}r Pedersen.
\newblock Controlled algebra and the {N}ovikov conjectures for {$K$}- and
  {$L$}-theory.
\newblock {\em Topology}, 34(3):731--758, 1995.

\bibitem{Farrell-Jones-isom}
F.~T. Farrell and L.~E. Jones.
\newblock Isomorphism conjectures in algebraic {$K$}-theory.
\newblock {\em J. Amer. Math. Soc.}, 6(2):249--297, 1993.

\bibitem{G-H-W}
Erik Guentner, Nigel Higson, and Shmuel Weinberger.
\newblock The {N}ovikov conjecture for linear groups.
\newblock {\em Publ. Math. Inst. Hautes \'Etudes Sci.}, (101):243--268, 2005.

\bibitem{GTY-rigid}
Erik Guentner, Romain Tessera, and Guoliang Yu.
\newblock A notion of geometric complexity and its application to topological
  rigidity.
\newblock {\em Invent. Math.}, 189(2):315--357, 2012.

\bibitem{GTY-FDC}
Erik Guentner, Romain Tessera, and Guoliang Yu.
\newblock Discrete groups with finite decomposition complexity.
\newblock {\em Groups Geom. Dyn.}, 7(2):377--402, 2013.

\bibitem{Hambleton-Pedersen}
Ian Hambleton and Erik~K. Pedersen.
\newblock Identifying assembly maps in {$K$}- and {$L$}-theory.
\newblock {\em Math. Ann.}, 328(1-2):27--57, 2004.

\bibitem{Hatcher}
Allen Hatcher.
\newblock {\em Algebraic topology}.
\newblock Cambridge University Press, Cambridge, 2002.

\bibitem{Hsiang-geom-applications}
Wu~Chung Hsiang.
\newblock Geometric applications of algebraic {$K$}-theory.
\newblock In {\em Proceedings of the {I}nternational {C}ongress of
  {M}athematicians, {V}ol.\ 1, 2 ({W}arsaw, 1983)}, pages 99--118, Warsaw,
  1984. PWN.

\bibitem{Loday}
Jean-Louis Loday.
\newblock {$K$}-th\'eorie alg\'ebrique et repr\'esentations de groupes.
\newblock {\em Ann. Sci. \'Ecole Norm. Sup. (4)}, 9(3):309--377, 1976.

\bibitem{OOY-quant}
Herv\'e Oyono-Oyono and Guoliang Yu.
\newblock On quantitative operator {$K$}--theory.
\newblock To appear in Ann. Inst. Fourier (Grenoble). Available at
  \url{www.math.univ-metz.fr/~oyono/pub.html}, 2011.

\bibitem{Pedersen-Weibel}
Erik~K. Pedersen and Charles~A. Weibel.
\newblock A nonconnective delooping of algebraic {$K$}-theory.
\newblock In {\em Algebraic and geometric topology (New Brunswick, N.J.,
  1983)}, volume 1126 of {\em Lecture Notes in Math.}, pages 166--181.
  Springer, Berlin, 1985.

\bibitem{Pedersen-Weibel-homology}
Erik~K. Pedersen and Charles~A. Weibel.
\newblock {$K$}-theory homology of spaces.
\newblock In {\em Algebraic topology ({A}rcata, {CA}, 1986)}, volume 1370 of
  {\em Lecture Notes in Math.}, pages 346--361. Springer, Berlin, 1989.

\bibitem{Quillen}
Daniel Quillen.
\newblock Higher algebraic {$K$}-theory. {I}.
\newblock In {\em Algebraic {$K$}-theory, {I}: {H}igher {$K$}-theories ({P}roc.
  {C}onf., {B}attelle {M}emorial {I}nst., {S}eattle, {W}ash., 1972)}, pages
  85--147. Lecture Notes in Math., Vol. 341. Springer, Berlin, 1973.

\bibitem{R-Y-K-theory}
Andrew Ranicki and Masayuki Yamasaki.
\newblock Controlled {$K$}-theory.
\newblock {\em Topology Appl.}, 61(1):1--59, 1995.

\bibitem{R-Y-L-theory}
Andrew Ranicki and Masayuki Yamasaki.
\newblock Controlled {$L$}-theory.
\newblock In {\em Exotic homology manifolds---{O}berwolfach 2003}, volume~9 of
  {\em Geom. Topol. Monogr.}, pages 105--153 (electronic). Geom. Topol. Publ.,
  Coventry, 2006.

\bibitem{Rosenthal}
David Rosenthal.
\newblock Splitting with continuous control in algebraic {$K$}-theory.
\newblock {\em $K$-Theory}, 32(2):139--166, 2004.

\bibitem{STY-coarse-BC}
G.~Skandalis, J.~L. Tu, and G.~Yu.
\newblock The coarse {B}aum-{C}onnes conjecture and groupoids.
\newblock {\em Topology}, 41(4):807--834, 2002.

\bibitem{Sperber}
Ron Sperber.
\newblock Comparing assembly maps in algebraic {$K$}-theory.
\newblock {\em J. K-Theory}, 7(1):145--168, 2011.

\bibitem{Weiss}
Michael Weiss.
\newblock Excision and restriction in controlled {$K$}-theory.
\newblock {\em Forum Math.}, 14(1):85--119, 2002.

\bibitem{Weiss-Williams}
Michael Weiss and Bruce Williams.
\newblock Pro-excisive functors.
\newblock In {\em Novikov conjectures, index theorems and rigidity, {V}ol.\ 2
  ({O}berwolfach, 1993)}, volume 227 of {\em London Math. Soc. Lecture Note
  Ser.}, pages 353--364. Cambridge Univ. Press, Cambridge, 1995.

\bibitem{Winges}
Christoph Winges.
\newblock A note on the {$L$}-theory of infinite product categories.
\newblock {\em Forum Math.}, 25(4):665--676, 2013.

\bibitem{Yu}
Guoliang Yu.
\newblock The {N}ovikov conjecture for groups with finite asymptotic dimension.
\newblock {\em Ann. of Math. (2)}, 147(2):325--355, 1998.

\bibitem{Yu-BC}
Guoliang Yu.
\newblock The coarse {B}aum-{C}onnes conjecture for spaces which admit a
  uniform embedding into {H}ilbert space.
\newblock {\em Invent. Math.}, 139(1):201--240, 2000.

\end{thebibliography}

 \def\cprime{$'$}

\end{document}